\documentclass{article}

\usepackage{fullpage}
\usepackage{amsmath}
\usepackage{amsfonts}
\usepackage{amsthm}
\usepackage{amssymb}
\usepackage{comment}
\usepackage{hyperref}
\usepackage[dvipsnames]{xcolor}
\usepackage{bm}
\usepackage{graphicx}

\newtheorem{theorem}{Theorem}[section]
\newtheorem{definition}[theorem]{Definition}
\newtheorem{remark}[theorem]{Remark}
\newtheorem{proposition}[theorem]{Proposition}

\newtheorem{lemma}[theorem]{Lemma}
\newtheorem{assumption}{Assumption}

\def\balpha{\boldsymbol{\alpha}}
\def\bbeta{\boldsymbol{\beta}}

\def\bxi{\boldsymbol{\xi}}

\newcommand\cB{\mathcal B}

\newcommand\cI{\mathcal I}
\newcommand\cJ{\mathcal J}
\newcommand\cK{\mathcal K}
\newcommand\cL{\mathcal L}

\newcommand\cP{\mathcal P}
\newcommand\cS{\mathcal S}

\def\AA{\mathbb{A}}

\newcommand\EE{\mathbb E}

\newcommand\PP{\mathbb P}

\newcommand{\RR}{\mathbb R}

\newcommand{\Bb}{\mathbf{B}}

\newcommand{\bI}{\mathbf{I}}

\newcommand{\Wb}{\mathbf{W}}
\newcommand{\bW}{\mathbf{W}}

\newcommand{\Zb}{\mathbf{Z}}
\newcommand{\bz}{\mathbf{z}}
\newcommand{\bZ}{\mathbf{Z}}

\newcommand{\zb}{\mathbf{z}}

\title{Stochastic Graphon Games: \\ I. The Static Case}
\author{Ren\'e Carmona \and Daniel B. Cooney \and Christy V. Graves \and Mathieu Lauri\`ere}

\begin{document}

\maketitle

\begin{abstract}
    We consider static finite-player network games and their continuum analogs, graphon games. Existence and uniqueness results are provided, as well as convergence of the finite-player network game optimal strategy profiles to their analogs for the graphon games. We also show that equilibrium strategy profiles of a graphon game provide approximate Nash equilibria for the finite-player games. Connections with mean field games and central planner optimization problems are discussed. Motivating applications are presented and explicit computations of their Nash equilibria and social optimal strategies are provided.
\end{abstract}


\vskip 12pt
\hrule
\vskip 12pt

\section{Introduction}
In this paper, we analyze static network games. They are often brought to bear on the solutions of challenging applications of importance when the interactions between agents are restrained by physical or economic constraints. See, for example Jackson's book \cite{jackson2010social}. For pedagogical reasons, we start with finite-player models as for instance, in the paper of Galeotti, Goyal, Jackson, Vega-Redondo, and Yariv \cite{galeotti2010network}. In these models, the level of interaction between each couple of players is quantified by a connection weight. 

As is often the case in game theory, the complexity of these problems grows rapidly with the number of players. To avoid the curse of dimensionality, we seek simplifications in the asymptotic regime when the number of players tends to infinity. This search for limiting objects leads to the notion of graphon as these structures appear as limits of dense networks. The book of Lov\'{a}sz \cite{lovasz2012large} is a standard reference for the theory of graphons. 
So in the limiting regime, we consider models of graphon games. Some of these models were considered in the static case by Parise and Ozdaglar in \cite{parise2018graphon}, and around the same time by Caines and Huang \cite{caines2018graphon} in a dynamic setting.

The paper \cite{parise2018graphon} was our starting point. We consider more involved structures much in the spirit of Bayesian games, introducing a state and random shocks. To be more specific, in our models, each player also considers their neighbors' states instead of considering only their controls. States could simply be noisy versions of their actions, or implicitly defined functions of their actions and their neighbors' states. Below, we introduce several motivating examples which we describe in very broad brushstrokes. These applications are presented to motivate the need for considering the more general structures which we propose. We will return to these applications at the end of the paper to show how they can be solved explicitly with the tools developed in the paper.

The reader will certainly notice similarities between our limiting procedure and the approaches leading to mean field games, anonymous games, or more general continuum player games ubiquitous in the extensive economics literature on the subject. See for example the paper \cite{carmona2004nash} of G.\ Carmona for an abstract axiomatic of Nash equilibria in games with a continuum of players. However, in a typical non-atomic anonymous game, the players interact through the distribution of actions of all the players whereas in the graphon games we introduce here, interactions between players are not necessarily symmetric. Another important difference with the mean field games paradigm as originally introduced independently by Lasry and Lions \cite{lasry2007mean,lasry2006jeux,lasry2006jeux2} and Huang, Malham\'e, and Caines \cite{huang2006large}, is that our game models are not fully anonymous: each player knows which other players are their neighbors. Because of the potential asymmetry in the connection structure of each player, it is not sufficient \textit{a priori} to consider a global mean field. However, for certain classes of graphons which enjoy specific symmetries, an informative connection can be made between graphon games and mean field games, a point which we will revisit. This has to be contrasted with the work of Delarue \cite{delarue2017mean} 
providing an enlightening insight into how the lack of full connectivity of the graph underpinning the network structure can affect the mean field game paradigm.

\subsection{Motivating Applications}
The purpose of this article is to define rigorously a class of games with an underlying graph structure between players, and to understand the limiting regime when the number of players tends to infinity. We will use the notion of graphon to extend the concept of graph structure between players to the infinite-player regime. To motivate our study of this class of games, we consider the following motivating applications.

\subsubsection{Where do I put my towel on the beach?}
This example was introduced by P.L. Lions in his lectures at the \emph{Coll\`ege de France} \cite{Lions_Lectures}. It is discussed in the Section 1.1.1 of the book by Carmona and Delarue \cite{carmona2018probabilistic}. A group of people are going to the beach and each person has to choose a location to place their towel on the beach. When deciding where to put their towel, a person considers that they do not want to walk too far away from the entrance of the beach, they want to be close to the concession stand, and they want to be close (or far in a different formulation of the game) to the average position of other beach-goers. In the version of this example which we plan to investigate, we would like to allow each individual to only take into account the locations of their friends, not everyone on the beach, departing from the traditional mean field game setting.

\subsubsection{Cities Game}
For our next example, we consider the model studied by Helsley and Zenou in \cite{helsley2014social}. Individuals decide how frequently they would like to visit the city center. Visits to the city center are motivated by an intrinsic benefit for visiting the city center more often, to be balanced with the cost for visiting the city center too frequently, as it takes time and resources to make the trips. However, in the version of the game which we consider, we want to add the benefit of visiting the city center when a specific group of individuals, say friends, are also in the city.

\subsubsection{Cournot Competition}
For our final example, we consider the case of multiple producers having to decide how many units to sell. In our version of this classical example, the price at which the producer can sell their products depends on the number of units they sell, as well as a weighted average of the prices of their competitors. There could be two interpretations for what makes two producers competitors: the level of exchangeability of their products, or how much their consumer bases overlap.

\subsection{Outline of the paper}
We introduce a specific class of  one-period (or static) games including for each player, as is the case of Bayesian game models, a state subject to random shocks. This leads to serious technical difficulties especially when we consider the limiting regime of large games. To resolve them, we use ideas put forth by economists to define rigorously the problem. They are based on the notion of abstract Fubini's extensions \cite{sun2006exact,podczeck2010existence}. See also the discussion in \cite[Section 3.7]{carmona2018probabilistic}. In Section \ref{sec:static_N}, we introduce a network game for $N$ players which we extend in Section \ref{sec:static_continuum} into a continuum version for which the interaction structure is given by a graphon.
Connections with mean field games are elucidated in Propositions \ref{pr:link_constant_MFG}, \ref{pr:link_MFG}, and \ref{pr:link_K_MFG}.
Proposition \ref{pr:graphon_uniqueness} gives existence and uniqueness for the graphon game. 
The connection between the $N$-player game and its continuum analog is considered in Section \ref{sec:link_N_continuum}.
Theorem \ref{thm:convergence} shows that the equilibrium strategy profiles for finite-player network games converge in some sense to their analogs for the limiting graphon game. When the finite-player networks are constructed from the graphon itself, the rate of convergence can be made precise, as in Theorem \ref{thm:convergence_2}. 
Theorem \ref{thm:epsilon_Nash} shows that the equilibrium strategy profiles of a graphon game can be used to construct  $\epsilon$-Nash equilibria for finite-player network games.
Notice that when a central planner can choose the feedback controls of all players in order to optimize the total cost, the problem becomes a graphon control optimization problem. This was the subject of the papers of Gao and Caines \cite{gao2018graphon}\cite{gao2019spectral}\cite{gao2019optimal}. While it is not the main thrust of this paper, we will give a brief discussion of the problem of a central planner who optimizes the overall social cost when we revisit the motivating examples introduced above. These applications are studied in detail in Section \ref{sec:revisiting_examples}. We conclude in Section \ref{sec:conclusion}.

Extensions of the static framework studied in this paper to the dynamic case are under investigation and will be presented in a forthcoming paper.

\section{Finite-Player Model}
\label{sec:static_N}

Let $N$ be a positive integer denoting the number of players, and $W=[W_{i,j}]_{i,j=1,\cdots,N}$  an $N \times N$ symmetric matrix of real numbers. We want to think of the integers in $\{1,\ldots,N\}$ as players and $W_{i,j}$ as a weight quantifying some form of interaction between players $i$ and $j$. When they only take values $0$ or $1$, the matrix $W$ reduces to the adjacency matrix of the graph of players, to be understood as which players interact with each other. We denote by $A$ the set of actions which are admissible to each player. $A$ will typically be a closed convex subset of a Euclidean space $\RR^d$. For the sake of convenience, we consider $A$ to be one dimensional, and to be the whole space $\mathbb{R}$. A strategy profile is an element $\balpha=(\alpha_1,\dots,\alpha_N)$ of $A^N$, $\alpha^i$ being interpreted as the action taken by player $i$. In order to model the random shocks, we shall use   a probability space $(\Omega, \mathcal{F}, \mathbb{P})$ on which are defined random variables $\xi_1,\dots,\xi_N$ assumed to be independent and identically distributed (i.i.d.\ for short), mean-zero with common distribution $\mu_0$ which for the sake of convenience, we often refer to as the distribution of a random variable $\xi_0$. We assume that the function $b:A \times \mathbb{R}\ni (\alpha,z) \mapsto b(\alpha,z)\in \mathbb{R}$ and the distribution $\mu_0$ satisfy the following hypothesis.

\begin{assumption}
\label{hyp:b} 
Suppose the following:
    \begin{itemize}
        \item $b$ is jointly Lipschitz continuous with respect to $\alpha$ and $z$, with respective Lipschitz constants $c_\alpha, c_z \geq 0$, in the sense that: 
        $$
        |b(\alpha,z)-b(\alpha',z')|^2\le c_\alpha|\alpha-\alpha'|^2 + c_z|z-z'|^2,\qquad (\alpha,z),(\alpha',z')\in A\times\mathbb{R}.
        $$
        \item $\EE[\xi_0^2]=\int \xi^2\mu_0(d\xi)<\infty$.
    \end{itemize}
\end{assumption}

\noindent
For an action $\alpha\in A$ and real numbers $z$ and  $\xi$ we define the quantity:
\begin{equation}
\label{fo:X}
X=X_{\alpha,z,\xi}=b(\alpha,z)+\xi
\end{equation}
which should be interpreted as the state of a player when they are affected by the shock $\xi$ and take action $\alpha$ while feeling 
the states of all the players' states through the aggregate $z$. Note that the player does not observe $\xi$ or $z$ when choosing their action $\alpha$. For $z$ to be understood as the aggregate of the states of the other players, it has to depend upon these states, so it is not clear if this intuitive interpretation of an \emph{aggregate influence} is consistent with the mathematical definition \eqref{fo:X}. The following proposition proves this consistency.

\begin{proposition} 
\label{pr:uniqueness_of_z_N}
Under Assumption \ref{hyp:b}, if $\sqrt{c_z} \| W\|_F < 1$, then given any strategy profile $\balpha\in A^N$, there is a unique family of random aggregates $\zb \in L^2(\Omega; \RR^N)$ satisfying:
\begin{equation}\label{eq:z_N_definition}
    \zb = \Bigl(\frac{1}{N}\sum_{j=1}^N W_{i,j}X_{\alpha_j,z_j,\xi_j}\Bigr)_{i=1,\dots,N},
\end{equation}
the equality being understood in the $L^2(\Omega; \RR^N)$ sense.
\end{proposition}
As a result, we can define an operator 
\begin{equation}
\label{fo:Z_N}
\Zb_N:A^N\ni\balpha \to \Zb_N\balpha=\zb\in L^2(\Omega; \RR^N).
\end{equation}

\begin{proof}
Let us fix a strategy profile $\balpha\in A^N$, and for each $\bz\in L^2(\Omega; \RR^N)$ and $i \in\{ 1, \dots, N\}$, let us define the random vector $T\bz$ in $\mathbb{R}^N$ from its $N$ components by:
$$
[T\bz]_i=\frac{1}{N}\sum_{j=1}^N W_{i,j} X_{\alpha_j,z_j,\xi_j}=\frac{1}{N}\sum_{j=1}^N W_{i,j}\bigl[b(\alpha_j,z_j)+\xi_j\bigr],
$$
if we denote by $z_j$ for $j=1,\ldots,N$ the components of $\bz$. We prove that this formula defines a map $T$ from $L^2(\Omega;\RR^N)$ into itself which we prove to be a strict contraction, hence the existence and uniqueness of a fixed point proving the proposition.
Let us first check that $T\bz \in L^2(\Omega; \RR^N)$.
We have for every $i \in \{1, \dots, N\}$,
\begin{align*}
    \EE\Bigl[|[T\bz]_i|^2\Bigr]
    &\leq 
    C \sum_{j=1}^N  W_{i,j}^2 \EE \Bigl[\bigl(b(\alpha_j,z_j)^2 + \xi_j^2\bigr)^2\Bigr]
    \\
    &\leq
    C \sum_{j=1}^N   \EE\left[b(0,0)^2 + c_\alpha |\alpha_j|^2 + c_z|z_j|^2 + \xi_j^2 \right]
    \\
    & < + \infty,
\end{align*}
where we used the Lipschitz continuity of $b$ for the second inequality, and the fact that $\bz \in L^2(\Omega; \RR^N)$ and $\EE[\xi_0^2]<\infty$ for the last inequality. Above and throughout the paper, $C$ is a finite constant whose value may change from one line to the next. We now approach the strict contraction property. For every $\bz_1,\bz_2 \in L^2(\Omega; \RR^N)$,
\begin{equation*}
\begin{split}
 \|T\bz_1-T\bz_2\|^2_{L^2(\Omega;\RR^N)}
 &= \EE  \Bigl[ \sum_{i=1}^N |[T\bz_1]_i-[T\bz_2]_i|^2\Bigr] \\
 &= \EE\sum_{i=1}^N \Bigl|\frac{1}{N}\sum_{j=1}^N W_{i,j}\bigl(X_{\alpha_j,z_{1,j},\xi_j}-X_{\alpha_j,z_{2,j},\xi_j}\bigl) \Bigr|^2 \\
 & \leq \EE\sum_{i=1}^N\Bigl( \frac{1}{N^2}\sum_{j=1}^N W_{i,j}^2\Bigl)\Bigl(\sum_{j=1}^N \Bigl(X_{\alpha_j,z_{1,j},\xi_j}-X_{\alpha_j,z_{2,j},\xi_j}\Bigr)^2\Bigr)  \\
& = \frac{1}{N^2} \sum_{i=1}^N \sum_{j=1}^N W_{i,j}^2 \EE\sum_{j=1}^N \bigl(X_{\alpha_j,z_{1,j},\xi_j}-X_{\alpha_j,z_{2,j},\xi_j}\Bigr)^2 \\
& = \|W \|_F^2 \EE\sum_{j=1}^N \bigl(b(\alpha_j,z_{1,j})-b(\alpha_j,z_{2,j})\bigr)^2 \\
& \leq c_z \|W \|_F^2 \EE\sum_{i=1}^N [z_{1,i}-z_{2,i}]^2=c_z^2 \|W \|_F^2\|\bz_1-\bz_2\|^2
\end{split}
\end{equation*}
which completes the proof since $\sqrt{c_z} \| W\|_F <1$ by assumption.
\end{proof}

If the players use the strategy profile $\balpha\in A^N$, for each $i\in\{1,\cdots,N\}$, player $i$ incurs a cost
\begin{equation}
\label{fo:J_i}
    J_i(\balpha) = \EE \left[ f(X_{\alpha_i, z_i,\xi_i},\alpha_i,  z_i)\right],
    \qquad\text{with}\qquad
    z_i=[\bZ_N\balpha]_i,
\end{equation}
where the function $f:\mathbb{R}\times A \times \mathbb{R} \to \mathbb{R}$ is assumed to be the same for all the players.
It is important to emphasize that the players choose their controls at the beginning of the game, without observing anything. In this sense, the information structure is not in feedback form, as the players are not allowed to observe their state or the states of the other players when choosing their controls, nor are the controls open loop, as they also do not observe their idiosyncratic noise or the idiosyncratic noises of the other players. The information structure is in the style of a Bayesian game, where the players choose their controls only with the above cost in mind, which depends only on the players' actions, and the distribution of the idiosyncratic noise and the random aggregate.

\vskip 6pt
We recall the following standard definition for the sake of definiteness.

\begin{definition}[Nash Equilibrium for $N$-Player Game] 
\label{def:Nash_traditional_N}
A strategy profile $\hat\balpha=(\hat\alpha_1,\cdots,\hat\alpha_N)$ is a Nash equilibrium if for all $i \in \{1,\dots,N\}$ and every $\beta \in A$,
$$
     J_i(\hat\balpha)
     \leq
      J_i([\beta;\hat\balpha_{-i}]),
$$
where $[\beta;\hat\balpha_{-i}]$ is a shorthand for $(\hat\alpha_1, \dots, \hat\alpha_{i-1}, \beta, \hat\alpha_{i+1},\dots,\hat\alpha_N)$.
\end{definition}

The existence of Nash equilibria in pure strategies as defined above is not guaranteed, even under natural assumptions. However, like in the classical case, equilibria in mixed strategies exist as proved by J.\ Nash for standard deterministic static games with compact action sets and continuous cost functions.

In order to prove a similar result in our context, we first revisit the construction of $\Zb_N\balpha$ and prove that it depends smoothly on $\balpha$. Recall that $\Zb_N\balpha$ is the unique fixed point of a strict contraction, and for the purpose of the following discussion we denote this contraction by $T_{\balpha}$. According to its definition, for every $\bz\in L^2(\Omega;\RR^N)$ we have:
$$
(T_{\balpha} \bz)_i=\frac1N\sum_{j=1}^NW_{i,j}\bigl[ b(\alpha_j,z_j) +\xi_j\bigr].
$$

\begin{proposition} 
\label{pr:Lip_of_z_N}
Under Assumption \ref{hyp:b}, and assuming $\sqrt{c_z} \| W\|_F < 1$, we have:
\begin{equation*}
    \|\Zb_N\balpha -\Zb_N\bbeta\|_{L^2(\Omega;\RR^N)} \le \frac{c_\alpha \| W \|_F^2}{1-c_z\| W \|_F^2}\|\balpha - \bbeta\|.
\end{equation*}
\end{proposition}

\begin{proof}
We fix $\bz_0\in L^2(\Omega;\RR^N)$ and for each strategy profile $\balpha \in A^N$ we have:
$$
\Zb_N\balpha=\lim_{k\to\infty}T^k_{\balpha}\zb_0.
$$
Now, we fix  $\balpha$ and $\bbeta$ in $A^N$, and for each integer $k\ge 2$ and $i \in\{ 1, \dots, N\}$, we have:
\begin{equation*}
\begin{split}
 \|T^k_{\balpha}\bz_0-T^k_{\bbeta}\bz_0\|^2_{L^2(\Omega;\RR^N)}
 &= \EE  \Bigl[ \sum_{i=1}^N |[T^k_{\balpha}\bz_0]_i-[T^k_{\bbeta}\bz_0]_i|^2\Bigr] \\
 &= \EE\sum_{i=1}^N \Bigl|\frac{1}{N}\sum_{j=1}^N W_{i,j}\bigl(X_{\alpha_j,[T^{k-1}_{\balpha}\bz_0]_j,\xi_j}-X_{\beta_j,[T^{k-1}_{\bbeta}\bz_0]_j,\xi_j}\bigl) \Bigr|^2 \\
 & \leq \EE\sum_{i=1}^N\Bigl( \frac{1}{N^2}\sum_{j=1}^N W_{i,j}^2\Bigl)\Bigl( \sum_{j=1}^N \Bigl(X_{\alpha_j,[T^{k-1}_{\balpha}\bz_0]_j,\xi_j}-X_{\beta_j,[T^{k-1}_{\bbeta}\bz_0]_j,\xi_j}\Bigr)^2\Bigr)  \\
& = \frac{1}{N^2}\sum_{i=1}^N\sum_{j=1}^N W_{i,j}^2 \EE\sum_{j=1}^N \bigl(X_{\alpha_j,[T^{k-1}_{\balpha}\bz_0]_j,\xi_j}-X_{\beta_j,[T^{k-1}_{\bbeta}\bz_0]_j,\xi_j}\Bigr)^2 \\
& = \|W \|_F^2\EE\sum_{j=1}^N \bigl(b(\alpha_j,[T^{k-1}_{\balpha}\bz_0]_j)-b(\beta_j,[T^{k-1}_{\bbeta}\bz_0]_j)\bigr)^2 \\
& \leq \|W \|_F^2\EE\sum_{j=1}^N\Bigl(c_\alpha|\alpha_j-\beta_j|^2 +  c_z\Bigl| [T^{k-1}_{\balpha}\bz_0]_j-[T^{k-1}_{\bbeta}\bz_0]_j\Bigr|^2\Bigr)\\
&= c_\alpha\|W \|_F^2 \| \balpha-\bbeta\|^2+c_z\|W \|_F^2\| T^{k-1}_{\balpha}\bz_0-T^{k-1}_{\bbeta}\bz_0\|^2\\
&\le c_\alpha  \|W \|_F^2 \| \balpha-\bbeta\|^2+ c_z\|W \|_F^2\bigl(c_\alpha \|W \|_F^2\| \balpha-\bbeta\|^2+c_z\|W \|_F^2\| T^{k-2}_{\balpha}\bz_0-T^{k-2}_{\bbeta}\bz_0\|^2\bigr)\\
&\le c_\alpha\|W \|_F^2(1+c_z\|W \|_F^2) \| \balpha-\bbeta\|^2+ (c_z \|W \|_F^2)^2\| T^{k-2}_{\balpha}\bz_0-T^{k-2}_{\bbeta}\bz_0\|^2 \\
&\le \hskip 45pt \ldots\ldots\\
&\le c_\alpha\|W \|_F^2\left(1+c_z\|W \|_F^2+ \cdots + (c_z\|W \|_F^2)^k \right) \| \balpha-\bbeta\|^2\\
&=c_\alpha \|W \|_F^2\frac{1-(c_z\|W \|_F^2)^{k+1}}{1-c_z\|W \|_F^2}\| \balpha-\bbeta\|^2,
\end{split}
\end{equation*}
where we used the fact that $\sqrt{c_z} \|W \|_F <1$. Taking the limit as $k\to\infty$ completes the proof.
\end{proof}

We are now in a position to prove existence of Nash equilibria in mixed strategies.
\begin{proposition} 
\label{pr:mixed_Nash_for_N}
Under Assumption \ref{hyp:b} and assuming $\sqrt{c_z} \| W\|_F < 1$, if $A$ is compact,
there exist Nash equilibria in mixed strategies if the function $f$ is bounded and continuous. 
\end{proposition}

\begin{proof}
Recall that since the set of admissible actions $A$ is compact, the proof of the existence of Nash equilibria in mixed strategies based on  Kakutani's fixed point theorem only requires that the cost functions are continuous. Note that for $i\in \{ 1,\cdots,N\}$, for any strategy profile $\balpha\in A^N$
we have
\begin{equation}
\label{fo:integral_cost}
J^i(\balpha)=\EE\Bigl[ f(X_{\alpha_i,[\Zb_N\balpha]_i,\xi_i},\alpha_i,\xi_i)\Bigr]=\int\int f\bigl(b(\alpha_i,z)+\xi,\alpha_i,z\bigr)\mu^i_{\balpha}(d\xi,dz),
\end{equation}
where $\mu^i_{\balpha}$ denotes the joint distribution of the random variables $\xi_i$ and $[\Zb_N\balpha]_i$. Since $f$ is assumed to be bounded and continuous, the function $\balpha\to J^i(\balpha)$ is continuous if $\balpha\to\mu^i_{\balpha}$ is continuous, which follows from the result of Proposition \ref{pr:Lip_of_z_N} above.
\end{proof}

\begin{remark}
\label{re:anonymous_1}
An anonymous game is a game where the dependence of the values $J^i(\balpha)$ upon $\balpha$ is only through the empirical measure of the individual actions $\alpha_j$, namely the push forward of the strategy profile by the counting measure on the set of players $\{1,\dots,N\}$.
The game model we are currently studying is in the same spirit since the dependence of the quantities $J^i(\balpha)$ upon $\balpha$ is through the measure $\mu_{\balpha}^i$ which happens to be the joint law of the idiosyncratic noise $\xi_i$ and the nonlinear function $\Zb_N\balpha$ of the strategy profile $\balpha$.
\end{remark}

Formula \eqref{fo:integral_cost} shows that the cost to player $i$ is a function of the action $\alpha_i$ of the player and the joint distribution $\mu_{\balpha}^i$ of the idiosyncratic noise $\xi_i$ and the (random) aggregate $(\Zb_N\balpha)_i$ of all the players states.
It is important to emphasize, even though it may not be immediately clear, that the cost to player $i$ depends upon the states, the actions and the idiosyncratic noises \emph{of all the players} in an implicit way through the aggregate state  $z_i=[\Zb_N\balpha]_i$. 
So this cost can be written as $J^i(\balpha) = \tilde\cJ(\alpha_i,\cL(\xi_i,[\Zb_N\balpha]_i))$ where we use the notation $\cL(\xi,z)$ for the joint distribution of the random variables $\xi$ and $z$. The function $\tilde\cJ$ is defined as:
\begin{equation*}
\tilde\cJ: A \times \cK \ni (\alpha,\nu) \to \tilde\cJ(\alpha,\nu)=\int_{\RR\times \RR}f(X_{\alpha,z,\xi},\alpha,\xi)\;\nu(d\xi,dz),
\end{equation*}
where $\cK$ denotes the subset of $\cP(\RR\times\RR)$ comprising all the probability measures $\nu$ on $\RR\times\RR$ whose first marginal is the common distribution $\mu_0$ of all the $\xi_i$.

The use of the function $\tilde\cJ$ can simplify the search for Nash equilibria in some cases.
\begin{proposition}
\label{pr:Nash_equivalence_N}
Assume the conclusion of Proposition \ref{pr:uniqueness_of_z_N} holds so that $\Zb_N$ is well defined. Suppose $b(\alpha,z)=\tilde{b}(\alpha)$ for some function $\tilde{b}:A\to \RR$, and that $W_{i,i}=0$ for $i=1,\dots,N$. 
Then $\hat\balpha\in A^N$ is a Nash equilibrium if and only if, setting $\mu_i=\cL((\Zb_N\hat\balpha)_i)$ for $i=1,\ldots,N$, we have:
\begin{equation*}
    \tilde\cJ(\hat\alpha_i,\mu_0\otimes\mu_i) \leq \tilde\cJ(\beta,\mu_0\otimes\mu_i), \qquad \forall \beta\in A,\; \forall i=1,\dots,N.
\end{equation*}
\end{proposition}
\begin{proof} 
From the assumption on $b$, $X_{\alpha,z,\xi}=\tilde b(\alpha)+\xi$ does not depend on $z$. Moreover, from equation \eqref{eq:z_N_definition}, and since $W_{i,i}=0$, the quantity
$$
(\Zb_N\balpha)_i=\frac1N\sum_{j\ne i}W_{i,j} X_{\alpha_j,z_j,\xi_j}=\frac1N\sum_{j\ne i}W_{i,j} \bigl(\tilde b(\alpha_j)+\xi_j\bigr)
$$
is independent of $\xi_i$ and consequently $\cL(\xi_i,[\Zb_N\balpha]_i)=\mu_0\otimes\mu_i$
so that $J^i(\balpha) = \tilde\cJ(\alpha_i,\mu_0\otimes\mu_i)$. Moreover
$$
(\Zb_N\balpha)_i=(\Zb_N[\beta;\balpha_{-i}])_i
$$ 
for any $\beta\in A$ and for all $i=1,\dots,N$. We conclude that the two characterizations are equivalent.
\end{proof}

The following assumption on the \emph{cost function} $\tilde\cJ$ will be used later on when we study the convergence of Nash equilibria.
\begin{assumption}
    \label{hyp:tilde_J}
    For each $\alpha\in A$, the function $\cK\ni \nu \mapsto \tilde\cJ(\alpha,\nu)$ is Lipschitz in $\nu$ with a Lipschitz constant $\tilde\ell_J$ uniform in $\alpha\in A$. This means that for every $\nu,\nu' \in \cK$ and $\alpha\in A$, we have
        \begin{equation*}
            \left\vert  \tilde\cJ(\alpha,\nu')- \tilde\cJ(\alpha,\nu) \right\vert \leq \tilde\ell_J W_2(\nu,\nu'),
        \end{equation*}
        where $W_2(\nu,\nu')$ denotes the $2$-Wasserstein distance between the probability measures $\nu$ and $\nu'$.
\end{assumption}

\section{Continuum Player Model}
\label{sec:static_continuum}
Players are labelled by the set $I := [0,1]$, and we denote by $ \mathcal{B}(I)$ its Borel $\sigma$-field and by $\lambda_I$ the Lebesgue measure.
In order to take into account the idiosyncratic randomness that can affect each player in a mathematically rigorously way,
we consider a rich Fubini extension $(I \times \Omega, \mathcal{I} \boxtimes \mathcal{F}, \lambda \boxtimes \PP)$ where $(I,\cI,\lambda)$ is an extension of the Lebesgue measure space $(I,\cB(I),\lambda_I)$. See \cite{sun2006exact}, \cite{podczeck2010existence} or
\cite[Section 3.7]{carmona2018probabilistic} and \cite[Section 4.5]{carmona2019_class_notes} for a self-contained presentation of this theory and references to original contributions on this subject. At this stage, it is important to emphasize that the $\sigma$-field $\cI$ extending the Borel $\sigma$-field $\cB(I)$ cannot be countably generated. 
Let $\bxi=(\xi_x)_{x\in I}$ be a real-valued $\mathcal{I} \boxtimes \mathcal{F}$-measurable essentially pairwise independent process such that the law of $\xi_x $ is the probability measure $\mu_0$ for every $x \in I$. As before, we assume that $\EE[\xi_0]=0$, where $\xi_0$ has distribution $\mu_0$. The random variable $\xi_x$ will play the role of the idiosyncratic random shock directly affecting player $x\in I$. 

\vskip 6pt
To model how the players interact with each other, we use a restricted form of the notion of graphon \cite{lovasz2012large}. For us,  a \textit{graphon} is any symmetric, $\cB(I)\times\cB(I)$ - measurable real-valued square - integrable function $w$ on $I\times I$, in particular:
$$
\|w\|^2_2:=\int_{I\times I} w(x,y)^2\;dxdy\;<\infty.
$$
We introduce the integral operator $\Wb$ on $L^2(I)$ associated to the graphon $w$ by the formula:
$$
[\Wb g]_x=\int_I \;w(x,y) g(y)\;dy,\qquad g\in L^2(I),\;\; x\in I,
$$
and by $\| \Wb \|$ we mean the operator norm, namely:
$$\| \Wb \|:=\sup_{\phi \in L^2(I), \| \phi \|_{L^2(I)}=1} \| \Wb \phi \|_{L^2(I)}.$$

\subsection[Definition of the Graphon Game]{\textbf{Definition of the Graphon Game}}
Let $w$ be a graphon and let us define a game in which the interactions are encoded by $w$.
We assume that the set $A$ of possible actions is the same as in the case of the finite-player game considered earlier. At this stage, we shall assume that a strategy profile is a function $\balpha$ which associates to each player $x\in I$ an action $\balpha(x)=\alpha_x$ in $A$. We shall denote by $L^2(I)$ the classical Lebesgue space $L^2(I,\cB(I),\lambda_I;\RR)$ of equivalent classes of $\RR$-valued square-integrable measurable functions on the Lebesgue measurable space $(I,\cB(I),\lambda_I)$. This is a separable Hilbert space. This space is different from the Hilbert space $L^2(I,\cI,\lambda;\RR)$. The latter is an extension of the classical space $L^2(I,\cB(I),\lambda_I;\RR)$. However, the fact that this space may not be separable will be a source of technicalities we shall need to address or avoid.

\vskip 4pt
We shall call admissible strategy profiles the elements of the subset $\AA$ of $L^2(I)$ of elements taking values in $A\subset\RR$. Note that if $\balpha$ is an admissible strategy profile in this sense, the action $\alpha_x$ of player $x$ is only defined for $\lambda_I$-almost every player $x\in I$. The following result is the analog of Proposition \ref{pr:uniqueness_of_z_N} in the present situation. It uses the same assumptions on the function $b$.

\begin{proposition} 
\label{pr:uniqueness_of_z}
Under Assumption \ref{hyp:b}, and assuming $ \sqrt{c_z} \| \Wb \| < 1$, for any $\balpha\in L^2(I)$ there is a unique $\zb \in L^2(I)$ satisfying 
\begin{equation}
\label{eq:z_definition}
    \zb_x=\int_I w(x,y)b(\alpha_y,z_y)\;d y,\qquad \text{for}\;\lambda_I\text{-almost every }\;x\in I.
\end{equation}
\end{proposition}

\begin{proof}
We fix $\balpha\in L^2(I)$ and we define the mapping $T$ from $L^2(I)$  into itself by:
$$
[T\zb]_x=\int_I w(x,y)b(\alpha_y,z_y)\;d y.
$$
First, we check that $T\zb$ is indeed in $L^2(I)$:
\begin{equation*}
\begin{split}
    \int_I ([T\zb]_x)^2dx&=\int_I \Bigl[\int_I w(x,y)b(\alpha_y,z_y)\;d y\Bigr]^2dx \\
    & = \| \Wb b(\alpha_\cdot,z_\cdot)\|^2 \\
    & \leq \| \Wb\|^2 \| b(\alpha_\cdot,z_\cdot)\|_2^2 \\
    &= \| \Wb\|^2 \int_Ib(\alpha_y,z_y)^2d y \\
    &\leq C  \int_I \left(b(0,0)^2+|\alpha_y|^2+z_y^2\right)d y\\
    & < + \infty.
\end{split}
\end{equation*}
The penultimate inequality uses Assumption~\ref{hyp:b}, the last inequality uses the fact that $\balpha$, $\zb$ are square-integrable. Next we prove that $T$ is a strict contraction.
\begin{equation*}
\begin{split}
    \|T\zb^1-T\zb^2\|^2_{L^2(I)}
    &=
    \int_I \left([T\zb^1]_x-[T\zb^2]_x \right)^2 d x
    \\
    &=\int_I \Bigl(\int_I w(x,y)\left(b(\alpha_y,z^1_y)-b(\alpha_y,z^2_y) \right) d y \Bigr)^2 dx
    \\
    & = \| \Wb \left(b(\alpha_\cdot,z^1_\cdot)-b(\alpha_\cdot,z^2_\cdot) \right) \|^2\\
    & \leq \| \Wb \|^2 \| \left(b(\alpha_\cdot,z^1_\cdot)-b(\alpha_\cdot,z^2_\cdot) \right) \|^2
    \\
    &=\| \Wb \|^2 \int_I \left(b(\alpha_y,z^1_y)-b(\alpha_y,z^2_y) \right)^2 dy 
    \\
    & \leq c_z \| \Wb \|^2 \int_I \left[z^1_y-z^2_y\right]^2dy,
\end{split}
\end{equation*}
where we used Assumption~\ref{hyp:b} again.
Since $ \sqrt{c_z} \| \Wb \| < 1$, $T$ is a strict contraction from $L^2(I)$ into itself, and by Banach fixed point theorem, $T$ has a unique fixed point.
\end{proof}

The result of the previous proposition defines without ambiguity an operator $\Zb$ from $L^2(I)$ into itself by setting $\Zb\balpha=\zb$ where $\zb$ is the unique fixed point of the mapping $T$ identified above. The following proposition uses the exact law of large numbers to show that $\Zb\balpha$ can be viewed as the fixed point of a random transformation.

\begin{proposition}
\label{pr:exact_law_large_numbers}
Under Assumption \ref{hyp:b} and assuming $ \sqrt{c_z} \| \Wb \| < 1$, for any $\balpha\in L^2(I)$, for $\lambda$ - almost every $x\in I$, the random variable 
$$
\Omega\ni\omega\mapsto\int_I w(x,y)X_{\alpha_y,[\Zb\balpha]_y,\xi_y(\omega)}\lambda(dy)
$$ 
is $\PP$-almost surely equal to the deterministic constant $[\Zb\balpha]_x$.
\end{proposition}
\begin{proof}
Again, we fix $\balpha \in L^2(I)$, we consider  $\Zb \balpha$ as constructed above in $L^2(I)$, and we momentarily fix $x\in I$ in a set of full $\lambda_I$-measure on which $\alpha_x$ is unambiguously defined and the equality \eqref{eq:z_definition} holds.
For $y \in I$, let us denote by $Y_{x,y}$ the random variable $w(x,y)X_{\alpha_y,z_y,\xi_y}=w(x,y)[b(\alpha_y,z_y)+ \xi_y]$. Since the random variables $(\xi_y)_{y\in I}$ are assumed to be essentially pairwise independent, the entries of the family $(Y_{x,y})_{y\in I}$ are also essentially pairwise independent. By the exact law of large numbers (see for example \cite[Theorem 3.44]{carmona2018probabilistic}) we have: 
$$
\int_IY_{x,y}(\omega)\lambda(dy)=\int_I \mathbb{E}[Y_{x,y}]\lambda(dy)=\int_I w(x,y) b(\alpha_y,z_y)\lambda(dy)=\int_I w(x,y) b(\alpha_y,z_y)dy, \quad \mathbb{P}\text{-a.e.\ } \omega\in\Omega,
$$
the last equality being justified by the fact that the integrand is $\cB(I)$ measurable. The proof is complete because the above right hand side is equal to $z_x$ by \eqref{eq:z_definition}.
\end{proof}

\begin{remark}
\label{re:z=w}
Note that the result of Proposition \ref{pr:uniqueness_of_z} implies that we have the following relationship between $\Wb$ and $\Zb$:
\begin{equation}
\label{eq:connection_W_A}
    \Zb\balpha= \Wb\left[b(\balpha, \Zb\balpha) \right],\qquad\qquad  \balpha\in L^2(I),
\end{equation}
since the expectation $\EE[\xi_0]$ is independent of $x\in I$ and equal to $0$.
\end{remark}

We now introduce the cost structure of the game. It is given by a function:
$$
J:\;I\times A\times \AA\ni(x,\alpha,\balpha)\mapsto J_x(\alpha,\balpha)\in\RR.
$$
Intuitively, the quantity $J_x(\alpha,\balpha)$ represents the cost to player $x\in I$ when they choose the action $\alpha\in A$ and the other players use the admissible strategy profile $\balpha$. Recall that $\balpha=(\alpha_x)_{x\in I}$ is only defined for $\lambda_I$-almost every $x\in I$ since it is an equivalence class. This is the reason why the cost to a player $x\in I$ cannot be directly defined as a function $J_x(\balpha)$ of an admissible strategy profile, since the latter would be independent upon whatever action $\alpha_x$ player $x$ could take!

\vskip 2pt
Given an admissible strategy profile $\balpha\in\AA$, we define the set of best response strategy profiles as: 
$$
\tilde\Bb\balpha:=\{\balpha' \in \AA : \hbox{for $\lambda_I$-a.e. $x\in I$, $\forall \beta\in A$, } J_x(\alpha'_x, \balpha) \leq J_x(\beta,\balpha) \}.
$$
In other words, an admissible strategy profile $\balpha'\in\AA$ is a best response to the admissible strategy profile $\balpha\in\AA$ if for $\lambda_I$-almost every player $x\in I$, $\alpha'_x$ is a best response for player $x$ to the strategy profile $\balpha$ according to the cost given by the function $J$. 

\vskip 6pt
We will make use of the following assumption.

\begin{assumption}
\label{hyp:J}
For any $\balpha\in\AA$ and $\lambda_I$-a.e.\ $x\in I$, there is a unique minimizer of the map $A\ni\alpha \mapsto J_x(\alpha, \balpha)$.
\end{assumption}

Although this assumption is quite vague, we will eventually consider a more restrictive setting and give specific conditions under which it holds. Under the above assumption, and under a mild measurability and integrability condition on this unique minimizer, the set $\tilde\Bb\balpha$ of best responses to $\balpha$ is a singleton, which defines a map $\AA\ni\balpha\mapsto \tilde\Bb\balpha\in\AA$, and hints at the natural notion of Nash equilibrium:

\begin{definition}[Nash Equilibrium for Continuum Player Graphon Game]
\label{def:Nash_B_continuum}
A Nash equilibrium is an admissible strategy profile $\hat\balpha\in\AA$ which is a fixed point of $\tilde\Bb$  i.e.\ satisfying $\tilde\Bb\hat\balpha=\hat\balpha$,
the equality being understood in the $L^2(I)$ sense.
\end{definition}

\begin{remark}
If $\tilde\Bb$ properly extended to $L^2(I)$ happens to be a strict contraction, the Banach fixed point theorem guarantees existence and uniqueness of a Nash equilibrium.
\end{remark}

As for the $N$-player game, we specify further the form of the costs that we consider. In analogy with the finite-player model, we work with costs defined 
for each player $x\in I$, for each possible action $\alpha\in A$ and for each admissible strategy profile $\balpha\in\AA$ as:
\begin{equation*}
    J_x(\alpha,\balpha) := \EE \bigl[ f(X_{\alpha,z_x,\xi_x},\alpha, z_x)\bigr],
\end{equation*}
where $\zb=(z_x)_{x\in I}$ is a measurable version of the $L^2$-equivalence class of $\Zb\balpha$, and where the function $f$ is the same as before in Section \ref{sec:static_N}. Note that choosing a different version of $\Zb\balpha$ would possibly change the value of
$J_x(\alpha,\balpha)$ on a set of players $x$ of $\lambda_I$-measure $0$, so given $\balpha\in\AA$, the above cost function is only defined for $\lambda_I$-almost every player $x\in I$.

Note that the actions of the other players appear only through $z_x$ which is an aggregate information since it is a version of $\Zb\balpha$. Hence according to Proposition \ref{pr:exact_law_large_numbers}, $[\Zb\balpha]_x$ is deterministic and we can write $J_x(\alpha,\balpha)$ as $\cJ(\alpha,[\Zb\balpha]_x)$ for a function $A\times\RR\ni(\alpha,z)\mapsto \mathcal{J}(\alpha,z)\in\RR$. 
It is useful to notice that this function $\cJ$ is intimately connected to the function $\tilde\cJ$ we introduced in the case of an $N$-player game. Indeed, recall that in the case of $N$ players, the cost to player $i$ was given by $\tilde\cJ\bigl(\alpha_i,(\cL(\xi_i,(\Zb_N\balpha)_i)\bigr)$. In the present situation, for almost every $x\in I$, $(\Zb\balpha)_x$ is purely deterministic, so statistically independent of $\xi_x$ and
$\cL(\xi_x,(\Zb\balpha)_x)=\mu_0\otimes\delta_{(\Zb\balpha)_x}$ where we use the notation $\delta_z$ for the unit mass at the point $z$. 
So if we were to compute $\tilde\cJ\bigl(\alpha_i,(\cL(\xi_i,(\Zb_N\balpha)_i)\bigr)$ in the present situation with $i$ replaced by $x\in I$, we would have $\tilde\cJ\bigl(\alpha_x,\mu_0\otimes\delta_{(\Zb\balpha)_x}\bigr)$, which only depends upon $\alpha_x$ and $(\Zb\balpha)_x$, and which we denote by $\cJ(\alpha_x,(\Zb\balpha)_x)$, essentially identifying the two functions $\tilde\cJ$ and $\cJ$.

\vskip 4pt
We shall formulate the assumptions on the costs incurred by the players in the graphon game in terms of properties of this function $\cJ$. Because of the particular form of the cost, Assumption \ref{hyp:J} is satisfied as soon as the following stronger assumptions hold:
\begin{assumption}
    \label{hyp:J_2}
    Suppose the following:
    \begin{itemize}
        \item For each $z\in\RR$, $A\ni\alpha\mapsto \mathcal{J}(\alpha,z)$ is continuously differentiable and strongly convex in $\alpha$ with a constant $\ell_c>0$ uniformly in $z\in \mathbb{R}$. This means that for every $\alpha,\alpha'\in A$ and $z  \in \mathbb{R}$, we have for every $\epsilon\in [0,1]$:
        \begin{equation}\label{eq:convexity}
            \mathcal{J}(\epsilon \alpha+(1-\epsilon)\alpha',z) \leq \epsilon \mathcal{J}(\alpha,z)+(1-\epsilon)\mathcal{J}(\alpha',z) - \frac{\ell_c}{2}\epsilon(1-\epsilon)|\alpha-\alpha'|^2.
        \end{equation}
        \item For each $\alpha\in A$, the function $\RR\ni z\mapsto \partial_\alpha \mathcal{J}(\alpha,z)$ is Lipschitz in $z$ with a Lipschitz constant $\ell_J$ uniformly in $\alpha\in A$. This means that for every $z',z \in \mathbb{R}$ and $\alpha\in A$, we have:
        \begin{equation*}
            \left\vert \partial_\alpha \mathcal{J}(\alpha,z')-\partial_\alpha \mathcal{J}(\alpha,z) \right\vert \leq \ell_J \left\vert z'-z \right\vert.
        \end{equation*}
    \end{itemize}
\end{assumption}

Under Assumption \ref{hyp:J_2} on the cost function, for any admissible strategy profile $\balpha\in\AA$, it is sufficient to consider, for $\lambda_I$-a.e.\ player $x\in I$, the aggregate $(\widetilde{\Zb}\balpha)_x$, where $\widetilde{\Zb}\balpha\in \mathbb{R}^{I}$ is a representative of  the $L^2$-equivalence class of $\Zb\balpha$ (in other words a $\cB(I)$-measurable version of $\Zb\balpha$), when determining their best response to the control profile $\balpha\in L^2(I)$. We will use this fact to define an equivalent form for the notion of Nash equilibrium which will be  more tractable, as the somewhat strong assumption of having a contraction mapping will not be needed.

Now instead of considering the operator $\tilde\Bb$ giving the best response to an admissible strategy profile $\balpha$, we define the set of best responses to a generic element $\zb\in L^2(I)$, whether or not the latter happens to be the aggregate constructed from an admissible strategy profile  $\balpha$.

$$
\Bb \zb:=\{\balpha' \in L^2(I) : \hbox{for $\lambda_I$-a.e. $x\in I$, $\alpha'_x\in A$ and $\forall \beta\in A$, } \mathcal{J}(\alpha'_x, z_x) \leq \mathcal{J}(\beta, z_x) \}.
$$

Note that under Assumption \ref{hyp:J_2}, if $\Bb\zb$ is not empty, it is necessarily a singleton in $L^2(I)$. 
Clearly, if $\zb\in L^2(I)$, for $\lambda_I$-almost every $x\in I$, $\alpha'_x$ has to be the unique minimizer of the strictly convex function $\alpha\mapsto \cJ(\alpha,z_x)$, but while the function $x\mapsto \alpha'_x$ can be chosen to be $\cB(I)$-measurable, there is no guarantee that it is square-integrable in general. However, this will be the case under Assumption \ref{hyp:J_2} as implied by the proof of Lemma \ref{lemma:B_Lipschitz} below. So this discussion defines
an operator $\Bb:L^2(I) \to L^2(I)$, and we have the following:

\begin{proposition}
\label{pr:Nash_equivalence_continuum}
Under Assumption \ref{hyp:J_2}, for $\hat\balpha\in \mathbb{A}$, the following are equivalent:
\begin{itemize}
    \item $\hat\balpha$ satisfies the conditions of a Nash equilibrium given in Definition \ref{def:Nash_B_continuum};
    \item $\hat \balpha$ is a fixed point of the mapping $\Bb\Zb$ from $L^2(I)$ into itself, i.e.\ $ \hat \balpha= \Bb\Zb\hat \balpha$
as elements of $L^2(I)$;
    \item For $\lambda_I$-a.e.\ $x \in I$ and for any action $\beta \in A$,
    \begin{equation*}
        \mathcal{J}(\hat{\alpha}_x,(\Zb\hat{\balpha})_x)\leq \mathcal{J}(\beta,(\Zb\hat{\balpha})_x).
    \end{equation*}
\end{itemize}
\end{proposition}

We now extend to our game set-up an estimate from Parise and Ozdaglar \cite{parise2018graphon}.

\begin{lemma} 
\label{lemma:B_Lipschitz}
Under Assumption \ref{hyp:J_2}, for any $\zb^1$ and  $\zb^2$ in $L^2(I)$, we have:
\begin{equation*}
    \| \Bb\zb^1-\Bb\zb^2 \|_{L^2(I)} \leq \frac{\ell_J}{\ell_c}\| \zb^1-\zb^2 \|_{L^2(I)}.
\end{equation*}
\end{lemma}
\begin{proof}
Recall that the operator $\Bb$ is not linear so that $\Bb\zb^1-\Bb\zb^2\ne \Bb(\zb^1-\zb^2)$. The strong convexity assumption in Assumption \ref{hyp:J_2} is equivalent to the following: for every $\alpha,\alpha' \in A$ and $z \in \mathbb{R}$ we have:
\begin{equation}\label{eq:strong_convexity_0}
    (\partial_\alpha\mathcal{J}(\alpha,z)-\partial_\alpha\mathcal{J}(\alpha',z)) (\alpha-\alpha') \geq \ell_c \left\vert \alpha-\alpha' \right\vert ^2.
\end{equation}
Since $(\Bb\zb^1)_x= \arg \inf_{\alpha\in A} \mathcal{J}(\alpha,z^1_x)$, by convexity we have:
\begin{equation*}
    \partial_\alpha \mathcal{J}((\Bb\zb^1)_x,z^1_x) \cdot (\alpha-(\Bb\zb^1)_x) \geq 0
\end{equation*}
for all $\alpha\in A$, so using $\alpha=(\Bb\zb^2)_x$ in this inequality, we get:
\begin{equation}\label{eq:u^1}
    \partial_\alpha \mathcal{J}((\Bb\zb^1)_x,z^1_x) \cdot ((\Bb\zb^2)_x-(\Bb\zb^1)_x) \geq 0.
\end{equation}
Similarly, we get:
\begin{equation}\label{eq:u^2}
    \partial_\alpha \mathcal{J}((\Bb\zb^2)_x,z^2_x) \cdot ((\Bb\zb^1)_x-(\Bb\zb^2)_x) \geq 0.
\end{equation}
Adding \eqref{eq:u^1} and \eqref{eq:u^2}, we get:
\begin{equation} \label{eq:u^1_u^2}
\begin{split}
    &\left[\partial_\alpha \mathcal{J}((\Bb\zb^1)_x,z^1_x)-\partial_\alpha \mathcal{J}((\Bb\zb^1)_x,z^2_x) \right] \cdot \left[(\Bb\zb^2)_x-(\Bb\zb^1)_x \right]\\
    &\hskip 75pt
    \geq \left[\partial_\alpha \mathcal{J}((\Bb\zb^2)_x,z^2_x)-\partial_\alpha \mathcal{J}((\Bb\zb^1)_x,z^2_x) \right] \cdot \left[(\Bb\zb^2)_x-(\Bb\zb^1)_x \right] \\
    & \hskip 75pt
    \geq \ell_c \left\vert (\Bb\zb^2)_x-(\Bb\zb^1)_x  \right\vert^2,
\end{split}
\end{equation}
where we used the strong convexity \eqref{eq:strong_convexity_0}. Clearly, the left hand side of \eqref{eq:u^1_u^2} is bounded above by:
\begin{equation*}
    \left\vert \partial_\alpha \mathcal{J}(\Bb\zb^1)_x,z^1_x)-\partial_\alpha \mathcal{J}(\Bb\zb^1)_x,z^2_x) \right\vert \cdot \left\vert(\Bb\zb^2)_x-\Bb\zb^1)_x \right\vert,
\end{equation*}
which together with \eqref{eq:u^1_u^2} gives:
\begin{equation*}
   \left\vert (\Bb\zb^2)_x-(\Bb\zb^1)_x  \right\vert \leq \frac{1}{\ell_c} \left\vert \partial_\alpha \mathcal{J}(\Bb\zb^1)_x,z^1_x)-\partial_\alpha \mathcal{J}(\Bb\zb^1)_x,z^2_x) \right\vert 
  \leq \frac{\ell_J}{\ell_c} \left\vert z^2_x-z^1_x \right\vert,
\end{equation*}
where we used the second part of Assumption \ref{hyp:J_2}. Since this inequality between non-negative real numbers is true for $\lambda_I$-a.e. $x \in I$, we can square both sides, integrate both sides between 0 and 1 and take square roots of
both sides, proving the desired norm estimate.
\end{proof}

\subsection[Connections with Mean Field Games]{\textbf{Connections with Mean Field Games}}
Before developing further our analysis of graphon games, we take a detour to explore how these games are connected with mean field games.
\subsubsection{Constant-Graphon Games as Mean Field Games}

Let $a\in \mathbb{R}$ be a fixed constant, and let us consider the (static) Mean Field Game (MFG) in which 
a representative infinitesimal player incurs a cost 
$$
    \mathcal{J}(\alpha,z) = \EE \left[ f(X_{\alpha, z}, \alpha, z)\right],
$$
for using action $\alpha$ when observing an aggregate $z$ given by $a$ times the mean of the other player's states.
The state of the player is defined by:
\begin{equation*}
    X_{\alpha,z}
    =
    b(\alpha,z)+\xi,
\end{equation*}
for the function $b$ and real-valued random variable $\xi$ as before.

\begin{definition}
$\hat{\alpha}\in A$ is said to be a Nash equilibrium for the above MFG if $\mathcal{J}(\hat{\alpha},\hat{z}) \leq \mathcal{J}(\beta,\hat{z})$ for all $\beta \in A$, where $\hat{z}$ is the unique solution of the fixed point equation $\hat{z}=a\mathbb{E}[X_{\hat{\alpha}, \hat{z}}]$ determining the expected aggregate state in equilibrium.
\end{definition}

\begin{proposition} \label{pr:link_constant_MFG}
    Assume the conclusion of Proposition \ref{pr:uniqueness_of_z} holds so that $\Zb$ is well defined. Let $\hat{\balpha} \in \mathbb{A}$ be a Nash equilibrium for the graphon game with the constant graphon 
$$
w(x,y)=a,\qquad\qquad x,y\in I.
$$
Then $\hat{\balpha}_x$ is constant for $\lambda_I$-a.e. $x \in I$ and its constant value $\hat{\alpha}$ is a Nash equilibrium for the above MFG with the same constant $a$. Conversely, if $\hat{\alpha}$ is a Nash equilibrium for the above MFG with constant $a$, then $\hat{\balpha} \in L^2(I)$ defined by $\hat{\balpha}_x=\hat{\alpha}$ for every $x\in I$, is a Nash equilibrium for the graphon game with the constant graphon $w(\cdot,\cdot)=a$.
\end{proposition}
\begin{proof}
    Let $\hat{\balpha}$ be a Nash equilibrium for the graphon game.
    One can easily check that a deterministic constant, say $\hat{z}$, satisfies the fixed point construction of $\Zb\hat\balpha$.
    So for $\lambda_I$-a.e.\ $y \in I$, we have $\mathcal{J}(\hat{\balpha}_y,\hat{z}) \leq \mathcal{J}(\beta,\hat{z})$ for all $\beta \in A$. By Assumption \ref{hyp:J_2}, there is a unique minimizer of $\mathcal{J}(\beta,\hat{z})$ in the $\beta$ argument, and thus, $\hat{\balpha}$ is constant for $\lambda_I$-a.e.\ $y$. Clearly the constant value $\hat{\balpha}$ is a Nash equilibrium for the above MFG.

    Now, let $\hat{\balpha}$ be a Nash equilibrium for the above MFG, and let $\hat{\balpha}(\cdot)=\hat{\alpha}$, then
let $\hat{z}$ be the unique solution to $\hat{z}=ab(\hat{\alpha}, \hat{z})$. Then
    \begin{equation*}
        \int_Iw(x,y)b(\hat{\balpha}(y), \hat{z})dy =ab(\hat{\alpha}, \hat{z})= \hat{z},
    \end{equation*}
    which is the definition of $[\Zb\hat{\balpha}]_x$. Thus, $(\Zb\hat{\balpha})_x=\hat{z}$ and $\mathcal{J}(\hat\alpha_x,(\Zb\hat{\balpha})_x) \leq \mathcal{J}(\beta,(\Zb\hat{\balpha})_x)$ for all $\beta \in A$ and for all $x\in I$, which means that $\hat{\balpha}$ satisfies the characterization in Proposition \ref{pr:Nash_equivalence_continuum}.
\end{proof}

\subsubsection{Constant Connection Strength Graphon Games as MFGs} \label{sec:constantconnection}
We now see that that the above MFG is closely related to a graphon game for a class of graphons which we call constant connection strength graphons.

\begin{definition}
We say a graphon $w$ is a constant connection strength graphon with strength $a$ for a given constant $a\in \mathbb{R}$ if $\int_I w(x,y)dy=a$ for all $x \in I$.
\end{definition}
\begin{remark}
The following are examples of constant connection strength graphons with strength $a$:
\begin{itemize}\itemsep=-2pt
    \item Constant graphon: $w(x,y)\equiv a$.
    \item A graphon of the form $w(x, y) = \tilde{w}\bigl(d(x, y)\bigr)$, for some $\tilde{w}: \mathbb{R} \to  \mathbb{R}$, where $d(x, y)$ is the distance on the unit circle (i.e.\ $[0,1)$ with periodic boundary) as long as $\int_{0,1]} \tilde w(z)dz=a$. In particular, this includes the Watts-Strogatz graphon.
    \item A piecewise constant graphon of the form $w(x,y)=a_1 1_{\{x \in [0,x^*),y \in [0,x^*)\}}+a_2 1_{\{x \in [0,x^*),y \in [x^*,1)\}}+a_2 1_{\{x \in [x^*,1),y \in [0,x^*)\}}+a_3 1_{\{x \in [x^*,1),y \in [x^*,1)\}}$ for some $x^* \in I, a_1, a_2 \in \mathbb{R}$ satisfying $a_1 x^* + a_2 (1-x^*) = a_2 x^* + a_3 (1-x^*)$.
\end{itemize}
\end{remark}

\begin{proposition} 
\label{pr:link_MFG}
    Assume the conclusion of Proposition \ref{pr:uniqueness_of_z} holds so that $\Zb$ is well defined. Suppose there exists a Nash equilibrium $\hat{\alpha}$ for the above MFG with constant $a$ introduced in the previous subsection. Then if $w$ is a constant connection strength graphon with strength $a$,  $\hat{\balpha}$ defined by $\hat{\balpha}(\cdot)=\hat{\alpha}$ is a Nash equilibrium for the graphon game.
\end{proposition}
\begin{proof}
    Let $\hat{z}$ be a solution to $\hat{z}=a\mathbb{E}[X_{\hat{\alpha}, \hat{z}}]$. Then
    \begin{equation*}
        \int_Iw(x,y)b(\hat{\balpha}(y), \hat{z})dy =b(\hat{\alpha}, \hat{z})\int_Iw(x,y)dy=ab(\hat{\alpha}, \hat{z})= \hat{z},
    \end{equation*}
    which is the definition of $[\Zb\hat{\balpha}]_x$. Thus, $(\Zb\hat{\balpha})\equiv\hat{z}$ and $\mathcal{J}(\hat{\alpha}_x,(\Zb\hat{\balpha})_x) \leq \mathcal{J}(\beta,(\Zb\hat{\balpha})_x))$ for all $\beta \in A$ and for all $x$, which implies that $\hat{\balpha}$ satisfies the characterization in Proposition \ref{pr:Nash_equivalence_continuum}.
\end{proof}

\subsubsection{Piecewise Constant Graphon Games as Multi-Population MFGs}

We now define, in our static framework, a type of multi-population MFG. Fix an integer $K$ and suppose there are $K$ mean field communities, each of equal relative size. A strategy profile is a collection of actions $\bm{\alpha} = (\alpha_k)_{k \in \{1,\dots,K\}} \in A^K$. Let us consider a symmetric matrix (of interaction strengths) $W \in \mathbb{R}^{K \times K}$. For a strategy profile $\bm{\alpha} = (\alpha_k)_{k \in \{1,\dots,K\}} \in A^K$ and a family of aggregates $\zb = (z_1)_{k \in \{1,\dots,K\}} \in \mathbb{R}^K$, a representative infinitesimal player from the $k$-th population has state given by:
\begin{equation}
\label{eq:static-K-pop-mfg-X}
    X_{k,\alpha_k,z_k}(\omega)=b(\alpha_k,z_k)+\xi_{k}(\omega),
\end{equation}
with $(\xi_k)_{k \in \{1,\dots,K\}}$ a sequence of independent identically distributed random variables with distribution $\mu_0$, as before. Generic players in the $k$-th community incur the cost:
\begin{equation}
\label{eq:static-K-pop-mfg-J}
    \mathcal{J}(\alpha_k,z_k) = \mathbb{E} \left[f(X_{k,\alpha_k,z_k},\alpha_k,z_k) \right],
\end{equation}
where, for each $k \in \{1,\dots,K\}$, $\zb$ solves:
$$
z_k:=\frac{1}{K}\sum_{k'=1}^K W_{k,k'} \mathbb{E}[X_{k',\alpha_{k'},z_{k'}}].
$$
Thus, the player's cost is impacted by a weighted average of each population's average state. The coefficients of this weighted average, given by $W$, depend on the population $k$ under consideration.

\begin{definition}
A strategy profile $\bm{\alpha} \in A^K$ is a Nash equilibrium for the $K$-population MFG if for all $k$ and all $\beta \in A$,
    $$
    \mathcal{J}(\alpha_k,z_k)
    \leq
    \mathcal{J}(\beta,z_k),
    $$
    and $\zb$ uniquely solves $z_k:=\frac{1}{K}\sum_{k'=1}^K W_{k,k'} \mathbb{E}[X_{k',\alpha_{k'},z_{k'}}].$
\end{definition}
A Nash equilibrium for the $K$-population MFG provides an approximate Nash equilibrium for a game with a finite but large number of players in each population. See e.g.\ Chapter 8 of the book by Bensoussan, Frehse, and Yam \cite{bensoussan2013mean} for a formal proof.

For convenience we let for each integer $K\ge 1$, $\psi^K:A^{K}\to L^2(I)$ denote the map that embeds a profile of $K$ strategies $\balpha$ into the evenly spaced step function $\psi^K(\balpha)$ defined by $\psi^K(\balpha)_x=\alpha_i$ whenever $x \in [(i-1)/K,i/K)$. For the sake of definiteness we could add $\psi^K(\balpha)_1=\alpha_K$, but this will not matter. Similarly, we let $\Psi^K$ denote the map that takes a  $K\times K$ symmetric matrix $W^K=[W^K_{i,j}]_{i,j=1,\ldots,K}$ into the graphon $w^K=\Psi^K(W^K)$ defined by $w^K(x,y)=W^K_{i,j}$ whenever $x \in [(i-1)/K,i/K)$ and $y \in [(j-1)/K,j/K)$. Again, we could choose a specific convention to extend naturally $w$ into a function defined everywhere on the unit square $I\times I$. It will also be useful to define for each integer $K\ge 1$, the map $\mu^K:L^2(I)\to \mathbb{R}^K$ by $[\mu^K\zb]_i:=K\int_{(i-1)/K}^{i/K}z_x dx$, which gives the average of $\zb$ on the interval $[(i-1)/K,i/K)$. 

\begin{proposition} \label{pr:link_K_MFG}
    Let $K\ge 1$, $W\in \mathbb{R}^{K\times K}$ be a symmetric matrix, and assume the conclusion of Proposition \ref{pr:uniqueness_of_z} holds so that $\Zb$ defined with respect to the piecewise constant graphon $w:=\Psi^K(W)$ is well defined. Let $\hat{\balpha}$ be a Nash equilibrium for the graphon game with $w$. Then $\hat{\balpha}$ is an evenly spaced step function for $\lambda_I$-a.e. $x \in I$ and let $\hat{\bm{\alpha}}^K:=\mu^K( \hat{\balpha})$ be the constant values. Then $\hat{\bm{\alpha}}^K$ is a Nash equilibrium for the $K$-population MFG defined by~\eqref{eq:static-K-pop-mfg-X}--\eqref{eq:static-K-pop-mfg-J} with interaction strength matrix $W$. Conversely, let $\hat{\bm{\alpha}}^K$ be a Nash equilibrium for the $K$-population MFG defined by~\eqref{eq:static-K-pop-mfg-X}--\eqref{eq:static-K-pop-mfg-J} with interaction strength matrix $W$. Then $\hat{\balpha}:=\psi^K(\hat{\bm{\alpha}}^K)$ is a Nash equilibrium for the graphon game with the piecewise constant graphon $w:=\Psi^K(W)$.
 \end{proposition}
 \begin{proof}
    Let $\hat{\balpha}$ be a Nash equilibrium for the graphon game. One can check that there exists $\zb \in \mathbb{R}^K$ such that $\psi^K(\zb)$ satisfies the unique definition of $\Zb\hat{\balpha}$, where $\Zb$ is defined with respect to the graphon $\Psi^K(W)$. Thus, for $\lambda_I$-a.e.\ $y \in [(k-1)/K,k/K)$, we have $\mathcal{J}(\hat{\alpha}_y,z_k) \leq \mathcal{J}(\beta,z_k)$ for all $\beta \in A$. By Assumption \ref{hyp:J_2}, there is a unique minimizer of $\mathcal{J}(\alpha,z_k)$ in the $\alpha$ argument, and thus, $\hat{\balpha}$ is constant for $\lambda_I$-a.e.\ $y \in [(k-1)/K,k/K)$. Clearly the constant values, given by $\mu^K( \hat{\balpha})$, is a Nash equilibrium for the above $K$-population MFG.
 
    Now let $\hat{\bm{\alpha}}^K$ be a Nash equilibrium for the $K$-population MFG.
    For all $k$ and for all $\beta \in A$, we have:
    $$
        \mathcal{J}(\hat{\alpha}^K_k, z_k) \leq \mathcal{J}(\beta, z_k),
    $$
    and $\zb$ uniquely solves $z_k:=\frac{1}{K}\sum_{k'=1}^K W_{k,k'} \mathbb{E}[X_{k',\hat{\alpha}^K_{k'},z_{k'}}].$
    Clearly, we have $\Zb\hat{\balpha}=\psi^K(\zb)$, where $\Zb$ is defined with respect to the graphon $\Psi^K(W)$. Thus, for all $x$ and all $\beta \in A$:
    $$
        \mathcal{J}(\hat{\alpha}_x, [\Zb\hat{\balpha}]_x) \leq \mathcal{J}(\beta, [\Zb\hat{\balpha}]_x), 
    $$
    which means $\hat{\balpha}$ is a Nash equilibrium for the graphon game with graphon $\Psi^K(W)$.
 \end{proof}

\subsection[Existence of Nash Equilibria]{\textbf{Existence of Nash Equilibria}}

Now we return to our analysis of graphon games. In the following, we make use of an additional assumption:
\begin{assumption}
\label{hyp:2} 
There exists a finite positive constant $c_0$ such that $|b(\alpha,z)|\leq c_0$ for all $\alpha\in A$ and $z\in \mathbb{R}$.
\end{assumption}

\begin{theorem}\label{thm:existence}
Under Assumptions \ref{hyp:b}, \ref{hyp:J_2}, and \ref{hyp:2}, and assuming $ \sqrt{c_z} \| \Wb \| < 1$, there is at least one Nash equilibrium.
\end{theorem}
\begin{proof}
First, let $\chi:=c_0$.
Remark \ref{re:z=w} implies that for any $\balpha\in L^2(I)$,  $\Zb\balpha$ belongs to the range $\Wb(B_\chi)$ where $B_\chi$ denotes the closed ball $B_\chi=\{\zb \in L^2(I); \| \zb \| \leq \chi\}$. Since the graphon operator $\Wb$ is a Hilbert-Schmidt operator, it is \textit{a fortiori} a compact operator, implying that $\Wb(B_\chi)$  is a relatively compact subset of $L^2(I)$. Being Lipschitz (recall Lemma \ref{lemma:B_Lipschitz}), $\Bb$ is also continuous and the image by $\Bb$ of the closure $\overline{\Wb(B_\chi)}$, call it momentarily $K$, is a compact subset of $L^2(I)$. So the (nonlinear) operator $\Bb\Zb$  from $L^2(I)$ into itself is continuous and its range is contained in the compact set $K$. Schauder's theorem implies that this operator has a fixed point, and according to Proposition \ref{pr:Nash_equivalence_continuum}, the latter is a Nash equilibrium for the game.
\end{proof}

Schauder's theorem allows us to prove existence of Nash equilibria under rather mild assumptions, but unfortunately, it does not guarantee uniqueness. The latter is obtained under stronger assumptions. For example, it holds when the fixed point is derived from a strict contraction. We give a sufficient condition for this to be the case in the next subsection. There, we also give a uniqueness condition inspired by the Lasry-Lions monotonicity condition prevalent in the theory of mean field games.

\subsection{Uniqueness of Nash Equilibria}

\begin{lemma}
\label{lemma:W_Lip}
Under Assumption \ref{hyp:b}, and assuming $ \sqrt{c_z} \| \Wb \| < 1$, for any $\balpha^1$ and $\balpha^2$ in $L^2(I)$, we have the $L^2$ bound:
\begin{equation}
\label{fo:l2}
    \|\Zb\balpha^1-\Zb\balpha^2\|_{L^2(I)} \leq \frac{\sqrt{c_\alpha} \|\bW\|}{1-\sqrt{c_z} \|\bW\|}\|\balpha^1-\balpha^2 \|_{L^2(I)},
\end{equation}
as well as the pointwise bound:
\begin{equation}
\label{fo:pointwise}
    \left\vert [\Zb\balpha^1]_x-[\Zb\balpha^2]_x \right\vert \leq \frac{\sqrt{c_\alpha}}{1-\sqrt{c_z} \|\bW\|}\Bigl( \int_I w(x,y)^2dy\Bigr)^{1/2} \left\Vert \balpha^1-\balpha^2 \right\Vert_{L^2(I)},\qquad \lambda_I\text{-a.e.}\;\; x\in I.
\end{equation}
\end{lemma}

\begin{proof}
We first prove \eqref{fo:l2}.
From equation \eqref{eq:connection_W_A} we have:
\begin{equation*}
\begin{split}
    \|\Zb\balpha^1-\Zb\balpha^2\|_{L^2(I)}
    &=\|\Wb\left[b(\balpha^1, \Zb\balpha^1)-b(\balpha^2, \Zb\balpha^2)\right]\|_{L^2(I)}\\
    &\leq \|\Wb\| \; \|b(\balpha^1, \Zb\balpha^1)-b(\balpha^2, \Zb\balpha^2)\|_{L^2(I)}\\
    & \leq  \|\Wb\| \; \left( \sqrt{c_{\alpha}}\|\balpha^1-\balpha^2 \|_{L^2(I)} +\sqrt{c_z} \|\Zb\balpha^1-\Zb\balpha^2\|_{L^2(I)} \right). 
\end{split}
\end{equation*}
We conclude because $\sqrt{c_z}\|\bW\|<1$ by assumption.
We now prove \eqref{fo:pointwise}. For $\lambda_I$-a.e.\ $x\in I$, we have:
\begin{equation*}
    \begin{split}
        \left\vert [\Zb\balpha^1]_x-[\Zb\balpha^2]_x \right\vert 
        &=\left\vert \int_I w(x,y)\left[b(\alpha^1_y,[\Zb\balpha^1]_y)-b(\alpha^2_y,[\Zb\balpha^2]_y) \right] dy \right\vert \\
        &\leq \Bigl( \int_I w(x,y)^2dy\Bigr)^{1/2}\Bigl(\int_I\Bigl|b(\alpha^1_y,[\Zb\balpha^1]_y)-b(\alpha^2_y,[\Zb\balpha^2]_y)\Bigr|^2 dy \Bigr)^{1/2} \\
        & \leq \Bigl( \int_I w(x,y)^2dy\Bigr)^{1/2}\Bigl(\sqrt{c_\alpha} \left\Vert \balpha^1-\balpha^2 \right\Vert_{L^2(I)}+\sqrt{c_z} \left\Vert \Zb\balpha^1-\Zb\balpha^2\right\Vert_{L^2(I)}\Bigr)\\
        & \leq \frac{\sqrt{c_\alpha}}{1-\sqrt{c_z} \|\bW\|} \Bigl( \int_I w(x,y)^2dy\Bigr)^{1/2}\left\Vert \balpha^1-\balpha^2 \right\Vert_{L^2(I)},
    \end{split}
\end{equation*}
where we used the $L^2$ bound \eqref{fo:l2}.
\end{proof}

\begin{proposition} 
\label{pr:graphon_uniqueness}
Under Assumption \ref{hyp:b} and Assumption \ref{hyp:J_2}, and assuming $ \sqrt{c_z} \| \Wb \| < 1$, if:
\begin{equation}
\label{fo:existence+uniqueness}
    \frac{\ell_J}{\ell_c}\cdot \frac{\sqrt{c_\alpha} \|\bW\|}{1-\sqrt{c_z} \|\bW\|}<1,
\end{equation}
then there exists a unique Nash equilibrium.
\end{proposition}
\begin{proof}
With Lemmas \ref{lemma:B_Lipschitz} and \ref{lemma:W_Lip}, we have:
\begin{equation*}
\begin{split}
    \| \Bb \Zb \balpha^1-\Bb \Zb \balpha^2 \|_{L^2(I)} &\leq \frac{\ell_J}{\ell_c} \|\Zb\balpha^1-\Zb\balpha^2\|_{L^2(I)}\\
    & \leq \frac{\ell_J}{\ell_c} \cdot \frac{\sqrt{c_\alpha}\|\bW\|}{1-\sqrt{c_z}\|\bW\|}\|\balpha^1-\balpha^2 \|_{L^2(I)}.
\end{split}
\end{equation*}
Thus, $\Bb \Zb$ is a strict contraction, and by Banach's fixed point theorem, it has a unique fixed point.
\end{proof}

We now introduce a notion of monotonicity, inspired by the Lasry-Lions monotonicity condition in the mean field games literature. For a given graphon $w$, we say a function $h(\alpha,z)$ is \textit{monotone} if for all $\balpha^1,\balpha^2 \in \mathbb{A}$, we have
\begin{equation}
\label{eq:monotonicity}
    h(\alpha^1_x,[\Zb \balpha^1]_x) -h(\alpha^1_x,[\Zb \balpha^2]_x) -h(\alpha^2_x,[\Zb \balpha^1]_x) +h(\alpha^2_x,[\Zb \balpha^2]_x) \geq 0
\end{equation}
holds on a set of positive $\lambda_I$-measure.

\begin{proposition}
\label{pr:monotonicity}
Assume the conclusion of Proposition \ref{pr:uniqueness_of_z} holds so that $\Zb$ is well defined. Also, suppose that for each $z\in \mathbb{R}$, $\cJ(\cdot,z)$ has a unique minimizer, and that $\cJ(\alpha,z)$ is monotone in the above sense. Then there is at most one Nash equilibrium.
\end{proposition}
\begin{proof}
Suppose $\balpha^1$ and $\balpha^2$ are two Nash equilibria for the sake of contradiction.
Then for $\lambda_I$-a.e.\ $x\in I$:
$$\cJ(\alpha^1_x,[\Zb \balpha^1]_x)-\cJ(\alpha^2_x,[\Zb \balpha^1]_x)+\cJ(\alpha^2_x,[\Zb \balpha^2]_x)-\cJ(\alpha^1_x,[\Zb \balpha^2]_x) < 0,$$
by minimality and uniqueness of the minimizers. However, by monotonicity, the above quantity is nonnegative on a set of positive measure, which is a contradiction.
\end{proof}

As an example, consider the constant graphon $w(\cdot,\cdot)=1$, and 
$$
b_0(\alpha,z)=\alpha, \quad \text{and} \quad f(x,\alpha,z)= \frac{3}{2} \alpha^2 -(x-z)^2.
$$
We have $\cJ(\alpha,z) = \mathbb{E} \left[ \frac{3}{2}\alpha^2 -(\alpha+\xi-z)^2 \right]= \frac{1}{2}\alpha^2 + 2\alpha z - z^2 -\mathbb{E}(\xi^2)$. Then for any $\balpha$, we have $[\Zb \balpha]_x=\int_I \alpha_x dx =: \bar{\balpha}$ for all $x \in I$. For any two control profiles $\balpha_1$ and $\balpha_2$ which are not equal in $L^2(I)$, we must have $\bar{\balpha}_1 \neq \bar{\balpha}_2$, as otherwise, the uniqueness of the minimizer would give $\balpha_1 = \balpha_2$. Then we have:
\begin{equation*}
\begin{split}
    &\int_I \Big[\frac{1}{2}\alpha_1^2 + 2\alpha_1 \bar{\alpha}_1 - \bar{\alpha}_1^2 -\mathbb{E}(\xi^2) -\frac{1}{2}\alpha_1^2 - 2\alpha_1 \bar{\alpha}_2 + \bar{\alpha}_2^2 +\mathbb{E}(\xi^2)\\
    &-\frac{1}{2}\alpha_2^2 - 2\alpha_2 \bar{\alpha}_1 + \bar{\alpha}_1^2 +\mathbb{E}(\xi^2) +\frac{1}{2}\alpha_2^2+2\alpha_2 \bar{\alpha}_2 - \bar{\alpha}_2^2 -\mathbb{E}(\xi^2)\Big] dx \\
     &= 2(\bar\alpha_1-\bar\alpha_2)\int_I (\alpha_1-\alpha_2)  dx\\
     &= 2(\bar\alpha_1-\bar\alpha_2)^2> 0,
\end{split}
\end{equation*}
so \eqref{eq:monotonicity} holds on a set of positive $\lambda_I$-measure, and thus $\cJ(\alpha,z)$ is monotone. Clearly, we also have a unique minimizer of $\cJ(\cdot,z)$ for each $z$. Thus, there is at most one Nash equilibrium. Notice that for this example, we have $c_\alpha=1$, $c_z=0$, $\| \Wb \|=1$, $\ell_c=1$, $\ell_J=2$, and thus, \eqref{fo:existence+uniqueness} is not satisfied. Conversely, consider the same example as above except with a different sign in $f$: $f(x,\alpha,z)= \frac{3}{2} \alpha^2 +(x-z)^2$. Now we have $\ell_c=5$ and $\ell_J=2$, and thus \eqref{fo:existence+uniqueness} is satisfied, even though $f$ is not monotone. Thus, Proposition \ref{pr:graphon_uniqueness} and Proposition \ref{pr:monotonicity} provide two different paths to uniqueness.

\subsection{Stability Estimate} 
\label{sub:stability}
The following stability result will be useful in the sequel.

\begin{lemma}\label{lemma:W_W_tilde}
Under Assumptions \ref{hyp:b}, \ref{hyp:J_2} and \ref{hyp:2}, for any two graphons $w$ and $w'$ with corresponding graphon operators $\Wb$ and $\Wb'$ satisfying $ \sqrt{c_z} \| \Wb \| < 1$ and $ \sqrt{c_z} \| \Wb' \| < 1$, and operators $\Zb$ and $\Zb'$  as defined in equation \eqref{eq:z_definition} respectively, for any $\balpha\in L^2(I)$ we have:
\begin{equation*}
    \|\Zb\balpha- \Zb'\balpha\|_{L^2(I)} \leq \frac{c_0}{1-\sqrt{c_z}\|\bW\|}\|\Wb-\Wb' \|.
\end{equation*}
\end{lemma}
\begin{proof}
Starting from equation \eqref{eq:connection_W_A}, we get:
\begin{equation*}
    \begin{split}
        \|\Zb\balpha- \Zb'\balpha\|_{L^2(I)}&=\left\Vert\Wb[b(\balpha, \Zb\balpha) ]-\Wb'[b(\balpha, \Zb'\balpha)]\right\Vert_{L^2(I)}\\
        & \leq\left\Vert\Wb\left[b(\balpha, \Zb\balpha)\right]-\Wb\left[b(\balpha, \Zb'\balpha)\right]\right\Vert_{L^2(I)}\\
        &\hskip 75pt
        +\left\Vert\Wb\left[b(\balpha, \Zb'\balpha)\right]-\Wb'\left[b(\balpha, \Zb'\balpha)\right]\right\Vert_{L^2(I)}\\
        & \leq \|\Wb\| \|b(\balpha, \Zb\balpha)-b(\balpha, \Zb'\balpha) \|_{L^2(I)} \\
        &\hskip 75pt 
        + \|\Wb-\Wb' \|  \|b(\balpha,\Zb'\balpha) \|_{L^2(I)}\\
        & \leq \|\bW\|\cdot \sqrt{c_z}\cdot \| \Zb\balpha-\Zb'\balpha \|_{L^2(I)}+c_0 \|\Wb-\Wb' \|,
    \end{split}
\end{equation*}
and the conclusion follows.
\end{proof}

\begin{theorem} \label{thm:perturbation}
Assume Assumptions \ref{hyp:b}, \ref{hyp:J_2} and \ref{hyp:2}, and $ \sqrt{c_z} \| \Wb \| < 1$ and \eqref{fo:existence+uniqueness} holds for the graphon $w$ with corresponding graphon operator $\Wb$. We denote its unique Nash equilibrium by $\balpha$. Then for any graphon $w'$ with corresponding graphon operator $\Wb'$ satisfying $ \sqrt{c_z} \| \Wb' \| < 1$ and any Nash equilibrium $\balpha'$ for the graphon game associated to $w'$, we have:
\begin{equation*}
    \|\balpha-\balpha'\|_{L^2(I)} \leq \kappa||\Wb-\Wb'||,
\end{equation*}
where
$$ 
\kappa:=\frac{c_0 \ell_J (1-\sqrt{c_z} \|\bW\|)}{(1-\sqrt{c_z}\|\bW\|)\left[\ell_c(1-\sqrt{c_z} \|\bW\|)-\ell_J\sqrt{c_\alpha}\|\bW\|\right]}.$$
\end{theorem}
\begin{proof}
By assumption, we have $\balpha=\Bb\Zb\balpha$ and $\balpha'=\Bb\Zb'\balpha'$. Thus,
\begin{equation*}
    \begin{split}
        \|\balpha-\balpha'\|_{L^2(I)}&=\|\Bb\Zb\balpha-\Bb\Zb'\balpha'\|_{L^2(I)}\\
        & \leq \frac{\ell_J}{\ell_c}\|\Zb\balpha- \Zb'\balpha'\|_{L^2(I)}\\
        & \leq \frac{\ell_J}{\ell_c}\|\Zb\balpha- \Zb\balpha'\|_{L^2(I)}+\frac{\ell_J}{\ell_c}\|\Zb\balpha'- \Zb'\balpha'\|_{L^2(I)}\\
        & \leq \frac{\ell_J}{\ell_c} \frac{\sqrt{c_\alpha}\|\bW\|}{1-\sqrt{c_z} \|\bW\|}  \|\balpha- \balpha'\|_{L^2(I)}+\frac{\ell_J}{\ell_c}\|\Zb\balpha'- \Zb'\balpha'\|_{L^2(I)},
    \end{split}
\end{equation*}
where we used the Lipschitz property of $\Bb$ from Lemma \ref{lemma:B_Lipschitz} and the result of Lemma \ref{lemma:W_Lip}. Because \eqref{fo:existence+uniqueness} holds, we have:
\begin{equation*}
    \|\balpha-\balpha'\|_{L^2(I)} \leq \frac{\ell_J(1-\sqrt{c_z} \|\bW\|)}{\ell_c(1-\sqrt{c_z} \|\bW\|)-\ell_J\sqrt{c_\alpha}\|\bW\|} \|\Zb\balpha'- \Zb'\balpha'\|_{L^2(I)},
\end{equation*}
and we conclude using Lemma \ref{lemma:W_W_tilde}.
\end{proof}

\subsection{Central Planner} 
\label{sub:graphon_control}
Even though we shall address the issue in the framework of some specific examples only, it is worth mentioning the case when a central planner chooses the controls of all the players in order to minimize an overall social cost. In the continuum player setting, the overall social cost is defined as:
\begin{equation*}
    \mathcal{S}(\balpha):=\int_I \mathcal{J}(\alpha_x,[\Zb \balpha]_x) dx.
\end{equation*}
In other words, the central planner problem is to solve for:
 $$
 \balpha^O\in \arg \inf_{\balpha \in \mathbb{A}} \mathcal{S}(\balpha).
 $$
Note that the social cost when computed for a Nash equilibrium can be no better than the optimal social cost $\cS(\balpha^O)$. So  if $\hat{\balpha}$ is a Nash equilibrium for the graphon game, and $\balpha^O$ is the optimal social cost, then $\mathcal{S}(\hat{\balpha}) \geq \mathcal{S}(\balpha^O)$. The discrepancy between the two costs is often referred to as the \emph{price of anarchy}. It was originally introduced to quantify the inefficiency of selfish behavior when compared to coordinated optimization. See for example 
\cite{christodoulou2005price,christodoulou2005price2,koutsoupias1999worst,roughgarden2009intrinsic,roughgarden2002bad}.
We shall compute the price of anarchy for one of our motivating examples in Section \ref{sub:cities_game} below.

\section{Links between Finite-Player Games and  Graphon Games}
 \label{sec:link_N_continuum}
We now turn our attention to the link between games with a finite number of players and the graphon games with a continuum of players considered in the previous section. But first, we recall some notations and definitions from the asymptotic theory of graphons and random graphs. See Lov\'{a}sz \cite[Chapter 8]{lovasz2012large} for details and complements.

The cut norm of a graphon $w$ is defined as:
$$
\|w\|_{\square}:=\sup_{S,T \in\cB(I)} \left\lvert\int_{S\times T}w(x,y)\; dx dy \right\rvert,
$$
where $\cB(I)$ is the Borel $\sigma$-field of $I$.
The cut metric between two graphons $w$ and $w'$ is defined as:
$$
d_{\square}(w,w'):=\|w-w'\|_{\square},
$$
and the cut distance by:
$$
\delta_{\square}(w,w'):=\inf_{\varphi \in \cS_I} d_{\square}(w^\varphi,w'),
$$
where $\cS_I$ denotes the set of invertible bi-measurable measure-preserving (i.e.\ preserving $\lambda_I$) maps from $I$  into  itself, and for each $\varphi \in \cS_I$, we  let $\balpha^\varphi:=(\alpha_{\varphi(x)})_{x \in I}$ and $w^\varphi(x,y):=w(\varphi(x),\varphi(y))$. 

Finally, the permutation invariant $L^2$ - distance between $\balpha,\balpha'\in L^2(I)$ is defined by:
$$
d_{\cS}(\balpha,\balpha'):=\inf_{\varphi\in \cS_I} \|\balpha^\varphi-\balpha' \|_{L^2(I)}.
$$

\subsection[Convergence of Finite-Player Game Equilibria]{\textbf{Convergence of Finite-Player Game Equilibria}}
In this section we give conditions under which Nash equilibria of sequences of $N$-player games converge in some sense to the Nash equilibria of a graphon game. In the next section we show that, conversely, Nash equilibria of graphon games can be used to provide approximate Nash equilibria for finite-player games.

First, we consider a convergent sequence of random graphs, $(W^N)_{N \geq 1}$, meaning that the corresponding evenly spaced step graphons, given by $\Psi^N(W^N)$ converge in the $\delta_{\square}$ cut distance to a graphon $w$.
\begin{remark}
The following are examples of convergent sequences of random graphs.
\begin{itemize}
    \item Erd\H{o}s-R\'{e}nyi model: Given $N \in \mathbb{Z}^+$ and $p \in [0,1]$, put an edge between each pair of the $N$ vertices independently with probability $p$. The limiting graphon is the constant graphon $w(\cdot,\cdot)=p$.
    \item Stochastic Block Model: Suppose there are $N$ vertices that are evenly divided among $K$ communities. The probability of an edge between a node in community $i$ and a node in community $j$ is given by $W_{i,j}$ where $W \in \mathbb{R}^{K \times K}$ is a symmetric matrix. The limiting graphon is the piecewise constant graphon $w:=\Psi^K(W)$.
    \item Watts-Strogatz model: Given $N\in \mathbb{Z}^+$, and $p,\gamma \in [0,1]$, construct the Watts-Strogatz random graph as follows: arrange the vertices in a circle, and connect each vertex to it's $pN/2$ nearest vertices on both sides. Then rewire each edge uniformly at random with probability $\gamma$. The limiting graphon is $w(x,y)=(1-\gamma (1-p))1_{d(x,y) \leq p/2}+\gamma p 1_{d(x,y) > p/2}$ where $d(x,y)$ is distance on the unit torus.
\end{itemize}
Other examples of convergent sequences of random graphs can be found in Lov\'{a}sz \cite[Section 11.4.2]{lovasz2012large}.
\end{remark}

\begin{remark}\label{remark:sampling}
Given a graphon $w$ which takes values in $[0,1]$, we can construct two sequences of random graphs, $\mathcal{G}^{N,1}$ and $\mathcal{G}^{N,1}$, according to the following sampling procedure:
\begin{itemize}
    \item Sample $N$ points $x^N_1,\dots,x^N_N$ i.i.d.\ from the uniform distribution on $I$. Without loss of generality, relabel the indices so that $x^N_1 \leq x^N_2 \dots \leq x^N_N$. We define $W^{N,1}, W^{N,2}$ as follows:
\begin{itemize}
    \item $W^{N,1}_{i,i}=0$, and for all $i\neq j$, $W^{N,1}_{i,j}=w(x_i,x_j)$.
    \item $W^{N,2}_{i,i}=0$, and for all $i\neq j$, $W^{N,2}_{i,j}=1$ with probability $w(x^N_i,x^N_j)$ and $0$ otherwise.
\end{itemize}
\end{itemize}
Then the corresponding evenly spaced step graphons for the above sampling procedures, $w^{N,1}:= \Psi^N(W^{N,1})$ and $w^{N,2}:=\Psi^N(W^{N,2})$, satisfy $\delta_\square(w^{N,1},w) \to 0$ almost surely and $\delta_\square(w^{N,2},w) \to 0$ almost surely. See Lov\'{a}sz \cite[Lemma 10.16]{lovasz2012large}.
\end{remark}

\begin{theorem} 
\label{thm:convergence}
Assume Assumptions \ref{hyp:b}, \ref{hyp:tilde_J}, \ref{hyp:J_2}, and \ref{hyp:2}. Suppose that $(W^N)_{N \geq 1}$ where $W^N=[W^N_{i,j}]_{i,j=1,\ldots,N} \in \mathbb{R}^{N \times N}$ for all $N\geq 1$ is a sequence of symmetric matrices, and suppose that for each $N\ge 1$,  there exists a unique Nash equilibrium denoted by $\balpha^{N,*}$ for the $N$-player game with the interaction matrix $W^N$. For each $N\ge 1$, let $w^N:=\Psi^N(W^N)$ and $\Wb^N$ be the corresponding graphon and graphon operator, and suppose there exists a graphon $w$ with corresponding graphon operator $\Wb$ such that $\lim_{N\to\infty}\delta_\square(w^N,w) = 0$, $\lim_{N\to\infty}\| w^N\|_2=\| w\|_2$, $\lim_{N\to\infty}\| \Wb^N\|=\| \Wb\|$, $w$ satisfies $\sqrt{2c_z} \| w \|_2 <1$ and $\Wb$ satisfies $ \sqrt{c_z} \| \Wb \| < 1$ and relation \eqref{fo:existence+uniqueness}. If 
$\hat{\balpha}$ denotes the unique Nash equilibrium for the graphon game with $w$, then $$\lim_{N\to\infty}d_{S}(\psi^N(\balpha^{N,*}),\hat{\balpha}) = 0.$$
\end{theorem}
In some cases, we can even obtain a rate of convergence, see Theorem~\ref{thm:convergence_2}.

\begin{remark}
This result is in a similar flavor to Delarue \cite[Theorem 1]{delarue2017mean}, which roughly states that the average of all the player's states in an $N$-player network game with an Erd\H{o}s-R\'{e}nyi graph converges to the average state in a mean field game. Since the graphon limit of an Erd\H{o}s-R\'{e}nyi graph is a constant graphon, Proposition \ref{pr:link_constant_MFG} suggests that it is sufficient to consider the analogous mean field game instead of considering the graphon game, which matches the approach of \cite{delarue2017mean}. The main difference with our result is that we show convergence of the strategy profiles, instead of convergence of the average of all player's states, and our result holds for any convergent sequence of random graphs. Also, our results are in the setting of static games, whereas \cite{delarue2017mean} considers a dynamic setting.
\end{remark}
 
The proof of Theorem \ref{thm:convergence} relies on the estimates proven in the following two lemmas.
\begin{lemma}
\label{lemma:W^N_and_W_N}
Let $N\ge 1$ and $W^N=[W^N_{i,j}]_{i,j=1,\ldots,N} \in \mathbb{R}^{N \times N}$ be a symmetric matrix, and we denote by $w^N:=\Psi^N(W^N)$ the corresponding evenly-spaced piecewise-constant graphon. Assume Assumption \ref{hyp:b} and $\sqrt{2c_z} \| W^N \|_F <1$. Let $\Zb_N:A^N\to L^2(\Omega;\RR)$ be defined as in \eqref{fo:Z_N} after Proposition \ref{pr:uniqueness_of_z_N} for the $N$-player game with interaction given by the matrix $W^N$, and $\Zb^N$ be the operator defined in Proposition \ref{pr:uniqueness_of_z} for the graphon $w^N$. Then for any $\balpha^N \in A^N$ we have: for all $i=1,\dots,N$,
\begin{equation*}
    \EE\left[\left\vert(\Zb_N{\balpha}^N)_i-[\mu^N(\Zb^N\psi^N(\balpha^N))]_i\right\vert^2\right]\leq \frac{1}{N^2}\sum_{j=1}^N (W^N_{i,j})^2 \left(\frac{2\mathbb{E}(\xi_0^2)}{1-2c_z \| W^N \|_F^2}\right).
\end{equation*}
\end{lemma}

\begin{proof}
We have:
\begin{equation*}
    \begin{split}
        &\left\vert(\Zb_N\balpha^N)_i-[\mu^N(\Zb^N\psi^N(\balpha^N))]_i\right\vert \\
        &\hskip 45pt
        \le \Bigl| (\Zb_N\balpha^N)_i-\frac{1}{N}\sum_{j=1}^N W^N_{i,j}X_{\alpha^N_j,[\mu^N(\Zb^N\psi^N(\balpha^N))]_j,\xi_j}\Bigr|\\
        &\hskip 75pt
        +\Bigl| \frac{1}{N}\sum_{j=1}^N W^N_{i,j}X_{\alpha^N_j,[\mu^N(\Zb^N\psi^N(\balpha^N))]_j,\xi_j}-\frac{1}{N}\sum_{j=1}^NW^N_{i,j}\EE\left[X_{\alpha^N_j,[\mu^N(\Zb^N\psi^N(\balpha^N))]_j,\xi_j}\right]\Bigr|\\
        &\hskip 75pt
        +\Bigl| \frac{1}{N}\sum_{j=1}^N W^N_{i,j}\mathbb{E}\left[X_{\alpha^N_j,[\mu^N(\Zb^N\psi^N(\balpha^N))]_j,\xi_j}\right]-[\mu^N(\Zb^N\psi^N(\balpha^N))]_i\Bigr| \\
        &\hskip 75pt
= (i) + (ii) + (iii).        
    \end{split}
\end{equation*}
We first consider the last term $(iii)$. Note that, using the definitions of $w^N$ and $\Zb^N$ we see that for any $\balpha\in\ L^2(I)$, if $x\in[(i-1)/N,i/N)$,
$$
(\Zb^N\balpha)_x=\sum_{j=1}^NW^N_{i,j}\int_{(j-1)/N}^{j/N}b\bigl(\alpha_y,[\Zb^N\balpha]_y\bigr)dy,
$$
showing that $x\mapsto (\Zb^N\balpha)_x$ is constant over each interval $[(i-1)/N,i/N)$, $i=1,\dots,N$. As a result, $[\mu^N(\bZ^N\balpha)]_i$ is merely the value of the piecewise constant function $\Zb^N\balpha$ over the  interval $[(i-1)/N,i/N)$.
Consequently:
\begin{equation*}
    \begin{split}
(iii)&= \Bigl| \frac{1}{N}\sum_{j=1}^N W^N_{i,j}\mathbb{E}\left[X_{\alpha^N_j,[\mu^N(\Zb^N\psi^N(\balpha^N))]_j,\xi_j}\right]
        -\sum_{j=1}^NW^N_{i,j}\int_{(j-1)/N}^{j/N}b\bigl([\psi^N(\balpha^N)]_y,[\Zb^N\psi^N(\balpha^N)]_y\bigr) dy \Bigr| \\
&= \Bigl| \frac{1}{N}\sum_{j=1}^N W^N_{i,j}\Bigl( b(\alpha^N_j,[\mu^N(\Zb^N\psi^N(\balpha^N))]_j)
        -N\int_{(j-1)/N}^{j/N}b(\alpha^N_j,[\Zb^N(\psi^N(\balpha^N)]_y)] dy\Bigr) \Bigr| \\
&= 0
\end{split}
\end{equation*}
because $y \mapsto [\Zb^N(\psi^N(\balpha^N)]_y$ is constant over the interval $[(j-1)/N,j/N)$, its value being $[\mu^N(\Zb^N\psi^N(\balpha^N))]_j$.
For the second term, since $[\mu^N(\Zb^N\psi^N(\balpha^N))]_j$ is deterministic, by Proposition \ref{pr:exact_law_large_numbers} we have:
\begin{equation*}
\begin{split}
        (ii)&=\Bigl| \frac{1}{N}\sum_{j=1}^N W^N_{i,j}\xi_j\Bigr|.
\end{split}
\end{equation*}
For the first term, we have:
\begin{equation*}
    \begin{split}
         (i)&=\Bigl| \frac{1}{N} \sum_{j=1}^N W^N_{i,j}X_{\alpha^N_j,(\Zb_N\balpha^N)_j,\xi_j}-\frac{1}{N}\sum_{j=1}^N W^N_{i,j}X_{\alpha^N_j,[\mu^N(\Zb^N\psi^N(\balpha^N))]_j,\xi_j}\Bigr|\\
         &
=\Bigl| \frac{1}{N} \sum_{j=1}^N W^N_{i,j}\Big[b\left(\alpha^N_j,(\Zb_N\balpha^N)_j\right)-b\left(\alpha^N_j,[\mu^N(\Zb^N\psi^N(\balpha^N))]_j\right)\Big]\Bigr|\\
         & 
 \le \sqrt{c_z} \frac{1}{N} \sum_{j=1}^N| W^N_{i,j}| \Bigl| (\Zb_N\balpha^N)_j-[\mu^N(\Zb^N\psi^N(\balpha^N))]_j\Bigr| ,
    \end{split}
\end{equation*}
where we used Assumption~\ref{hyp:b}.
Putting the above inequalities together, squaring, and taking the expectation, we have:
\begin{equation}
\label{eq:per_i}
\begin{split}
    &\mathbb{E}\left[\left((\Zb_N\balpha^N)_i-[\mu^N(\Zb^N\psi^N(\balpha^N))]_i\right)^2 \right] \\
    \le& 2c_z \frac{1}{N}\sum_{j=1}^N (W^N_{i,j})^2 \mathbb{E}\left( (\Zb_N\balpha^N)_j-[\mu^N(\Zb^N\psi^N(\balpha^N))]_j \right)^2  + 2 \mathbb{E} \left[\left(\frac{1}{N}\sum_{j=1}^N (W^N_{i,j})\xi_j\right)^2\right] \\
    \le& 2c_z \frac{1}{N}\sum_{j=1}^N (W^N_{i,j})^2 \frac{1}{N}\sum_{j=1}^N \mathbb{E}\left( (\Zb_N\balpha^N)_j-[\mu^N(\Zb^N\psi^N(\balpha^N))]_j \right)^2  + 2 \left(\frac{1}{N^2}\sum_{j=1}^N (W^N_{i,j})^2\mathbb{E}(\xi_0^2)\right),
\end{split}
\end{equation}
where we used the independence of the mean-zero $\xi_j$.
Thus, taking the average of the left hand side over $i$, we get:
\begin{equation*}
\begin{split}
    &\frac{1}{N} \sum_{i=1}^N \mathbb{E} \left[\left((\Zb_N\balpha^N)_i-[\mu^N(\Zb^N\psi^N(\balpha^N))]_i\right)^2\right] \\
    &\leq 2c_z \| W^N \|_F^2 \frac{1}{N}\sum_{j=1}^N \mathbb{E}\left( (\Zb_N\balpha^N)_j-[\mu^N(\Zb^N\psi^N(\balpha^N))]_j \right)^2  + \frac{2\| W^N \|_F^2\mathbb{E}(\xi_0^2)}{N},
\end{split}
\end{equation*}
and since $\sqrt{2c_z}\| W^N \|_F <1$ by assumption, we have:
\begin{equation*}
    \frac{1}{N} \sum_{i=1}^N\mathbb{E} \left((\Zb_N\balpha^N)_i-[\mu^N(\Zb^N\psi^N(\balpha^N))]_i\right)^2 \leq \frac{1}{N} \cdot \frac{2\| W^N \|_F^2\mathbb{E}(\xi_0^2)}{1-2c_z \| W^N \|_F^2}.
\end{equation*}
Returning to \eqref{eq:per_i}, we have:
\begin{equation*}
\begin{split}
    &\mathbb{E}\left((\Zb_N\balpha^N)_i-[\mu^N(\Zb^N\psi^N(\balpha^N))]_i\right)^2 \\
    \le& 2c_z \frac{1}{N}\sum_{j=1}^N (W^N_{i,j})^2 \frac{1}{N} \cdot \frac{2\| W^N \|_F^2\mathbb{E}(\xi_0^2)}{1-2c_z \| W^N \|_F^2}  + 2 \left(\frac{1}{N^2}\sum_{j=1}^N (W^N_{i,j})^2\mathbb{E}(\xi_0^2)\right) \\
    =&\frac{1}{N^2}\sum_{j=1}^N (W^N_{i,j})^2 \left(\frac{2\mathbb{E}(\xi_0^2)}{1-2c_z \| W^N \|_F^2}\right),
\end{split}
\end{equation*}
which concludes the proof.
\end{proof}

\begin{lemma}
\label{lemma:convergence_controls}
As before, for a sequence of symmetric matrices,  $W^N=[W^N_{i,j}]_{i,j=1,\ldots,N} \in \mathbb{R}^{N \times N}$, $N \geq 1$, we denote by $w^N:=\Psi^N(W^N)$ the corresponding evenly-spaced piecewise-constant graphon. 
Let $\balpha^{N,*}\in A^N$ be a Nash equilibrium for the $N$-player game with interaction $W^N$, and $\hat\balpha^N$ a Nash equilibrium for the graphon game with graphon $w^N$. Suppose $\lim_{N\to\infty}\| w^N\|_2=\zeta_1$ and $\lim_{N\to\infty}\| \Wb^N\|=\zeta_2$ where $\Wb^N$ is the graphon operator associated to the piecewise constant graphon $w^N$. Suppose $\sqrt{2c_z} \zeta_1 <1$, $\sqrt{c_z} \zeta_2 <1$, and $\frac{\ell_J}{\ell_c}\cdot \frac{\sqrt{c_\alpha} \zeta_2}{1-\sqrt{c_z} \zeta_2}<1$. Then under Assumptions \ref{hyp:b}, \ref{hyp:tilde_J}, and \ref{hyp:J_2}, there exists an $\epsilon>0$ such that for $N$ large enough we have:
\begin{equation}
\label{fo:d_s}
\|\psi^N(\balpha^{N,*})-\hat{\balpha}^N\|_{L^2(I)}\leq \frac{\tilde{\kappa}}{N^{1/4}},\qquad \tilde{\kappa}:=2\kappa_1 \sqrt{\frac{2\tilde \ell_J}{\ell_c}} \left(2\kappa_0+\kappa_2\right)^{1/4},
\end{equation}
where
\begin{equation*}
    \kappa_0:=\left( \frac{2(\zeta_1 + \epsilon)^2 \mathbb{E}(\xi_0^2)}{1-2c_z (\zeta_1+ \epsilon)^2} \right),
\end{equation*}
\begin{equation*}
            \kappa_1:= \left( 1-\frac{\ell_J}{\ell_c} \cdot \frac{\sqrt{c_\alpha}(\zeta_2+\epsilon)}{1-\sqrt{c_z} (\zeta_2+ \epsilon)} \right)^{-1}
\qquad\text{and}\qquad
            \kappa_2:=\frac{c_\alpha (\zeta_1+ \epsilon)^2}{(1-\sqrt{c_z}(\zeta_2+\epsilon))^2}.
\end{equation*}
\end{lemma}
    
\begin{proof}
We use the notation $\Zb^N$ for the operator associated to the graphon $w^N$ through Propositions \ref{pr:uniqueness_of_z} and \ref{pr:exact_law_large_numbers}. We have:
    \begin{equation}
    \label{eq:lem-kappas-tmp}
        \begin{split}
             \|\psi^N(\balpha^{N,*})-\hat{\balpha}^N\|_{L^2(I)}
             &=\|\psi^N(\balpha^{N,*})-\Bb\Zb^N\hat{\balpha}^N\|_{L^2(I)}\\
            & \leq \|\psi^N(\balpha^{N,*})-\Bb\Zb^N\psi^N(\balpha^{N,*})\|_{L^2(I)}+\|\Bb\Zb^N\psi^N(\balpha^{N,*})-\Bb\Zb^N\hat{\balpha}^N\|_{L^2(I)}.
        \end{split}
    \end{equation}
Since $ \sqrt{c_z} \zeta_2 <1$ and $\frac{\ell_J}{\ell_c}\cdot \frac{\sqrt{c_\alpha} \zeta_2}{1-\sqrt{c_z} \zeta_2}<1$, there exists an $\epsilon>0$ such that:
\begin{equation}
\label{eq:choose_epsilon_zeta_2}
\sqrt{c_z}(\zeta_2+ \epsilon)<1, \quad \text{and} \quad\frac{\ell_J}{\ell_c} \cdot \frac{\sqrt{c_\alpha}(\zeta_2 + \epsilon)}{1-\sqrt{c_z}(\zeta_2+ \epsilon)} < 1.
\end{equation}
Since $\|\Wb^N\| \to \zeta_2$,  there exists an $N^*$ such that for all $N \geq N^*$, $\left\vert \zeta_2-\|\Wb^N\|\right\vert \leq \epsilon$. For the remainder of the proof,
 consider $N \geq N^*$. Using Lemmas \ref{lemma:B_Lipschitz} and \ref{lemma:W_Lip}, we get:
$$
\|\Bb\Zb^N\psi^N(\balpha^{N,*})-\Bb\Zb^N\hat{\balpha}^N\|_{L^2(I)}
     \leq \frac{\ell_J}{\ell_c} \cdot \frac{\sqrt{c_\alpha}(\zeta_2+ \epsilon)}{1-\sqrt{c_z} (\zeta_2+ \epsilon)} 
     \|\psi^N(\balpha^{N,*})-\hat{\balpha}^N\|_{L^2(I)},
$$
which implies by~\eqref{eq:lem-kappas-tmp} that:
$$
\Bigl(1-\frac{\ell_J}{\ell_c} \cdot  \frac{\sqrt{c_\alpha}(\zeta_2+ \epsilon)}{1-\sqrt{c_z} (\zeta_2+ \epsilon)}\Bigr) \|\psi^N(\balpha^{N,*})-\hat{\balpha}^N\|_{L^2(I)} 
\le \|\psi^N(\balpha^{N,*})-\Bb\Zb^N\psi^N(\balpha^{N,*})\|_{L^2(I)},
$$
from which we conclude that:
\begin{equation}
\label{fo:d_S_temp}
    \|\psi^N(\balpha^{N,*})-\hat{\balpha}^N\|_{L^2(I)} \leq \kappa_1\|\psi^N(\balpha^{N,*})-\Bb\Zb^N\psi^N(\balpha^{N,*})\|_{L^2(I)}.
\end{equation}
We now estimate the above right hand side. For $N$ and $i \in \{1,\ldots,N\}$ fixed, we denote:
    $$
    F_i(\beta):=\tilde\cJ\left(\beta,\mathcal{L}(\xi_i,[\Zb_N[\beta;\balpha^{N,*}_{-i}]]_i)\right), \qquad \beta\in A.
    $$
 We have  $\alpha^{N,*}_i \in \arg \inf_{\beta\in A} F_i(\beta)$, since $\balpha^{N,*}$ is a Nash equilibrium for the $N$-player game.
      To simplify notation, let $\beta^1_i:=\alpha^{N,*}_i$.
    
    Note that since $\psi^N(\balpha^{N,*})$ is a step function, taking constant values on intervals $[(j-1)/N,j/N)$, we can choose an element in the $L^2$-class of $\Zb^N\psi^N(\balpha^{N,*})$ which is also constant on the intervals $[(j-1)/N,j/N)$. In particular, $\mu^N\bigl(\Zb^N\psi^N(\balpha^{N,*})\bigr)\in A^N$ gives these constant values. Now, denote:
$$
G_i(\beta)=\cJ\left(\beta,[\mu^N(\Zb^N\psi^N(\balpha^{N,*}))]_i\right).
$$
Then $\left[\Bb\psi^N(\mu^N(\Zb^N\psi^N(\balpha^{N,*})))\right]_x= \arg \inf_{\beta \in A}G_i(\beta)$ for $\lambda_I$-a.e.\ $x \in [(i-1)/N,i/N)$, where we used the fact that the minimizer exists and is unique from Assumption \ref{hyp:J_2}. To simplify notation, let $\beta^2_i:=\arg \inf_{\beta \in A}G_i(\beta)$.

From Assumption \ref{hyp:J_2}, $G_i$ is strongly convex with parameter $\ell_c$, and thus letting $\epsilon=\frac{1}{2}$ in equation \eqref{eq:convexity}, and using the fact that $\beta^2_i$ minimizes $G_i(\cdot)$ and $\beta^1_i$ minimizes $F_i(\cdot)$, we have:
    \begin{equation}
    \label{fo:beta_diff}
        \begin{split}
            (\beta^1_i-\beta^2_i)^2 &\leq \frac{4}{\ell_c}(G_i(\beta^1_i)-G_i(\beta^2_i))\\
            &\leq \frac{4}{\ell_c} \Bigl( G_i(\beta^1_i)-F_i(\beta^1_i)+ F_i(\beta^2_i)-G_i(\beta^2_i)\Bigr) \\
            &\leq \frac{4}{\ell_c} \Bigl( \left\vert  F_i(\beta^1_i)-G_i(\beta^1_i)\right\vert+ \left\vert F_i(\beta^2_i)-G_i(\beta^2_i)\right\vert\Bigr).
        \end{split}
    \end{equation}
From Assumption \ref{hyp:tilde_J} and the definitions of $F_i$ and $G_i$, we have that for all $\beta \in A$ and all $i\in \{1,\dots,N\}$:
\begin{equation}
\label{eq:ineq-Fi-Gi-beta}
    \left\vert F_i(\beta)-G_i(\beta) \right\vert\leq \tilde \ell_J W_2\left(\mathcal{L}\left(\xi_i,\Zb_N[\beta;(\balpha^{N,*})_{-i}])_i\right),\mu_0 \otimes \delta_{[\mu^N(\Zb^N\psi^N(\balpha^{N,*}))]_i}\right).
\end{equation}
The $W_2$ Wasserstein distance between two probability measures is bounded from above by the $L^2$-distance between random variables with these distributions. By choosing random variables $(\xi_i, \Zb_N[\beta;(\balpha^{N,*})_{-i}])_i) \sim \mathcal{L}(\xi_i,\Zb_N[\beta;(\balpha^{N,*})_{-i}])_i)$ and $(\xi_i,[\mu^N(\Zb^N\psi^N(\balpha^{N,*}))]_i) \sim \mu_0 \otimes \delta_{[\mu^N(\Zb^N\psi^N(\balpha^{N,*}))]_i}$, we have
\begin{equation}
\label{eq:ineq-W2-ZN-delta}
    \begin{split}
        &W_2(\mathcal{L}(\xi_i,\Zb_N[\beta;(\balpha^{N,*})_{-i}])_i), \mu_0 \otimes\delta_{[\mu^N(\Zb^N\psi^N(\balpha^{N,*}))]_i})^2 \\
        \leq & \EE\left[\left\vert (\xi_i-\xi_i)^2 \right\vert\right] + \EE\left[\left\vert [\Zb_N[\beta;(\balpha^{N,*})_{-i}]]_i-[\mu^N(\Zb^N\psi^N(\balpha^{N,*}))]_i\right\vert^2\right]
        \\
        \leq & 2 \EE\Bigl[\bigl|[\Zb_N[\beta;(\balpha^{N,*})_{-i}]]_i-[\mu^N(\Zb^N\psi^N([\beta;(\balpha^{N,*})_{-i}]))]_i\bigr|^2\Bigr]
        \\
        &+2\left\vert [\mu^N(\Zb^N\psi^N[\beta;(\balpha^{N,*})_{-i}])]_i-[\mu^N(\Zb^N\psi^N(\balpha^{N,*}))]_i\right\vert^2.
    \end{split}
\end{equation}
Averaging equation \eqref{fo:beta_diff}  over $i$ and combining~\eqref{eq:ineq-Fi-Gi-beta} and \eqref{eq:ineq-W2-ZN-delta} for $\beta = \beta_i^1, \beta_i^2$, we arrive at:
    \begin{equation}
    \label{eq:average_beta_diffs}
        \begin{split}
            \frac{1}{N}\sum_{i=1}^N(\beta^1_i-\beta^2_i)^2 \leq \frac{4\sqrt{2}\tilde \ell_J}{\ell_c}&\Bigl( \frac{1}{N}\sum_{i=1}^N \EE\Bigl[\bigl|[\Zb_N[\beta^1_i;(\balpha^{N,*})_{-i}]]_i-[\mu^N(\Zb^N\psi^N([\beta^1_i;(\balpha^{N,*})_{-i}]))]_i\bigr|^2\Bigr]
            \\
            &+\frac{1}{N}\sum_{i=1}^N\EE\Bigl[\bigl|[\Zb_N[\beta^2_i;(\balpha^{N,*})_{-i}]]_i-[\mu^N(\Zb^N\psi^N([\beta^2_i;(\balpha^{N,*})_{-i}]))]_i\bigr|^2\Bigr]
            \\
            &+\frac{1}{N}\sum_{i=1}^N \left\vert [\mu^N(\Zb^N\psi^N[\beta^1_i;(\balpha^{N,*})_{-i}])]_i-[\mu^N(\Zb^N\psi^N(\balpha^{N,*}))]_i\right\vert^2
            \\
            &+\frac{1}{N}\sum_{i=1}^N \left\vert [\mu^N(\Zb^N\psi^N[\beta^2_i;(\balpha^{N,*})_{-i}])]_i-[\mu^N(\Zb^N\psi^N(\balpha^{N,*}))]_i\right\vert^2 \Bigr)^{1/2}.
        \end{split}
    \end{equation}

Since $\sqrt{2c_z} \zeta_1 <1$, we can also choose $\epsilon$ so that $\sqrt{2c_z} (\zeta_1 + \epsilon) <1$, and we can also take $N^*$ large enough so that $|\zeta_1 - \|W^N\|_F | < \epsilon$ for all $N \geq N^*$. Then for $N \geq N^*$, we have the hypothesis of Lemma \ref{lemma:W^N_and_W_N}, $\sqrt{2c_z} \| W^N \|_F <1$. Using Lemma \ref{lemma:W^N_and_W_N}, for any $\bbeta \in L^2(I)$ we get:
\begin{equation*}
\frac{1}{N}\sum_{i=1}^N \EE\left[\left\vert(\Zb_N[\beta_i;(\balpha^{N,*})_{-i}]_i-[\mu^N(\Zb^N\psi^N([\beta;(\balpha^{N,*})_{-i}]))]_i\right\vert^2\right]\leq \frac{\kappa_0}{N}.
\end{equation*}

For any $\bbeta$, using \eqref{fo:pointwise} from Lemma \ref{lemma:W_Lip}, and averaging over $i$ we have:
    \begin{equation*}
        \begin{split}
            &\frac{1}{N} \sum_{i=1}^N\left\vert \left[\mu^N\Big(\Zb^N\psi^N([\beta;\balpha^{N,*}_{-i}\right])\Big)]_i-\left[\mu^N\Big(\Zb^N\psi^N(\balpha^{N,*})\Big)\right]_i\right\vert^2
            \\
            &\hskip 45pt
            \leq \frac{c_\alpha}{(1-\sqrt{c_z} \|\Wb^N\|)^2} \Bigl( \frac{1}{N^2}\sum_{i=1}^N\sum_{j=1}^N (W^N_{i,j})^2\Bigr)\frac{1}{N}\sum_{i=1}^N \|\psi^N[\beta;(\balpha^{N,*})_{-i}]-\psi^N(\balpha^{N,*}) \|^2_{L^2(I)}
            \\
            &\hskip 45pt
            = \frac{c_\alpha \|W^N\|_F^2}{(1-\sqrt{c_z} \|\Wb^N\|)^2} \cdot \frac{1}{N^2}\sum_{i=1}^N | \beta-\alpha^{N,*}_i |^2
            \\
            &\hskip 45pt
            \leq \frac{c_\alpha (\zeta_1 + \epsilon)^2}{(1-\sqrt{c_z}(\zeta_2+\epsilon))^2} \cdot \frac{1}{N^2}\sum_{i=1}^N | \beta-\alpha^{N,*}_i |^2
            \\
            &\hskip 45pt
            =\frac{\kappa_2}{N^2} \sum_{i=1}^N| \beta-\beta^1_i |^2,
        \end{split}
    \end{equation*}
    where we used the fact that $\beta^1_i=\alpha^{N,*}_i$. Now returning to equation \eqref{eq:average_beta_diffs} we have:
    \begin{equation*}
    \begin{split}
        \left(\frac{1}{N}\sum_{i=1}^N(\beta^1_i-\beta^2_i)^2 \right)^2 &\leq \frac{32\tilde \ell_J^2}{\ell_c^2} \left(\frac{2\kappa_0}{N}+ \frac{\kappa_2}{N^2} \sum_{i=1}^N ( \beta^1_i-\beta^2_i )^2 \right) \\
        & \leq \frac{32\tilde \ell_J^2}{\ell_c^2} \left(\frac{2\kappa_0+\kappa_2}{N}+\frac{\kappa_2}{N} \left( \frac{1}{N} \sum_{i=1}^N ( \beta^1_i-\beta^2_i )^2\right)^2 \right),
    \end{split}
    \end{equation*}
or after rearranging:
    \begin{equation*}
    \begin{split}
        \left(1- \frac{32\tilde \ell_J^2 \kappa_2}{\ell_c^2} \frac{1}{N}\right)\left(\frac{1}{N}\sum_{i=1}^N(\beta^1_i-\beta^2_i)^2 \right)^2 & \leq \frac{32\tilde \ell_J^2}{\ell_c^2} \left(\frac{2\kappa_0+\kappa_2}{N}\right).
    \end{split}
    \end{equation*}
We can again choose $N^*$ such that $\frac{32\tilde \ell_J^2 \kappa_2}{\ell_c^2} \frac{1}{N} < \frac{1}{2}$ for all $N \geq N^*$, and thus:
    \begin{equation*}
    \begin{split}
        \left(\frac{1}{N}\sum_{i=1}^N(\beta^1_i-\beta^2_i)^2 \right)^2 & \leq \frac{64\tilde \ell_J^2}{\ell_c^2} \left(\frac{2\kappa_0+\kappa_2}{N}\right).
    \end{split}
    \end{equation*}

We compute:
    \begin{equation*}
        \begin{split}
            \left\Vert \psi^N(\balpha^{N,*})-\Bb\Zb^N\psi^N(\balpha^{N,*}) \right\Vert ^2_{L^2(I)}&=\left\Vert \psi^N(\balpha^{N,*})-\Bb\psi^N(\mu^N(\Zb^N\psi^N(\balpha^{N,*}))) \right\Vert ^2_{L^2(I)} \\
            &=\frac{1}{N}\sum_{i=1}^N\left[\beta^1_i-\beta^2_i\right]^2\\
            & \leq \frac{8\tilde \ell_J}{\ell_c} \sqrt{\left(\frac{2\kappa_0+\kappa_2}{N}\right)},
        \end{split}
    \end{equation*}
proving the desired estimate \eqref{fo:d_s} because of \eqref{fo:d_S_temp}.
\end{proof}

\vskip 12pt
\begin{proof}[Proof of Theorem \ref{thm:convergence}]
Let $\hat\balpha^N$ be a Nash equilibrium for the graphon game with the graphon $w^N$, i.e.\ $\hat\balpha^N$ is a fixed point of $\Bb\Zb^N$. For each $N\ge 1$,  the triangle inequality implies
$$
d_{\cS}(\hat\balpha,\psi^N(\balpha^{N,*}))\leq d_\cS(\hat\balpha,\hat\balpha^N) +  d_{\cS}(\hat\balpha^N,\psi^N(\balpha^{N,*})).
$$
Since by assumption we have $\delta_\square(w^N,w) \to 0$,  \cite[Theorem 11.59]{lovasz2012large} implies that for each $N\ge 1$, there exists a relabeling of the \emph{steps} of $w^N$ which is a measure preserving invertible function $\varphi^N \in \cS_{I}$  such that if we set $\tilde w^N=(w^N)^{\varphi^N}$ and denote by $\tilde\Wb^N$ the associated graphon operator, then we have the convergence  of the operators: $||\tilde\Wb^N-\Wb|| \to 0$. Now using Theorem \ref{thm:perturbation}, we get $\left\Vert\left(\balpha^N\right)^{\varphi^N}-\hat{\balpha}\right\Vert_{L^2(I)} \to 0$, which implies $d_{\cS}(\hat{\balpha},\hat{\balpha}^N) \to 0$. Since by assumption $||w^N ||_2 \to ||w||_2$ and $||\Wb^N || \to ||\Wb||$, where $||w||_2$ and $||\Wb||$ satisfy the assumptions of $\zeta_1$ and $\zeta_2$, respectively, in Lemma \ref{lemma:convergence_controls}, we have $d_{\cS}(\hat\balpha^N,\psi^N(\balpha^{N,*}))\to 0$. This concludes the proof.
\end{proof}

Theorem \ref{thm:convergence} says that the set of controls for the $N$-player Nash equilibrium (assuming it exists) converges in some sense to the Nash equilibrium strategy for the graphon game. If the graphs defining the $N$-player games are constructed with one of the sampling procedures in Remark \ref{remark:sampling}, then we can estimate the rate of convergence.

\begin{theorem} \label{thm:convergence_2}
Consider the same setting as Theorem \ref{thm:convergence}. In addition, assume that the graphon $w$ takes values in $[0,1]$, and that the sequence of matrices $(W^N)_{N \geq 1}$ are a random sequence constructed with one of the sampling procedures in Remark \ref{remark:sampling}. Further suppose $\sqrt{2c_z} < 1$. There exists an $N^*$ such that for every $N \geq N^*$, with probability at least $1-\exp(-N/(2\log_2N))$ it holds:
$$
d_{S}(\psi^N(\balpha^{N,*}),\hat{\balpha}) \leq \frac{\kappa^*}{N^{1/4}}+ \frac{\kappa \sqrt{184}}{(\log_2N)^{1/4}},
$$
with $$\kappa^*:=2\kappa_1^* \sqrt{\frac{2\tilde \ell_J}{\ell_c}} \left(2\kappa_0^*+\kappa_2^*\right)^{1/4}, \quad  \kappa_1^*:= \left( 1-\frac{\ell_J}{\ell_c} \cdot \frac{\sqrt{c_\alpha}(\| \Wb \|+\epsilon)}{1-\sqrt{c_z} (\| \Wb \|+ \epsilon)} \right)^{-1},$$

$$\kappa_0^*:=\left( \frac{2 \mathbb{E}(\xi_0^2)}{1-2c_z} \right), \quad \text{and} \quad \kappa_2^*:=\frac{c_\alpha }{(1-\sqrt{c_z}(\| \Wb \|+\epsilon))^2}.$$
\end{theorem}

\begin{proof}
By the triangle inequality,
$$
d_{S}(\psi^N(\balpha^{N,*}),\hat{\balpha})\leq d_{S}(\psi^N(\balpha^{N,*}),\hat{\balpha}^N)+d_{S}(\hat{\balpha}^N,\hat{\balpha}).
$$
As in the proof of \cite[Theorem 6]{parise2018graphon}, there exists $\varphi^N \in \cS_{I}$ for every $N \geq 1$ such that with probability at least $1-\exp(-N/(2\log_2N))$ it holds:
\begin{equation*}
    \left\Vert (w^N)^{\varphi^N}-w \right\Vert_\square \leq \frac{23}{\sqrt{\log_2 N}}.
\end{equation*}
Thus with \cite[Lemma 4]{parise2018graphon}, with probability at least $1-\exp(-N/(2\log_2N))$ it holds:
\begin{equation*}
    ||(\Wb^N)^{\varphi^N}-\Wb|| \leq \sqrt{8\left\Vert(w^N)^{\varphi^N}-w\right\Vert_\square} \leq \frac{\sqrt{184}}{(\log_2N)^{1/4}}.
\end{equation*}
Then with Theorem \ref{thm:perturbation}:
\begin{equation*}
    d_{S}(\hat{\balpha}^N,\hat{\balpha}) \leq \left\Vert (\hat{\balpha}^N)^{\varphi^N}-\hat{\balpha} \right\Vert_{L^2(I)} \leq \kappa ||(\Wb^N)^{\varphi^N}-\Wb|| \leq \frac{\kappa \sqrt{184}}{(\log_2N)^{1/4}}.
\end{equation*}

Now we modify the proof of Lemma \ref{lemma:convergence_controls} to see that we can choose $N^*$ independent of the realization of the sequence $(W^N)_{N \geq 1}$, so that for all $N \geq N^*$, with probability at least $1-\exp(-N/(2\log_2N))$, the statement of Lemma \ref{lemma:convergence_controls} holds. Indeed, let $\epsilon>0$ so that \eqref{eq:choose_epsilon_zeta_2} holds with $\zeta_2$ replaced by $\| \Wb \|$. Now let $N^*$ so that $\frac{\sqrt{184}}{(\log_2N^*)^{1/4}} \leq \epsilon$. Then for all $N \geq N^*$, with probability at least $1-\exp(-N/(2\log_2N))$, we have $||(\Wb^N)^{\varphi^N}-\Wb|| \leq \epsilon $, which implies $\left\vert \|(\Wb^N)^{\varphi^N}\|-\|\Wb\| \right\vert \leq \epsilon $. Because the graphon takes values in $[0,1]$ by assumption, clearly we have $\| W^N \|_F \leq 1$ for all realizations of the sequence and for all $N \geq 1$. Since we assumed $\sqrt{2c_z} < 1$, the hypothesis of Lemma \ref{lemma:W^N_and_W_N} is satisfied for all realizations and for all $N \geq 1$, and we can give a larger upper bound for the right hand side of the statement in Lemma \ref{lemma:W^N_and_W_N}: $\EE\left[\left\vert(\Zb_N{\balpha}^N)_i-[\mu^N(\Zb^N\psi^N(\balpha^N))]_i\right\vert^2\right]\leq \frac{1}{N} \left(\frac{2\mathbb{E}(\xi_0^2)}{1-2c_z}\right)$.
Thus for all $N \geq N^*$, with probability at least $1-\exp(-N/(2\log_2N))$:
$$
d_{S}(\psi^N(\balpha^{N,*}),\hat{\balpha}^N) \leq \frac{\kappa^*}{N^{1/4}},
$$
and we conclude.
\end{proof}

\subsection[Approximate Equilibria]{\textbf{Approximate Equilibria}}
Now we consider whether $N$ players using some controls picked from the graphon game strategy are in an approximate Nash equilibrium. More specifically, we define the following.

\begin{definition}[$\epsilon$-Nash equilibrium]
We say a control profile $\bm{\alpha}^N=(\alpha^N_1,\dots,\alpha^N_N)$ is an $\epsilon$-Nash equilibrium for the $N$-player game if:
\begin{equation*}
    \tilde\cJ(\alpha^N_i,\mathcal{L}(\xi_i,[\Zb_N\bm{\alpha}^N]_i))\leq \tilde\cJ(\beta,\mathcal{L}(\xi_i, [\Zb_N[\beta;\bm{\alpha}^N_{-i}]]_i) + \epsilon, \quad \forall i=1,\dots, N, \forall \beta\in A.
\end{equation*}
\end{definition}

\begin{theorem}\label{thm:epsilon_Nash}
Assume Assumptions \ref{hyp:b}, \ref{hyp:tilde_J}, \ref{hyp:J_2}, and \ref{hyp:2}, and that the set of controls, $A$, is bounded. Suppose that $(W^N)_{N \geq 1}$ where $W^N=[W^N_{i,j}]_{i,j=1,\ldots,N} \in \mathbb{R}^{N \times N}$ for all $N\geq 1$ is a sequence of symmetric matrices satisfying $\max_{i=1,\dots,N}\frac{1}{N^2}\sum_{j=1}^N (W^N_{i,j})^2 \leq \tilde{\zeta}$. For each $N\ge 1$, let $w^N:=\Psi^N(W^N)$, and suppose there exists a graphon $w$ satisfying $\sqrt{2c_z} \|w\|_2 <1$, $ \sqrt{c_z} \| \Wb \| < 1$, and $\sup_{x\in I} (\int_I w(x,y)^2dy )^{1/2}\leq \zeta$, such that $\lim_{N\to\infty}\| w^N\|_2=\| w\|_2$, $\lim_{N\to\infty}\| \Wb^N\|=\| \Wb\|$, and $\delta_\square(w^N,w) \leq  \gamma_N$ with $ N\gamma_N \to 0$ as $N\to \infty$. Suppose the indices of $W^N$ are already permuted such that the same is true for $d_\square(w^N,w)$.
Let $\hat{\balpha}$ be a Nash equilibrium for the graphon game with graphon $w$. Suppose we use $\hat{\balpha}$ to define a strategy $\mu^N(\hat{\balpha}) \in A^N$. Then $\mu^N(\hat{\balpha})$ is an $\epsilon_N$-Nash equilibrium for the $N$-player games specified by $W^N$ with $\epsilon_N\to 0$ as $N\to \infty$.
\end{theorem}

The proof of Theorem \ref{thm:epsilon_Nash} relies on the following two lemmas.
\begin{lemma}\label{lemma:A_convergence}
Suppose for a sequence of symmetric matrices $(W^N)_{N\geq1}$, and corresponding piecewise constant graphons $w^N:=\Psi^N(W^N)$, there exists a graphon $w$ such that $d_\square(w^N,w) \leq  \gamma_N$ with $ N\gamma_N \to 0$ as $N\to \infty$. Let $g:I\to \mathbb{R}$ be a bounded and measurable function. Then
$$ \sup_{i \in \{1,\dots,N \}} N \left\vert \int_{(i-1)/N}^{i/N}\int_I \left( w(x,y)-w^N(x,y) \right)g(y)dydx \right\vert \to 0.$$
\end{lemma}
\begin{proof}
    For a given $N \geq 1$ and $i\in \{1,\dots,N \}$, we have:
    \begin{equation*}
        \begin{split}
            N\left\vert\int_{(i-1)/N}^{i/N}\int_I \left( w(x ,y)-w^N(x ,y)\right)g(y) dy dx  \right\vert=& N\left\vert\int_I \int_{(i-1)/N}^{i/N} \left( w(x ,y)-w^N(x ,y)\right) dx  g(y) dy\right\vert\\
            \leq & N\left\vert \int_{T^+} \int_{(i-1)/N}^{i/N} \left( w(x ,y)-w^N(x ,y)\right) dx  g(y) dy \right\vert\\
            &+N\left\vert \int_{T^-} \int_{(i-1)/N}^{i/N} \left( w(x ,y)-w^N(x ,y)\right) dx  g(y) dy \right\vert,
        \end{split}
    \end{equation*}
    where $T^+:=\{y: \int_{(i-1)/N}^{i/N} \left( w(x ,y)-w^N(x ,y)\right) dx  \geq 0\}$ and $T^-:=\{y: \int_{(i-1)/N}^{i/N} \left( w(x ,y)-w^N(x ,y)\right) dx  < 0\}$.
    Then:
    \begin{equation*}
            \left\vert \int_{T^+} \int_{(i-1)/N}^{i/N} \left( w(x ,y)-w^N(x ,y)\right) dx  g(y) dy \right\vert \leq C \int_{T^+} \int_{(i-1)/N}^{i/N} \left( w(x ,y)-w^N(x ,y)\right) dx  dy,
    \end{equation*}
    and
    \begin{equation*}
            \left\vert \int_{T^-} \int_{(i-1)/N}^{i/N} \left( w(x ,y)-w^N(x ,y)\right) dx  g(y) dy \right\vert \leq C \int_{T^-} \int_{(i-1)/N}^{i/N} \left( w^N(x ,y)-w(x ,y)\right) dx  dy.
    \end{equation*}
    Since $d_\square(w^N,w)\leq  \gamma_N$, we have:
    \begin{equation*}
        \sup_{S,T \in \mathcal{B}(I)} \left\lvert\int_{S\times T}\left(w(x,y)-w^N(x,y)\right)dx dy \right\rvert \leq  \gamma_N.
    \end{equation*}
    Thus,
    $$\int_{T^+} \int_{(i-1)/N}^{i/N} \left( w(x ,y)-w^N(x ,y)\right) dx  dy \leq  \gamma_N,$$
    and
    $$\int_{T^-} \int_{(i-1)/N}^{i/N} \left( w^N(x ,y)-w(x ,y)\right) dx  dy \leq  \gamma_N.$$
    Therefore,
    \begin{equation*}
        \sup_{i \in \{1,\dots,N \}} N\left\vert\int_{(i-1)/N}^{i/N}\int_I \left( w(x ,y)-w^N(x ,y)\right)g(y) dy dx  \right\vert \leq 2 N  \gamma_N \to 0.
    \end{equation*}
\end{proof}

\begin{lemma}\label{lemma:W_W_N}
Suppose for a sequence of symmetric matrices $(W^N)_{N\geq1}$ with corresponding piecewise constant graphons $w^N:=\Psi^N(W^N)$ satisfying $ \sqrt{c_z} \| \Wb^N \| < 1$, there exists a graphon $w$ satisfying \\ $\sup_{x\in I} (\int_I w(x,y)^2dy )^{1/2}\leq \zeta$ and $ \sqrt{c_z} \| \Wb \| < 1$, such that $d_\square(w^N,w) \leq  \gamma_N$ with $ N\gamma_N \to 0$ as $N\to \infty$. Under Assumptions \ref{hyp:b}, \ref{hyp:J_2}, and \ref{hyp:2}:
$$\sup_{i \in \{1,\dots,N \}} N \left\vert \int_{(i-1)/N}^{i/N}   [\Zb \psi^N \bm{\alpha}^N]_x - [\Zb^N \psi^N \bm{\alpha}^N]_x dx \right\vert \to 0,$$
for an arbitrary sequence of controls, $\bm{\alpha}^N\in A^N$, $N\geq 1$.
\end{lemma}
\begin{proof}
Fix $N \geq 1$ and $i \in \{1,\dots,N \}$. With equation \eqref{eq:connection_W_A} and its counterpart for $\Zb^N$ and $\Wb^N$, we have:
\begin{equation*}
    \begin{split}
        & N\left\vert \int_{(i-1)/N}^{i/N} [\Zb\psi^N(\bm{\alpha}^N)]_x-[\Zb^N\psi^N(\bm{\alpha}^N)]_x dx\right\vert 
        \\
        =&N \left\vert \int_{(i-1)/N}^{i/N} \left[\Wb\left[b(\psi^N(\bm{\alpha}^N),\Zb\psi^N(\bm{\alpha}^N))\right]\right]_x-\left[\Wb^N\left[b(\psi^N(\bm{\alpha}^N),\Zb^N\psi^N(\bm{\alpha}^N))\right]\right]_x  dx\right\vert 
        \\
         \leq& N \int_{(i-1)/N}^{i/N}\left\vert  \left[\Wb\left[b(\psi^N(\bm{\alpha}^N),\Zb\psi^N(\bm{\alpha}^N))\right]\right]_x-\left[\Wb\left[b(\psi^N(\bm{\alpha}^N),\Zb^N\psi^N(\bm{\alpha}^N))\right]\right]_x \right\vert dx
         \\
        &+N\left\vert\int_{(i-1)/N}^{i/N} \left[\Wb\left[b(\psi^N(\bm{\alpha}^N),\Zb^N\psi^N(\bm{\alpha}^N))\right]\right]_x-\left[\Wb^N\left[b(\psi^N(\bm{\alpha}^N),\Zb^N\psi^N(\bm{\alpha}^N))\right]\right]_x dx\right\vert 
        \\
        =& N \int_{(i-1)/N}^{i/N}\left\vert \int_I w(x,y)\left[b((\psi^N \bm{\alpha}^N)_y,[\Zb\psi^N(\bm{\alpha}^N)]_y)-b((\psi^N \bm{\alpha}^N)_y,[\Zb^N\psi^N(\bm{\alpha}^N)]_y) \right] dy \right\vert dx
        \\
        &+N \left\vert \int_{(i-1)/N}^{i/N}\int_I \left( w(x,y)-w^N(x,y) \right) b((\psi^N \bm{\alpha}^N)_y,[\Zb^N\psi^N(\bm{\alpha}^N)]_y) dydx \right\vert 
        \\
        \leq & \sqrt{c_z}\left\Vert \Zb\psi^N(\bm{\alpha}^N)-\Zb^N\psi^N(\bm{\alpha}^N) \right\Vert_{L^2(I)} N \int_{(i-1)/N}^{i/N} \left(\int_0^1 w(x,y)^2dy \right)^{1/2}dx 
        \\
        &+ N \left\vert \int_{(i-1)/N}^{i/N}\int_I \left( w(x,y)-w^N(x,y) \right) b(\psi^N(\bm{\alpha}^N)_y,[\Zb^N\psi^N(\bm{\alpha}^N)]_y) dydx \right\vert 
        \\
        \leq & \zeta \sqrt{c_z}\left\Vert \Zb\psi^N(\bm{\alpha}^N)-\Zb^N\psi^N(\bm{\alpha}^N) \right\Vert_{L^2(I)}  
        \\
        &+ N \left\vert \int_{(i-1)/N}^{i/N}\int_I \left( w(x,y)-w^N(x,y) \right) b(\psi^N(\bm{\alpha}^N)_y,[\Zb^N\psi^N(\bm{\alpha}^N)]_y) dydx \right\vert.
    \end{split}
\end{equation*}
From Lemma \ref{lemma:W_W_tilde}, we have $\left\Vert \Zb\psi^N(\bm{\alpha}^N)-\Zb^N\psi^N(\bm{\alpha}^N) \right\Vert_{L^2(I)} \to 0$. From Lemma \ref{lemma:A_convergence}, we have:
$$
    \sup_{i \in \{1,\dots,N \}}N \left\vert \int_{(i-1)/N}^{i/N}\int_I \left( w(x,y)-w^N(x,y) \right)\left[b(\psi^N(\bm{\alpha}^N)_y,[\Zb^N\psi^N(\bm{\alpha}^N)]_y)\right]dydx \right\vert \to 0,
$$
and we conclude.
\end{proof}

\begin{proof}[Proof of Theorem \ref{thm:epsilon_Nash}]
Since $\hat{\balpha}$ is a Nash equilibrium for the graphon game, from Proposition \ref{pr:Nash_equivalence_continuum}, we have:
\begin{equation*}
    \mathcal{J}(\hat{\alpha}_x,[\Zb\hat{\balpha}]_x)\leq \mathcal{J}(\beta,[\Zb\hat{\balpha}]_x), \quad \text{for a.e.\ }x, \forall \beta \in A,
\end{equation*}
where $\Zb$ is the aggregate operator induced by the graphon operator $\Wb$ associated to $w$, 
and thus,
\begin{equation}
\label{eq:Nash_mu_N}
\begin{split}
    &[\mu^N(\cJ(\hat{\alpha}_\cdot,[\Zb\hat{\balpha}]_\cdot)]_i -[\mu^N(\cJ(\beta,[\Zb\hat{\balpha}]_\cdot))]_i  \\
    &=N \int_{(i-1)/N}^{i/N} \Big(\cJ(\hat{\alpha}_x,[\Zb\hat{\balpha}]_x) -\cJ(\beta,[\Zb\hat{\balpha}]_x)) \Big) dx \leq 0, \quad \forall i \in \{1, \dots, N\}, \forall \beta \in A.
\end{split}
\end{equation}
Since $ \sqrt{c_z} \| \Wb \| < 1$ and $\lim_{N\to\infty}\| \Wb^N\|=\| \Wb\|$ by assumption, we can choose $N^*$ such that $ \sqrt{c_z} \| \Wb^N \| < 1$, and thus $\Zb^N$ is well defined, for $N \geq N^*$.
We compute:
\begin{equation}
\label{eq:epsilon_Nash_bound}
\begin{split}
    &\sup_{i \in \{1,\dots,N\}}\tilde\cJ([\mu^N(\hat{\balpha})]_i,\mathcal{L}(\xi_i,[\Zb_N \mu^N(\hat{\balpha})]_i))-\tilde\cJ(\beta,\mathcal{L}(\xi_i,[\Zb_N[\beta;\mu^N(\hat{\balpha})_{-i}]]_i))
    \\
    \leq &\sup_{i \in \{1,\dots,N\}} \left\vert \tilde\cJ([\mu^N(\hat{\balpha})]_i,\mathcal{L}(\xi_i,[\Zb_N \mu^N(\hat{\balpha})]_i)) - \cJ([\mu^N(\hat{\balpha})]_i,[\mu^N(\Zb^N \psi^N \mu^N(\hat{\balpha}))]_i) \right\vert \\
    & +\sup_{i \in \{1,\dots,N\}} \left\vert \cJ([\mu^N(\hat{\balpha})]_i,[\mu^N(\Zb^N \psi^N \mu^N(\hat{\balpha}))]_i) -\cJ([\mu^N(\hat{\balpha})]_i,[\mu^N(\Zb \psi^N \mu^N(\hat{\balpha}))]_i) \right\vert \\
    & +\sup_{i \in \{1,\dots,N\}} \cJ([\mu^N(\hat{\balpha})]_i,[\mu^N(\Zb \psi^N \mu^N(\hat{\balpha}))]_i) - [\mu^N(\cJ(\hat{\alpha}_\cdot,[\mu^N(\Zb \psi^N \mu^N(\hat{\balpha}))]_i)]_i \\
    & +\sup_{i \in \{1,\dots,N\}} \left\vert [\mu^N(\cJ(\hat{\alpha}_\cdot,[\mu^N(\Zb \psi^N \mu^N(\hat{\balpha}))]_i)]_i -[\mu^N(\cJ(\hat{\alpha}_\cdot,[\Zb \hat{\balpha}]_\cdot)]_i \right\vert \\
    & +\sup_{i \in \{1,\dots,N\}} [\mu^N(\cJ(\hat{\alpha}_\cdot,[\Zb \hat{\balpha}]_\cdot)]_i -[\mu^N(\cJ(\beta,[\Zb \hat{\balpha}]_\cdot))]_i \\
    & +\sup_{i \in \{1,\dots,N\}} \left\vert [\mu^N(\cJ(\beta,[\Zb \hat{\balpha}]_\cdot))]_i -\cJ(\beta,[\mu^N(\Zb \psi^N \mu^N(\hat{\balpha}))]_i) \right\vert \\
    & +\sup_{i \in \{1,\dots,N\}} \left\vert \cJ(\beta,[\mu^N(\Zb \psi^N \mu^N(\hat{\balpha}))]_i) -\cJ(\beta,[\mu^N(\Zb \psi^N [\beta;\mu^N(\hat{\balpha})_{-i}])]_i) \right\vert \\
    & +\sup_{i \in \{1,\dots,N\}} \left\vert \cJ(\beta,[\mu^N(\Zb \psi^N [\beta;\mu^N(\hat{\balpha})_{-i}])]_i) - \cJ(\beta,[\mu^N(\Zb^N \psi^N [\beta;\mu^N(\hat{\balpha})_{-i}])]_i) \right\vert \\
    & +\sup_{i \in \{1,\dots,N\}} \left\vert \cJ(\beta,[\mu^N(\Zb^N \psi^N [\beta;\mu^N(\hat{\balpha})_{-i}])]_i) -\tilde\cJ(\beta,\mathcal{L}(\xi_i,[\Zb_N[\beta;\mu^N(\hat{\balpha})_{-i}]]_i))\right\vert \\
    =:& (i) +(ii) +(iii) +(iv) +(v) +(vi) +(vii) + (viii) + (ix).
\end{split}
\end{equation}
From Assumption \ref{hyp:J_2}, the convexity of $\cJ$ in $\alpha$ implies $(iii) \leq 0$. From \eqref{eq:Nash_mu_N}, we have $(v) \leq 0$. For the remaining terms, we are going to repeatedly use Assumption \ref{hyp:tilde_J} in the same way as in the proof of Lemma \ref{lemma:convergence_controls}.

Now, consider $(i)$.
Since $\| w^N \|_2 \to \|w\|_2$ by assumption, there exists an $\tilde \epsilon$ such that $\sqrt{2c_z} (\| w \|_2 + \tilde \epsilon) <1$ and we can also choose $N^*$ such that for all $N \geq N^*$, $\left\vert \| W^N \|_F - \| w \|_2 \right\vert < \tilde \epsilon$. By Lemma \ref{lemma:W^N_and_W_N} applied to  $\bm{\alpha}^N= \mu^N(\hat{\balpha})$:
\begin{equation*}
\begin{split}
    (i) &\leq \sup_{i \in \{1,\dots,N\}}\tilde \ell_J \frac{1}{N^2}\sum_{j=1}^N (W^N_{i,j})^2 \left(\frac{2\mathbb{E}(\xi_0^2)}{1-2c_z \| W^N \|_F^2}\right) 
    \\
    & \leq  \frac{1}{N} \tilde \ell_J \tilde{\zeta} \left(\frac{2\mathbb{E}(\xi_0^2)}{1-2c_z (\| w \|_2 + \tilde \epsilon)^2}\right) \to 0.
\end{split}
\end{equation*}
Similarly, we also have $(ix) \to 0$. Next, consider $(ii)$:
\begin{equation*}
\begin{split}
    (ii) &\leq \sup_{i \in \{1,\dots,N\}}\tilde \ell_J \left\vert [\mu^N(\Zb^N \psi^N \mu^N(\hat{\balpha}))]_i - [\mu^N(\Zb \psi^N \mu^N(\hat{\balpha}))]_i \right\vert
    \\
    &= \tilde  \ell_J \sup_{i \in \{1,\dots,N\}} N \left\vert \int_{(i-1)/N}^{i/N}  \Big( [\Zb^N \psi^N \mu^N(\hat{\balpha})]_x - [\Zb \psi^N \mu^N(\hat{\balpha})]_x \Big) dx \right\vert \to 0,
\end{split}
\end{equation*}
where we used Lemma \ref{lemma:W_W_N} applied to  $\bm{\alpha}^N= \mu^N(\hat{\balpha})$. Similarly, we also have $(viii) \to 0$. Now consider, $(iv)$:
\begin{equation*}
\begin{split}
    (iv)& \leq \sup_{i \in \{1,\dots,N\}} \tilde \ell_J \left\vert [\mu^N(\Zb \psi^N \mu^N(\hat{\balpha}))]_i - [\mu^N(\Zb \hat{\balpha})]_i \right\vert \\
    &=\tilde \ell_J \sup_{i \in \{1,\dots,N\}}\left\vert [\mu^N( \Zb \psi^N \mu^N(\hat{\balpha}) -\Zb \hat{\balpha})]_i \right\vert \\
    &\leq \tilde \ell_J \sup_{i \in \{1,\dots,N\}}N \int_{(i-1)/N}^{i/N} \left\vert [\Zb \psi^N \mu^N(\hat{\balpha})]_x -[\Zb \hat{\balpha}]_x \right\vert dx \\
    & \leq \tilde \ell_J \sup_{i \in \{1,\dots,N\}}N \int_{(i-1)/N}^{i/N}  \frac{\sqrt{c_\alpha}}{1-\sqrt{c_z} \|\bW\|}\Bigl( \int_I w(x,y)^2dy\Bigr)^{1/2} \left\Vert \psi^N \mu^N(\hat{\balpha})-\hat{\balpha} \right\Vert_{L^2(I)} dx \\
    & \leq \frac{\tilde \ell_J \zeta \sqrt{c_\alpha}}{1-\sqrt{c_z} \|\bW\|} \left\Vert \psi^N \mu^N(\hat{\balpha})-\hat{\balpha} \right\Vert_{L^2(I)},
\end{split}
\end{equation*}
where we used Lemma \ref{lemma:W_Lip}. Let $\epsilon^*>0$ and let $\hat{\bbeta}$ be a continuous function approximating $\hat{\balpha}$ such that $\|\hat{\balpha} - \hat{\bbeta} \|_{L^2(I)} < \epsilon^*$. Since $\hat{\bbeta}$ is continuous, for $N$ large enough, we have $\|\hat{\bbeta} - \psi^N\mu^N(\hat{\bbeta}) \|_{L^2(I)} < \epsilon^*$. Now consider
\begin{equation*}
\begin{split}
    \|\psi^N\mu^N(\hat{\bbeta}) - \psi^N\mu^N(\hat{\balpha}) \|^2_{L^2(I)} & = \frac{1}{N} \sum_{i=1}^N ([\mu^N(\hat{\bbeta})]_i - [\mu^N(\hat{\balpha})]_i)^2 \\
    & = \frac{1}{N}\sum_{i=1}^N \left(N\int_{(i-1)/N}^{i/N}\hat{\beta}_x - \hat{\alpha}_x dx \right)^2 \\
    & \leq \frac{1}{N} \sum_{i=1}^N N\int_{(i-1)/N}^{i/N} (\hat{\beta}_x - \hat{\alpha}_x)^2 dx \\
    & = \int_I (\hat{\beta}_x - \hat{\alpha}_x)^2 dx \\
    & = \|\hat{\balpha} - \hat{\bbeta} \|^2_{L^2(I)} \\
    & \leq (\epsilon^*)^2.
\end{split}
\end{equation*}
Then we have $\left\Vert \psi^N \mu^N(\hat{\balpha})-\hat{\balpha} \right\Vert_{L^2(I)} \leq \|\hat{\balpha} - \hat{\bbeta} \|_{L^2(I)} + \|\hat{\bbeta} - \psi^N\mu^N(\hat{\bbeta}) \|_{L^2(I)} + \|\psi^N\mu^N(\hat{\bbeta}) - \psi^N\mu^N(\hat{\balpha}) \|_{L^2(I)} \leq 3 \epsilon^*$, and we conclude $(iv) \to 0$. Similarly, we also have $(vi) \to 0$. Finally, we consider $(vii)$:
\begin{equation*}
\begin{split}
    (vii)& \leq \tilde \ell_J \sup_{i \in \{1,\dots,N\}} \left\vert [\mu^N(\Zb \psi^N \mu^N(\hat{\balpha}))]_i - [\mu^N(\Zb \psi^N [\beta;\mu^N(\hat{\balpha})_{-i}])]_i \right\vert 
    \\
    & = \tilde \ell_J \sup_{i \in \{1,\dots,N\}}\left\vert [\mu^N\Big(\Zb \psi^N \mu^N(\hat{\balpha}) - \Zb \psi^N [\beta;\mu^N(\hat{\balpha})_{-i}]\Big)]_i \right\vert 
    \\
    & \leq \tilde \ell_J \sup_{i \in \{1,\dots,N\}}N \int_{(i-1)/N}^{i/N} \left\vert [\Zb \psi^N \mu^N(\hat{\balpha})]_x - [\Zb \psi^N [\beta;\mu^N(\hat{\balpha})_{-i}]]_x \right\vert dx  
    \\
    & \leq \tilde \ell_J \sup_{i \in \{1,\dots,N\}}N \int_{(i-1)/N}^{i/N}  \frac{\sqrt{c_\alpha}}{1-\sqrt{c_z} \|\bW\|}\Bigl( \int_I w(x,y)^2dy\Bigr)^{1/2} \left\Vert \psi^N \mu^N(\hat{\balpha}) - \psi^N [\beta;\mu^N(\hat{\balpha})_{-i}] \right\Vert_{L^2(I)} dx 
    \\
    & \leq  \frac{\tilde \ell_J \zeta \sqrt{c_\alpha}}{1-\sqrt{c_z} \|\bW\|} \cdot \sup_{i \in \{1,\dots,N\}} \frac{1}{\sqrt{N}} \left\vert \beta - [\mu^N(\hat{\balpha})]_i \right\vert^{1/2} 
    \\
    & \leq  \frac{\tilde \ell_J \zeta \sqrt{c_\alpha\text{diam}(A)} }{1-\sqrt{c_z} \|\bW\|} \frac{1}{\sqrt{N}} \to 0,
\end{split}
\end{equation*}
where we used Lemma \ref{lemma:W_Lip} and the boundedness of $A$. Thus, the left hand side of \eqref{eq:epsilon_Nash_bound} is bounded by an $\epsilon_N$ with $\epsilon_N \to 0$ as $N \to \infty$.
\end{proof}

\begin{remark}
Let us stress the significance of Propositions \ref{pr:link_constant_MFG}, \ref{pr:link_MFG}, and \ref{pr:link_K_MFG} in conjunction with Theorems \ref{thm:convergence} and \ref{thm:epsilon_Nash}. For a sequence of $N$-player games with wildly asymmetric players, as long as the limiting graphon $w$ is a constant connection strength graphon it is sufficient to consider the analogous MFG and only consider a single representative player to construct an $\epsilon$-Nash equilibrium for the $N$-player game, while if the limiting graphon $w$ is a piecewise constant graphon it is sufficient to consider the analogous $K$-population MFG and only consider a single representative player from each population to construct an $\epsilon$-Nash equilibrium for the $N$-player game.
\end{remark}

\section{Revisiting the Motivating Applications Introduced Earlier} 
\label{sec:revisiting_examples}

\subsection{Where do I put my towel on the beach?}
\label{sub:revising_beach}
For the sake of illustration, we assume that the state space (i.e. the \emph{beach}) and the set of actions, namely the locations where the players can put their towels are the real line $\RR$.

\vskip 6pt
We first treat the case of a population of $N$ individuals, and we prove existence of Nash equilibria in pure strategies. Each individual $i \in \{1,\dots,N\}$ chooses a location $\alpha_i \in A=\RR$ on the beach and hands their towel to a beach attendant after telling the location of their choice, and since the attendant does not want to place towels directly on top of each other, the actual position of the towel ends up being a noisy version of their choice. The noise is represented by a mean-zero real-valued random variable with distribution $\mu_0$. In order to use notations similar to those of the paper we write:
$$
    b(\alpha,z) = \alpha
    \qquad\text{and}\qquad
    X_{\alpha,\xi}=\alpha+\xi.
$$
Let us assume that the location of the concession stand is $1$, and that players $i$ and $j$ \emph{interact} with weight 
$W_{i,j}\in \RR$. For example, one could think of $W_{i,j}=1$ if individuals $i$ and $j$ are friends, $W_{i,j}=0$ otherwise, and $W_{i,i}=0$ for all $i=1,\dots,N$. We denote by $W$ the $N\times N$ matrix of the $W_{i,j}$. So player $i$ could choose their location in order to minimize the cost:
\begin{equation*}
    J_i(\bm{\alpha})=\mathbb{E}\Bigl[(\alpha_i)^2+(X_{\alpha_i,\xi_i}-1)^2+ \Bigl(X_{\alpha_i,\xi_i}-\frac{1}{N} \sum_{j=1}^N W_{i,j} X_{\alpha_j,\xi_j}\Bigr)^2 \Bigr],
\end{equation*}
corresponding to the function:
$$
f(x,\alpha,z) = \alpha^2 + (x-1)^2 + (x - z)^2.
$$
We could specify the relative importance of the three components of the cost by including coefficients. We set them to $1$ for the sake of illustration. Using the notation of Section \ref{sec:static_N} and the fact that $\int \xi\,\nu(d\xi,dz)=0$, we have:
\begin{equation*}
\begin{split}
    \tilde\cJ(\alpha,\nu)&=3\alpha^2-2\alpha\Bigl(1+ \int_{\mathbb{R} \times \mathbb{R}}z \nu(d\xi,dz)\Bigr)
    +\int_{\mathbb{R} \times \mathbb{R}}[(\xi-1)^2+(\xi-z)^2] \nu(d\xi,dz).
\end{split}
\end{equation*}
For a given $\nu$, this function is minimized by:
$$
\hat\alpha=\frac13\Bigl(1+ \int_{\mathbb{R} \times \mathbb{R}}z \nu(d\xi,dz)\Bigr).
$$
Given the special form of the function $b$, the aggregate is explicitly given by:
$$
[\bZ\balpha]_i=\frac1N\sum_{j=1}^N W_{i,j}(\alpha_j + \xi_j ).
$$
Consequently, a strategy profile $\hat\balpha$ is a Nash equilibrium if and only if
   \begin{equation*}
   \begin{split}
   \hat{\alpha}_i&=\frac{1}{3}\Bigl(1+ \EE \Bigl[ \bigl[\bZ\hat\balpha]_i\bigr]\Bigr]\Bigr), \\
   &=\frac{1}{3}\Bigl(1+  \frac{1}{N}\sum_{j=1}^N W_{i,j}\hat{\alpha}_j \Bigr), \qquad i=1,\cdots,N
   \end{split}
    \end{equation*}
which can be rewritten in matrix form as:
$$
\hat\balpha=\frac13\bigl(\mathbf{1}_N+\frac1N W \hat\balpha\bigr),
$$
where $\mathbf{1}_N$ denotes the $N$-vector of all ones. So for this explicit model we proved the following:

\begin{proposition}
If $3N$ is not an eigenvalue of the matrix $W=[W_{i,j}]_{i,j=1,\cdots,N}$ then the unique Nash equilibrium in pure strategies is given by:
\begin{equation} 
\label{eq:explicit_particular_N}
 \hat\balpha=\frac13[\bI_N-\frac1{3N} W]^{-1}\mathbf{1}_N
 \end{equation}
where $\bI_N$ is the $N \times N$ identity matrix and $\mathbf{1}_N$ is the $N$-dimensional vector of all ones.
\end{proposition}
\begin{remark}
The above result shows that the strategy profile of the unique Nash equilibrium of the game is up to the additive constant $1$, the Katz centrality measure of the interaction graph of the players. See \cite{katz1953new}.
\end{remark}

\begin{remark}
    Note that the assumption is automatically satisfied if all the entries $W_{i,j}$ are in the interval $[0,1]$. Indeed, in this case we have:
 \begin{equation*}
\begin{split}
\|W\|=\max_{\{\balpha \in A^N: \|\balpha\|_2=1\}}\sqrt{\sum_{i=1}^N \left(\sum_{j=1}^N W_{i,j}\alpha_j \right)^2}
    &\leq \max_{\{\balpha \in A^N: \|\balpha\|_2=1\}}\sqrt{\sum_{i=1}^N \sum_{j=1}^N (W_{i,j})^2 \sum_{k=1}^N (\alpha_k)^2}\\
    &=\sqrt{\sum_{i=1}^N \sum_{j=1}^N (W_{i,j})^2}\leq N,
\end{split}
\end{equation*}
implying that $3N$ cannot be an eigenvalue since the absolute value of an eigenvalue has to be smaller than $\|W\|$.  
\end{remark}

\vskip 6pt
We now consider the game model with a continuum of players. As before we assume $A=\RR$. The state of player $x\in I$ when all the players use the controls given by $\balpha\in \mathbb{A}$ is given by:
\begin{equation*}
    X_{\alpha_x,\xi_x}(\omega)=\alpha_x+\xi_x(\omega),
\end{equation*}
and they face the cost:
\begin{equation*}
    J_{x}(\alpha,\balpha):=\mathbb{E} \left[\alpha^2+(X_{\alpha,\xi_0}-1)^2+ \left(X_{\alpha,\xi_0}-\int_I  w(x,y) X_{\alpha_y,\xi_y} \lambda(dy) \right)^2 \right]
\end{equation*}
as given by the same cost function as before. Note that in this form of the cost, we need to compute the aggregate using the measure $\lambda(dy)$ appearing in the rich Fubini extension. However, after developing the square and computing the expectation, we see that the computation of the cost only involves the classical Lebesgue measure $\lambda_I(dy)=dy$.
As before, for the sake of simplicity, we set to one the coefficients of the three terms in the cost. Using the notation of the paper, the cost is given by the function:
\begin{equation*}
    \mathcal{J}(\alpha,z)=3\alpha^2+z^2-2\alpha(1+ z)+1+2\EE[\xi^2_0].
\end{equation*}
For $z\in\RR$ fixed, this function is minimized for $\hat\alpha=\frac13(1+z)$ and $\hat\balpha=(\hat\alpha_x)_{x\in I}$
is a Nash equilibrium for the graphon game if and only if
\begin{equation*}
    \hat{\alpha}_x=\frac{1}{3}\Bigl(1+\int_I w(x,y)\hat{\alpha}_y dy\Bigr), \qquad \lambda_I-a.e.  x \in I.
\end{equation*}
Writing this condition in terms of the graphon operator we get:
\begin{equation*}
    \hat{\balpha}=\frac{1}{3}\Bigl(\mathbf{1}+\bW\hat{\balpha}\Bigr)
 \end{equation*}
which gives
\begin{equation*}
    \hat{\balpha}=\frac{1}{3}[\bI-\frac{1}{3}\bW]^{-1}\mathbf{1},
\end{equation*}
provided the above inverse operator exists. So we proved that as in the case of finitely many players, the strategy profile of the unique Nash equilibrium is given by the Katz centrality measure of the graphon \cite{katz1953new}.

\begin{proposition}
\label{pr:particular_continuum}
If the graphon $w$ is such that $\|\bW\|<3$,  the unique Nash equilibrium $\hat{\balpha}$ for the graphon game is given by:
    \begin{equation} 
    \label{eq:explicit_particular_continuum}
        \hat{\balpha}=\frac{1}{3}\left(\bI-\frac{1}{3}\Wb\right)^{-1}\mathbf{1}.
    \end{equation}
\end{proposition}

\begin{remark}
Consider a sequence of symmetric matrices $W^N$, $N\geq 1$, and let $\balpha^{N,*}$ be the vector of strategies given by equation \eqref{eq:explicit_particular_N} with the matrix $W^N$. Then we have $\psi^N(\balpha^{N,*}) =\frac{1}{3}\left(\bI-\frac{1}{3}\Wb^N\right)^{-1}\mathbf{1}$, where $\Wb^N$ is the graphon operator corresponding to the evenly spaced piecewise constant graphon $w^N:=\Psi^N(W^N)$. Then we notice that the conclusion of Theorem \ref{thm:convergence}, namely $\lim_{N\to\infty}d_{S}(\psi^N(\balpha^{N,*}),\hat{\balpha}) = 0$, holds under the weaker assumption that the resolvent, $\left(\bI-\frac{1}{3}\Wb^N\right)^{-1}$, converges to $\left(\bI-\frac{1}{3}\Wb\right)^{-1}$ for some graphon $w$.
\end{remark}

We now proceed to the computation of the Nash equilibrium $\balpha$ for some of the most commonly used graphon models.
The form \eqref{eq:explicit_particular_continuum} of the equilibrium strategy requires the computation of the resolvent $[\bI-\theta\bW]^{-1}\phi$ for a  real number $\theta$ such that $\|\theta\bW\|<1$ and a function $\phi$ which, in most of the actual applications we consider, will be the constant function $\mathbf{1}$. We shall use two different methods to compute this resolvent.
\begin{enumerate}
\item Since $\|\theta\bW\|<1$, the resolvent operator is given by its convergent Taylor expansion, namely:
\begin{equation}
\label{fo:Taylor_expansion}
[\bI -\theta \Wb]^{-1} = \bI + \theta \Wb + \theta^2 \Wb^2 + \cdots\;.
\end{equation}
\item Since $\bW$ is a symmetric Hilbert-Schmidt operator, there exists a complete orthonormal basis of $L^2(I)$, say $\{\phi_k\}_{k\ge 1}$ and a square summable sequence of real numbers $\{\lambda_k\}_{k\ge 1}$ such that $\bW\phi_k=\lambda_k\phi_k$ for $k\ge 1$. Recall that $L^2(I)$ is the classical Lebesgue space over the unit interval. Accordingly, for all $\phi\in L^2(I)$:
\begin{equation*}
[\bI-\theta\bW]^{-1}\phi=\sum_{k\ge 1}\frac{<\phi,\phi_k>}{1-\theta\lambda_k}\phi_k.
\end{equation*}
This eigenvalue expansion of the resolvent will be useful when we can identify the eigenvalues and the orthonormal basis of eigenfunctions via explicit formulas.
\end{enumerate}
\begin{itemize}
\item \emph{Constant Connection Strength Graphon}: for this graphon model,  $\int_I w(x,y)dy=a$ for some $a\in\RR$ and all $x \in I$.
Notice that in this case, $\bW\mathbf{1}=a\mathbf{1}$, so that $\bW^n\mathbf{1}=a^n\mathbf{1}$ for each integer $n\ge 1$. Hence the Taylor expansion \eqref{fo:Taylor_expansion} gives:
\begin{equation*}
\begin{split}
[\bI - \frac{1}{3} \Wb]^{-1} \mathbf{1}&=\mathbf{1}+\frac{a}3\mathbf{1}+\bigl(\frac{a}3\bigr)^2\mathbf{1}+\bigl(\frac{a}3\bigr)^3\mathbf{1} + \cdots\\
&=\frac3{3-a}\mathbf{1}.
\end{split}
\end{equation*}
showing that the Nash equilibrium is given by:
$$
\hat\balpha=\frac1{3-a}\mathbf{1}.
$$
\item \emph{Piecewise Constant Graphon}: Here we consider a graphon of the form $w(x,y)=a_1\cdot 1_{\{x,y \leq x^*\}}+a_2 \cdot 1_{\{x \leq x^*,y > x^*\}}+a_2\cdot 1_{\{x > x^*,y \leq x^*\}} + a_3 \cdot 1_{\{x,y > x^*\}}$ for some constants $x^* \in I, a_1,a_2,a_3 \in \mathbb{R}$. From Proposition \ref{pr:link_K_MFG}, we expect the unique Nash equilibrium to have the form $\hat{\alpha}_x= c_1 \cdot 1_{\{x \leq x^*\}}+c_2 \cdot 1_{\{x > x^*\}}$. By plugging this control profile directly in \eqref{eq:explicit_particular_continuum}, we find that $c_1$ and $c_2$ must solve the system:
\begin{equation*}
    \begin{split}
         (3-a_1x^*)c_1-a_2(1-x^*)c_2&=1, \\
         -a_2x^*c_1+(3-a_3(1-x^*))c_2&=1,
    \end{split}
\end{equation*}
which has a unique solution when $(3-a_1x^*)(3-a_3(1-x^*))-|a_2|^2(1-x^*) x^* \neq 0$.
 
\item \emph{Power-Law Graphon}: Here we consider the power-law graphon introduced by Medvedev and Tang \cite{medvedev2018kuramoto}, which was used to study synchronization of coupled oscillators. It is given by the formula $w(x,y)=(xy)^{-\gamma}$ for some $\gamma \in (0,\frac{1}{3})$. 
Note that this function $w$ is symmetric and square-integrable, in accordance with our definition of a graphon. Since we consider $\gamma \in (0,\frac{1}{3})$, we have:
\begin{equation*}
    \begin{split}
        \|\bW\|&\leq \sqrt{\int_I \int_I(w(x,y))^2dy  dx} \\
        & =\sqrt{\int_I \int_I(x y)^{-2\gamma}dy  dx}\\
        & = \frac{1}{1-2\gamma} \\
        & <3,
    \end{split}
\end{equation*}
showing that the assumption of Proposition \ref{pr:particular_continuum} is satisfied.
In the present situation 
$$ 
[\Wb \mathbf{1}]_x = \int_I x^{-\gamma} y^{-\gamma} d y = x^{-\gamma} \int_I y^{-\gamma} d y = \frac{1}{1-\gamma} x^{-\gamma},  
$$
from which it is easy to see by induction that:
$$
[\Wb^n \mathbf{1}]_x = \frac{1}{(1 - \gamma)(1 - 2 \gamma)^{n-1}} x^{-\gamma},
$$
for $n\ge 1$. Indeed, if this formula is true for $n\ge 2$, then:
\begin{equation*}
\begin{split} 
[\Wb^{n+1}\mathbf{1} ]_x = [\Wb [\Wb^{n} \mathbf{1}] ]_x &=  \int_I x^{-\gamma}y^{-\gamma}\frac{1}{(1 - \gamma)(1 - 2 \gamma)^{n-1}}y^{-\gamma}\;dy\\
&=\frac{1}{(1 - \gamma)(1 - 2 \gamma)^{n-1}}x^{-\gamma}\int_I y^{-2\gamma}\;dy\\
&=\frac{1}{(1 - \gamma)(1 - 2 \gamma)^{n}}x^{-\gamma}. 
\end{split}
\end{equation*} 
Plugging this in the Taylor expansion \eqref{fo:Taylor_expansion} we get:
\begin{equation*}
\bigl[[\bI -\theta \Wb ]^{-1} \mathbf{1}\bigr]_x =1+\theta\frac{x^{-\gamma}}{1-\gamma}\Bigl(1+\frac{\theta}{1-2\gamma}+\frac{\theta^2}{(1-2\gamma)^2}+\cdots\Bigr)
=1+\frac{\theta(1-2\gamma)}{(1-\gamma)(1-2\gamma-\theta)}x^{-\gamma}
\end{equation*}
from which we conclude that the unique Nash equilibrium is given by the formula:
$$ 
\hat{\alpha}_x =\frac13 +\frac{1-2\gamma}{6(1-\gamma)(1-3\gamma)}x^{-\gamma}. 
$$
This is an example in which, in equilibrium, different players take different actions, though depending upon a power of their label.

\item \emph{Min-Max Graphon}: If $w(x,y)=\min(x,y) [1-\max(x,y)]$, it is known that an orthonormal basis of eigenfunctions is given by
$\phi_k(\cdot) = \sqrt{2} \sin(\pi k \cdot)$ with corresponding eigenvalues $\lambda_k = \frac{1}{\pi^2 k^2}$ for $k \ge1$. See
\cite{avella2018centrality}. Accordingly,
$$
[\bI-\frac13\bW]^{-1}\mathbf{1}=\sum_{k\ge 1}\frac{<\mathbf{1},\phi_k>}{1-\lambda_k/3}\phi_k.
$$
Notice that
$$
<\mathbf{1},\phi_k>=\int_I\phi_k(x)dx=\sqrt{2}\int_I\sin (k\pi x)\;dx=
\begin{cases}
\frac{2\sqrt{2}}{(2m+1)\pi}&if \;\;k=2m+1\\
0& if \;\;k=2m.
\end{cases}
$$
Therefore,
$$
\bigl[[\bI-\frac13\bW]^{-1}\mathbf{1}\bigr]_x=12\pi\sum_{m\ge 0}\frac{2m+1}{3\pi^2(2m+1)^2-1} \sin(\pi (2m+1)x)
$$
which yields
\begin{equation*}
    \hat{\alpha}_x = 4 \pi \displaystyle\sum_{k \: \mathrm{odd}}  \frac{k}{3 \pi^2 k^2 - 1} \sin\left(\pi k x\right).
\end{equation*}

\item \emph{Simple Threshold Graphon}: If $w(x,y)=1_{\{x+y \leq 1\}}$, then 
the unique Nash equilibrium $\hat{\balpha}$ solves:
\begin{equation}\label{eq:simple_threshold}
    3\hat{\alpha}_x-\int_0^{1-x}\hat{\alpha}_y dy =1.
\end{equation}
We make an ansatz that $\hat{\balpha}$ takes the following form:
\[ \hat{\alpha}_x = c \left[ \cos\left( \frac{1-x}{3} \right) - \sin\left(\frac{x}{3} \right) \right], \]
and we compute:
\[ \int_0^{1-x} \hat{\alpha}_y dy = 3c \left[\left( \cos\left( \frac{1-x}{3} \right) - \sin\left(\frac{x}{3} \right)  \right) + \left( \sin\left( \frac{1}{3} \right) - 1 \right) \right]. \]
Thus, one can check that equation \eqref{eq:simple_threshold} is satisfied for the above strategy profile when $c = \frac{1}{3 \left[1 -  \sin\left(\frac{1}{3}\right)\right]}$.
\end{itemize}

\subsection{Cities Game}
\label{sub:cities_game}
Here we directly consider the game with a continuum of players. In this game model, each player $x\in I$ uses control $\alpha_{x}$ and the frequency they actually go to the city center is given by:
\begin{equation*}
X_{\alpha_x,\xi_x} = \alpha_{x} + \xi_{x},\qquad x\in I,
\end{equation*}
and they seek to minimize the cost:
\begin{equation*}
J_{x}({\balpha}) = \mathbb{E} \left[ \frac{1}{2} \alpha_{x}^2 -k \alpha_{x} - \theta \alpha_{x} \int_I w(x,y) X_{\alpha_y,\xi_y} d y  \right].
\end{equation*}
So using the notation of the paper, we have:
$$
b(\alpha,z)=\alpha, 
\qquad \text{and}\qquad
f(x,\alpha,z)=\frac{1}{2}\alpha^2-k\alpha-\theta \alpha z,
$$
where $k \alpha$ represents the intrinsic benefit of effort level $\alpha$, $\frac{1}{2} \alpha^2$ is the cost of taking the action, and $\theta$ measures the relative impact of the complementary effects of social interactions for $k, \theta>0$, $z$ being the aggregate frequency of the city visits by all the individuals.

\begin{proposition} 
\label{pr:cityNE} 
If $\theta \|\bW\|<1$, the unique Nash equilibrium $\hat{\balpha}$ for the city graphon game is given by 
\begin{equation*}
\hat{\balpha} =k[ \bI - \theta \Wb ]^{-1}\mathbf{1}. 
\end{equation*}
\end{proposition}

\begin{proof} In this model,
$$
\mathcal{J}(\alpha,z) =  \frac{1}{2} \alpha^2 -k \alpha - \theta \alpha z, $$
whose minimum in $\alpha$ is attained for $\alpha=k+\theta z$.
So $\balpha$ is a Nash equilibrium if and only if
$$
\hat\alpha_x=k+\theta[\bZ\balpha]_x=k+\theta\int_Iw(x,y)\hat\alpha_y \;dy,\qquad \lambda_I \; a.e.\ x\in I,
$$
so using the definition of $\Zb$, we have:
$$
\hat{\balpha} - \theta \Wb \hat{\balpha} = k \mathbf{1}, 
$$
which proves the desired result since $\bI - \theta \Wb $ is invertible by assumption. 
\end{proof}

\subsubsection*{The Price of Anarchy}
Notice that:
\begin{equation*}
    \begin{split}
        \mathcal{S}(\hat{\balpha})&=\int_I\mathcal{J}\left(\hat{\alpha}_x,\frac{\hat{\alpha}_x-k}{\theta}\right)dx \\
        &=\int_I \Bigl(\frac{1}{2}\hat{\alpha}_x^2-k \hat{\alpha}_x-\theta \hat{\alpha}_x \cdot \frac{\hat{\alpha}_x-k}{\theta}\Bigr)dx \\
        &=-\frac{1}{2} \int_I \hat{\alpha}^2_x dx,
    \end{split}
\end{equation*}
from which we deduce that the social cost for the unique Nash equilibrium for the city graphon game is given by:
\begin{equation*}
    \mathcal{S}(\hat{\balpha})=-\frac{k^2}{2} \Big\Vert [I - \theta \Wb ]^{-1} \mathbf{1} \Big\Vert^2_{L^2(I)}
    =-\frac{k^2}{2} < [I - \theta \Wb ]^{-2} \mathbf{1},\mathbf{1}>.
\end{equation*}
The form of the model allows us to investigate 
the central planner optimization of the social cost as explained in Section \ref{sub:graphon_control}.
Note that for any admissible strategy profile $\balpha$ we have:
\begin{eqnarray} 
\label{eq:citysocialwelfare} 
\mathcal{S}\left(\balpha \right) = \int_I J_{x}({\balpha}) dx& =& \int_I \left(\frac{1}{2} \alpha_{x}^2 - k \alpha_{x} \right) dx - \theta \int_I \int_I \alpha_{x} w(x,y)\alpha_y dy dx\nonumber\\
&=&\frac12\|\balpha\|^2 - k<\balpha,\mathbf{1}>-\theta<\balpha,\bW\balpha>\nonumber\\
&=&\frac12<\balpha,[\bI-2\theta\bW]\balpha -2k\mathbf{1}>
\end{eqnarray}
where we used the notation $<\cdot,\cdot>$ for the $L^2(I)$ inner product. We have:
$$
\frac{d \cS(\balpha)}{d\balpha}=[\bI-2\theta\bW]\balpha-k\mathbf{1}
$$
and if $2\theta\|\bW\|<1$, 
$$
\balpha^O=k[\bI-2\theta\bW]^{-1}\mathbf{1}
$$ 
is the argument of the minimum and plugging this value in \eqref{eq:citysocialwelfare} we find that the optimal social cost is given by:
\begin{equation*}
    \mathcal{S}(\balpha^O)=-\frac{k^2}{2}< [I - 2\theta \Wb ]^{-1} \mathbf{1},\mathbf{1}>
\end{equation*}
and the price of anarchy (PoA) is given by:
\begin{equation*}
    \text{PoA}=\frac{< [I - \theta \Wb ]^{-2} \mathbf{1},\mathbf{1}>}{< [I - 2\theta \Wb ]^{-1} \mathbf{1},\mathbf{1}>}.
\end{equation*}
Note that the numerator will always be greater than or equal to the denominator, so if both quantities are positive, PoA will be a real number greater than or equal to $1$.
 
For the sake of illustration, we compute the PoA for a few graphons to further understand how the PoA for the graphon cities game can depend on the network connectivity. First, we consider a constant connection strength graphon as introduced in Section \ref{sec:constantconnection}, with strength $a$. Because $[\bW \mathbf{1}]_x = a$, we can find inductively that $[\bW^n \mathbf{1}]_x = a^n$. When $\theta a < 1$, the Nash equilibrium and its associated social cost are 
\[ \hat{\balpha} = \left( \frac{1}{1 - \theta a} \right) 1_x \: \: \mathrm{and} \: \: \mathcal{S}(\hat{\balpha}) = -\frac{k^2}{2} \left( \frac{1}{1 - \theta a} \right)^2.  \]
When $2 \theta a < 1$, the social optimum and its associated social cost are 
\[ \balpha^O = \left( \frac{1}{1 -  2 \theta a} \right) 1_x \: \: \mathrm{and} \: \: \mathcal{S}(\balpha^O) = -\frac{k^2}{2} \left( \frac{1}{1 - 2 \theta a} \right).  \]
Then one can compute that the price of anarchy for constant connection strength graphons when $2 \theta a < 1$ is given by 
\[ \mathrm{PoA} = \frac{1 - 2 \theta a}{\left(1 - \theta a \right)^2} = 1 - \frac{\theta^2 a^2}{\left(1 - \theta a \right)^2}. \]
We note that $\text{PoA}  = 1$ when $\theta a = 0$, which is the case when either individuals do not have network connections or their cost functions do not depend on interactions. We also note that $\text{PoA}  \to 0$ as $\theta a \to \tfrac{1}{2}$,
since the social optimum will feature larger and larger controls, and $\mathcal{S}(\balpha^O)$ blows up as a result. Furthermore, when $2 \theta \alpha < 1$, we find that 
\[ \frac{\partial}{\partial a} \mathrm{PoA} = \frac{\partial}{\partial \theta} \mathrm{PoA} = \left( \frac{2 \theta^2 a}{\left( 1 - 2 \theta a \right)^2} \right) \left[2 \theta a - 1 \right] < 0  \]
so the PoA is decreasing in both the benefit of social interactions parametrized by $\theta$ and the connection strength of the graphon, $a$.

Now we turn to graphons with heterogeneous degree distributions like the power-law or simple threshold graphons, to explore how the social costs and PoA depend upon the strength and distribution of connection strength. First, we consider the power-law graphon $w(x,y) = (x y)^{-\gamma}$. We saw earlier that for any integer $n\ge 1$, 
$$
[\bW^n\mathbf{1}]_x=\frac1{(1-\gamma)(1-2\gamma)^{n-1}}x^{-\gamma},
$$
which could be injected into the Taylor expansion of the resolvent to give:
$$
\bigl[ [I-2\theta\bW]^{-1}\mathbf{1}\bigr]_x=1+\frac{2\theta(1-2\gamma)}{(1-\gamma)(1-2\gamma-2\theta)}x^{-\gamma}
$$
and after integration over the unit interval we find:
$$
\cS(\balpha^O)=-\frac{k^2}{2}\Bigl(1+\frac{2\theta(1-2\gamma)}{(1-\gamma)^2(1-2\gamma-2\theta)}\Bigr).
$$
Using now the Taylor expansion of $[\bI-\theta\bW]^{-2}$ we compute:
\begin{eqnarray*}
\cS(\hat\balpha)&=&-\frac{k^2}{2}<[\bI-\theta\bW]^{-2}\mathbf{1},\mathbf{1}>\\
&=&-\frac{k^2}{2}<\sum_{n\ge 0}(n+1)\theta^n\bW^n\mathbf{1},\mathbf{1}>\\
&=&-\frac{k^2}{2}\bigl(1+<\sum_{n\ge 1}(n+1)\theta^n\bW^n\mathbf{1},\mathbf{1}>\bigr)\\
&=&-\frac{k^2}{2}\Bigl(1+\sum_{n\ge 1}(n+1)\theta^n\frac1{(1-\gamma)^2(1-2\gamma)^{n-1}}\Bigr)\\
&=&-\frac{k^2}{2}\Bigl(1+\frac{1-2\gamma}{(1-\gamma)^2}\sum_{n\ge 1}(n+1)\Bigl(\frac{\theta}{1-2\gamma}\Bigr)^n\Bigr)\\
&=&-\frac{k^2}{2}\Bigl(1+\frac{\theta(1-2\gamma)(2-4\gamma-\theta)}{(1-\gamma)^2(1-2\gamma-\theta)^2}\Bigr)
\end{eqnarray*}
and finally
$$
\text{PoA} =\frac{1+\frac{\theta(1-2\gamma)(2-4\gamma-\theta)}{(1-\gamma)^2(1-2\gamma-\theta)^2}}{1+\frac{2\theta(1-2\gamma)}{(1-\gamma)^2(1-2\gamma-2\theta)}}.
$$
We illustrate how the PoA varies with the parameters $\theta$ and $\gamma$ in Figure \ref{fig:powerlawPoA}.

 \begin{figure}
 \centering
 \hspace{-5mm}
 \includegraphics[width = 0.36\textwidth, height = 0.25\textheight]{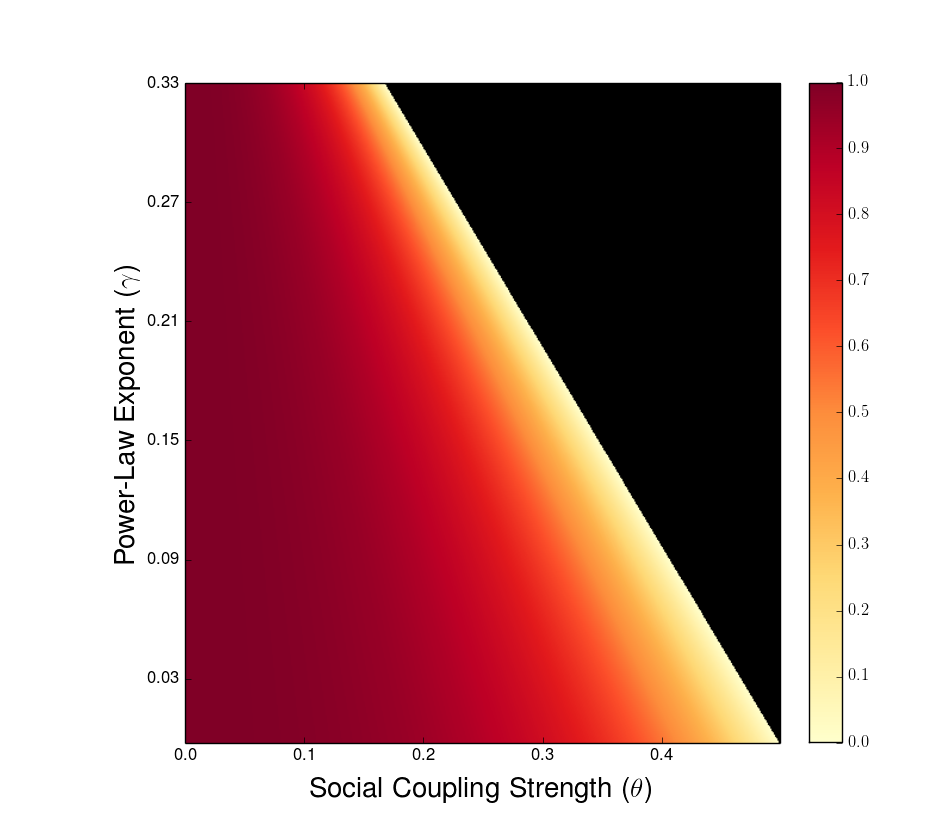}
 \hspace{-5mm}	
 \includegraphics[width=0.34\textwidth,height = 0.25\textheight]{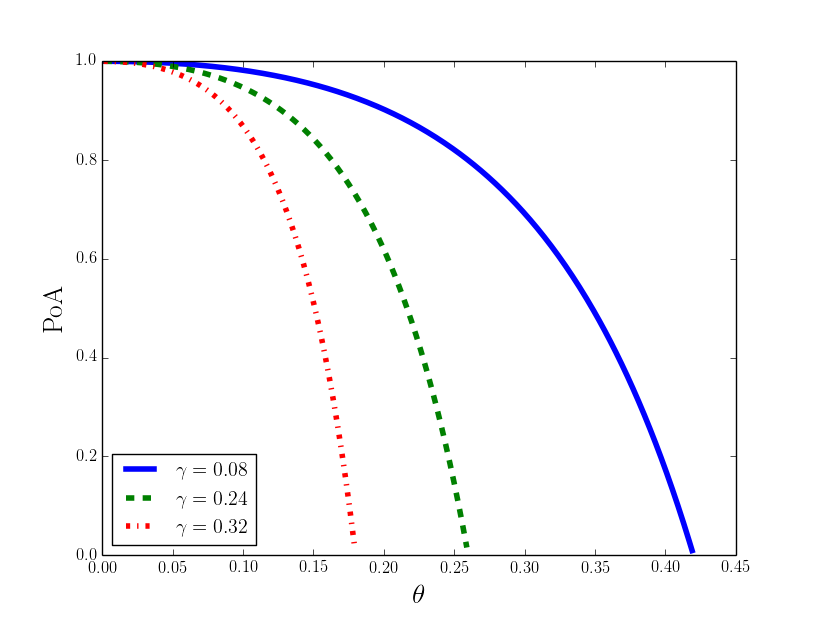}
 \hspace{-5mm}
 \includegraphics[width = 0.34\textwidth,height = 0.25\textheight]{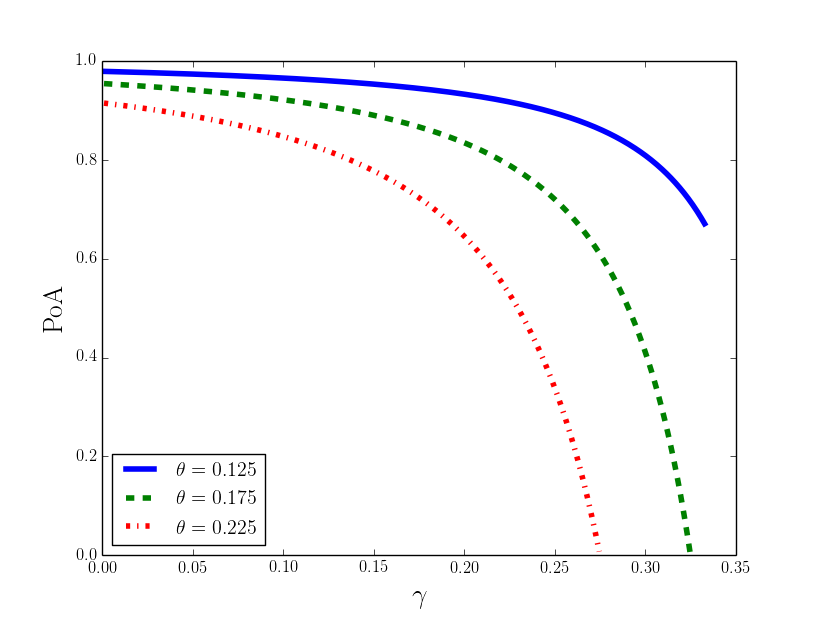}
 	
	\caption{Heatmap of PoA of cities game on power-law graphon for varying values of $\theta$ and $\gamma$. Region in black corresponds to cases in which $1 - 2 \gamma - 2 \theta < 0$ and the social optimum is not defined. }
	 \label{fig:powerlawPoA}
 \end{figure}

Next, consider a modification of the power-law graphon: $w(x,y) = g(\gamma) x^{-\gamma} y^{-\gamma}$, where $g(\gamma)$ can be used to normalize the connection strength across different power-law exponents $\gamma$. For this graphon, the strength of connection for an individual with index $x \in I$ and the average connection strength are respectively given by 
\begin{subequations} \label{eq:normalizepowerlaw} \begin{align} & \int_0^1 w(x,y) dy = \left( \frac{g(\gamma)}{1 - \gamma} \right) x^{-\gamma}, \\ &\int_0^1 \int_0^1 w(x,y) dx dy = \frac{g(\gamma)}{\left(1 - \gamma\right)^2}. \end{align} \end{subequations}
In particular, the choice of normalization $g(\gamma) = \left(1 - \gamma \right)^2$ yields an average connection strength of \\ $\int_0^1 \int_0^1 w(x,y) dx dy = 1$ for any $\gamma$, even though these connections are still distributed less uniformly for larger values of $\gamma$. With this normalization, we can study how the social cost and PoA can vary with the heterogeneity parameter $\gamma$ even though the total strength of connections is fixed amongst the population of players. 
For the generalized power-law graphon, we have the following Nash equilibrium and associated social cost 
\[\hat{\balpha} = 1 + \left( \frac{\theta g(\gamma)}{1 - \gamma} \right)\left(\frac{1 - 2 \gamma}{ 1 - 2 \gamma - \theta g(\gamma) } \right) x^{-\gamma} \: \: \mathrm{and} \: \: \mathcal{S}(\hat{\balpha}) = 1 + \frac{g(\gamma) \theta \left(1 - 2 \gamma \right) \left( 2 - 4 \gamma - \theta g(\gamma)\right)}{\left(1 - \gamma\right)^2 \left(1 - 2 \gamma - \theta g(\gamma) \right)^2},\]
provided that $1 - 2 \gamma - \theta g(\gamma) > 0$. If we further have that $1 - 2 \gamma - 2 \theta g(\gamma) < 0$, then we have a social optimum with associated social cost 
\[\balpha^O = 1 + \left( \frac{2 \theta g(\gamma)}{1 - \gamma} \right)\left(\frac{1 - 2 \gamma}{ 1 - 2 \gamma - 2 \theta g(\gamma) } \right) x^{-\gamma} \: \: \mathrm{and} \: \: \mathcal{S}(\hat{\balpha}) = 1 + \left( \frac{2 \theta g(\gamma)}{\left( 1 - \gamma \right)^2} \right) \left( \frac{1 - 2 \gamma}{1 - 2 \gamma - 2 \theta g(\gamma)} \right).\]
Using these social costs, we can compute the price of anarchy as 
\[ \text{PoA} = \left( 1+\frac{\theta(1-2\gamma)(2-4\gamma-\theta g(\gamma))}{(1-\gamma)^2(1-2\gamma-\theta g(\gamma))^2} \right) \bigg/ \left(1+\frac{2\theta g(\gamma) (1-2\gamma)}{(1-\gamma)^2(1-2\gamma-2\theta g(\gamma))} \right). \]
We can see the impact of $\gamma$ and $\theta$ on the PoA in Figures \ref{fig:powerlawPoAnormalized1} and \ref{fig:powerlawPoAnormalized2}. In particular, we notice that the social cost of both the Nash equilibrium $\mathcal{S}(\hat{\balpha})$ and the social optimum $\mathcal{S}(\balpha^O)$ are decreasing in $\gamma$, so more uneven distributions of social contacts improves the collective cost of the population of players.

However, the social cost of the social optimum decreases much more rapidly than the social cost of the Nash equilibrium, so we see that the PoA is also decreasing in $\gamma$, and the relative efficiency of the Nash equilibrium decreases as connectivity follows a steeper power law

 \begin{figure}[h!]
 \centering
 \hspace{-5mm}
 	\includegraphics[width = 0.48\textwidth]{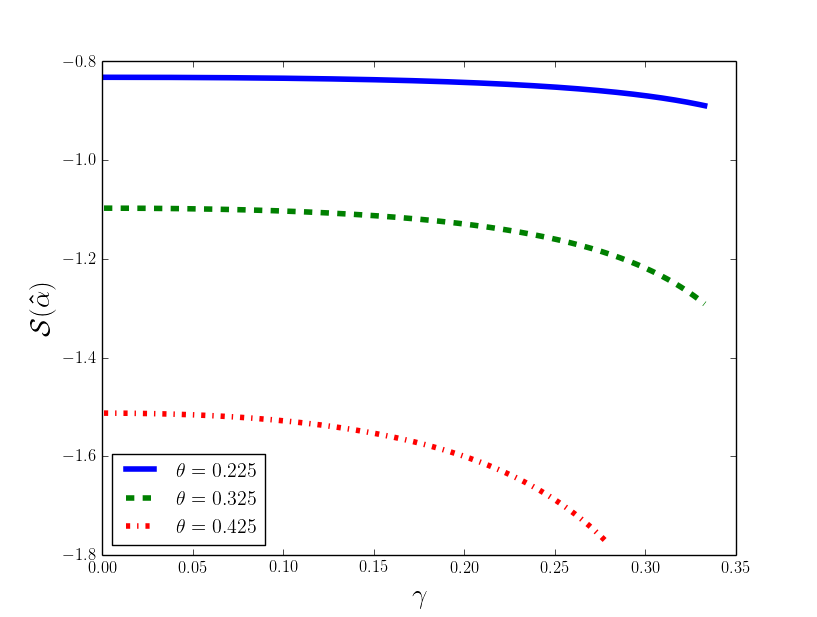}
 	\hspace{-5mm}
 	\includegraphics[width = 0.48\textwidth]{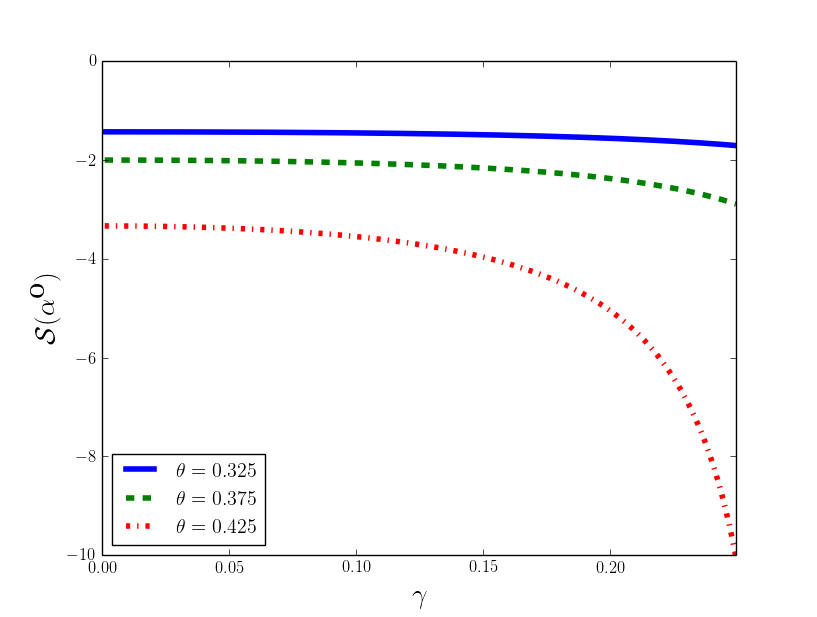}
	\caption{Social cost of Nash equilibrium (left) and social optimum (right) strategy profiles as a function of the power-law exponent $\gamma$ for the normalized power-law graphon.}
	 \label{fig:powerlawPoAnormalized1}
 \end{figure}

 \begin{figure}[h!]
 \centering
 \hspace{-5mm}
 	\includegraphics[width = 0.56\textwidth,height = 0.275\textheight]{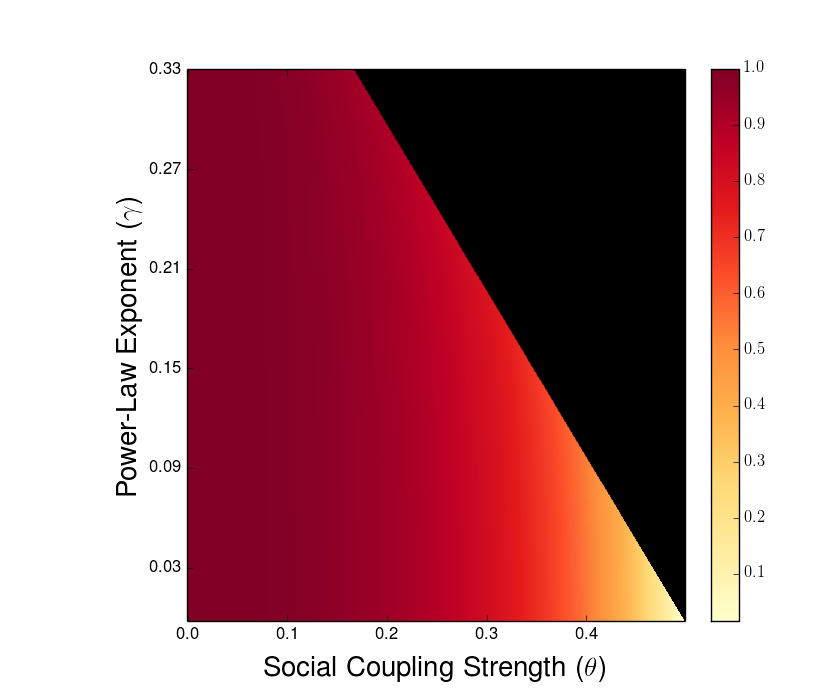}
 	\hspace{-5mm}
 	\includegraphics[width = 0.48\textwidth]{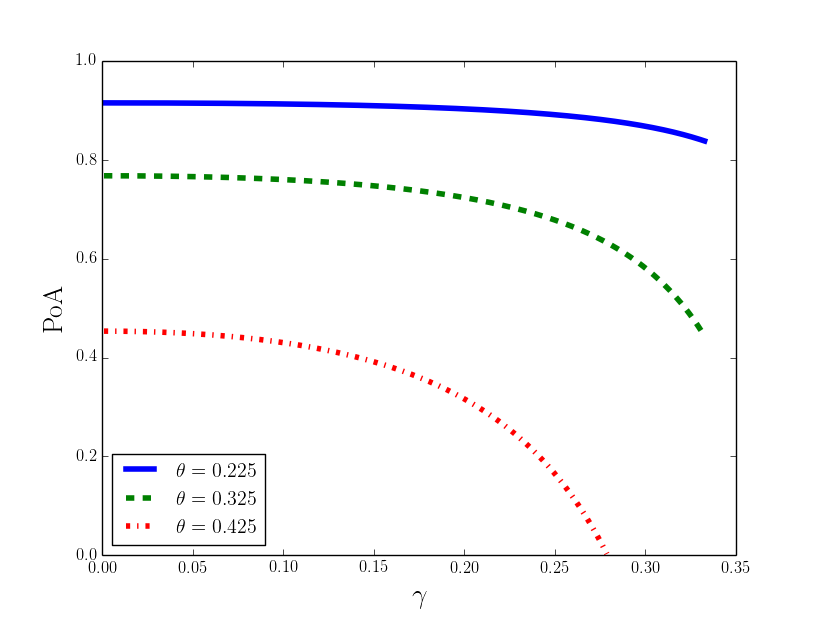}
	\caption{Heatmap of PoA of cities game on normalized power-law graphon for varying values of $\theta$ and $\gamma$. Region in black corresponds to cases in which $1 - 2 \gamma - 2 \theta g(\gamma) < 0$ and the social optimum is not defined. }
	 \label{fig:powerlawPoAnormalized2}
 \end{figure}

Now, we consider the simple threshold graphon, $w(x,y) = 1_{x + y \leq 1}$. We compute that the Nash equilbrium is given by 
\[ \hat{\alpha}_x = \left( \frac{k}{1 - \sin\left(\theta\right)} \right) \left[ \cos\left( \theta \left(1-x\right) \right) - \sin\left( \theta x \right) \right], \]
and the social optimal is given by
\[\balpha^{O}_x = \left( \frac{k}{1 - \sin\left( 2\theta \right)} \right) \left[ \cos\left( 2 \theta \left(1-x\right) \right) - \sin\left(2 \theta x \right) \right]. \]
Notice that for the simple threshold graphon, the index of the player directly relates to how connected they are with the other players, and the above two control profiles have different shapes with respect to the players' index, i.e.\ their level of connectivity.
Finally, we calculate that the PoA on the simple threshold graphon is given by
\[\mathrm{PoA} \: = \left( \frac{2 \theta}{1 - \sin(\theta)} \right) \left(\frac{1 - \sin \left(2 \theta\right)}{\cos(2 \theta) + \sin(2 \theta) - 1} \right)\]
where we can verify that $\lim_{\theta \to 0} \mathrm{PoA} = 1$ and that the PoA is a decreasing function of $\theta$, meaning that stronger benefits from social interactions causes greater inefficiency of the Nash equilibrium strategy profile. The PoA and social costs are shown in Figures \ref{fig:thresholdstrategyprofiles} and \ref{fig:thresholdPoA}. 

 \begin{figure}[h!]
 \centering
 \hspace{-5mm}
 	\includegraphics[width = 0.48\textwidth]{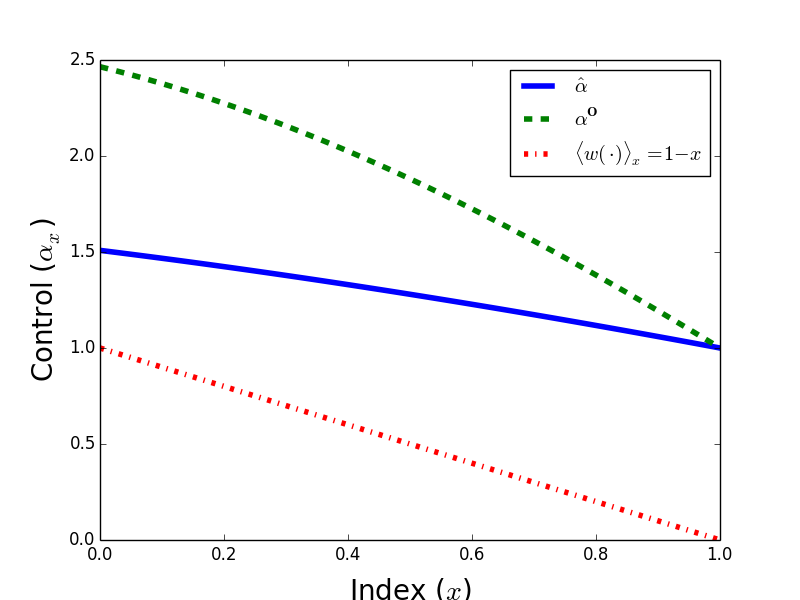}
 	\hspace{-5mm}
 	\includegraphics[width = 0.48\textwidth]{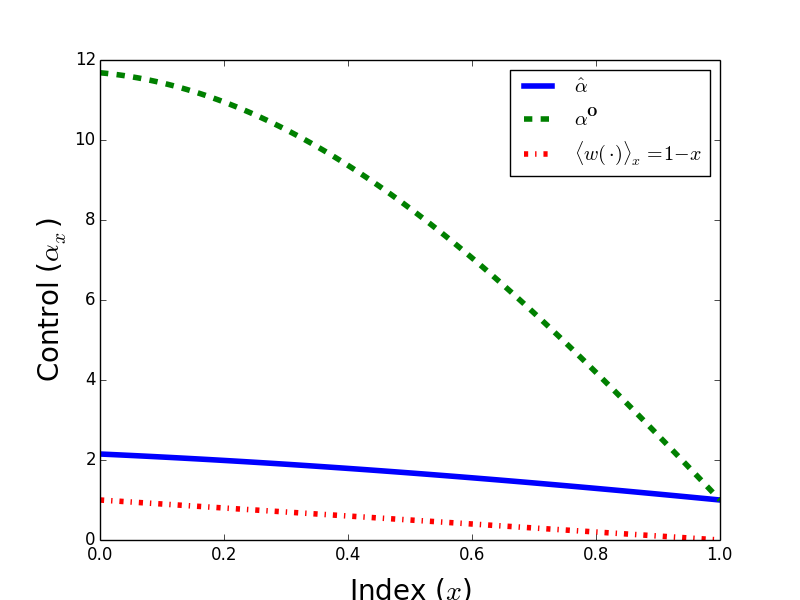}
	\caption{Nash equilibrium strategy profile (blue solid line), socially optimal strategy profile (green dashed line), and connection strength for individual with index $x$ (red dash-dotted line) for $\theta =0.4$ (left) and $\theta = 0.7$ (right).  }
	 \label{fig:thresholdstrategyprofiles}
 \end{figure}

 \begin{figure}[h!]
 \centering
 	\includegraphics[width = 0.48\textwidth]{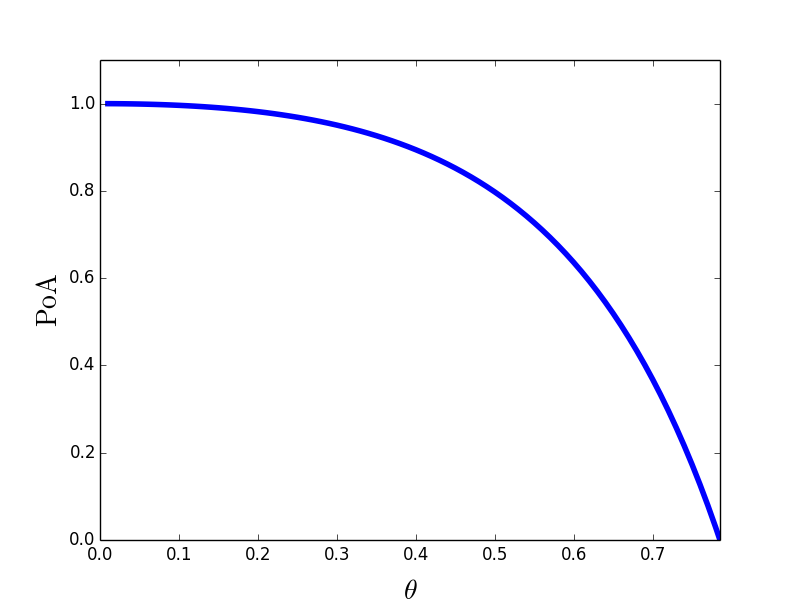}
 
	\caption{Price of anarchy as a function of $\theta$ for the threshold graphon.  }
	 \label{fig:thresholdPoA}
 \end{figure}

\subsection{Cournot Competition}
As before, we directly consider the model with a continuum of players. Here, the weight $w(x,y)$  quantifying the interaction between producers with indices $x$ and $y$ represents either the level of exchangeability of their products, or how much their consumer bases overlap. This model is more involved than the two preceding ones because, if $\balpha$ is the strategy profile giving the sales of all the producers, the price at which producer $x\in I$ will sell $\alpha_x$ units is given by:
\begin{equation*}
    X_{\alpha_x,[\bZ\balpha]_x,\xi_x}=a-b\alpha_x+c[\bZ\balpha]_x+\xi_x,
\end{equation*}
for positive constants $a$, $b$, and $c$, and where $\bZ\balpha$ solves
\begin{equation*}
  [\bZ\balpha]_x=\int_I w(x,y)\Bigl(a-b\alpha_y+c[\bZ\balpha]_y\bigr) dy, \qquad \lambda_I-a.e. x\in I.
\end{equation*}
Accordingly, producer $x$ will incur cost:
\begin{equation*}
    J_x(\alpha,\balpha)=\mathbb{E}\left[\frac{1}{2}\alpha^2-\alpha X_{\alpha,z_x,\xi_x}\right].
\end{equation*}
So in the notation of the paper, 
$$ 
b(\alpha,z)=a-b\alpha+cz, 
\qquad\text{and}\qquad
f(x,\alpha,z)=\frac{1}{2}\alpha^2-\alpha x.
$$
Notice that, in contrast with the previous examples, the function $b$ depends upon the aggregate variable $z$.

\begin{proposition}
    If $c(b+1)\|\bW\|<1$, there is a unique Nash equilibrium $\hat{\balpha}$ for the graphon game specified by the above state equation and cost. It is given by:
    \begin{equation*}
       \hat\balpha=\frac{a}{1+2b}[\bI-c(1+b)\bW]^{-1}\mathbf{1}.
    \end{equation*}
\end{proposition}

\begin{proof}
    Expanding the cost, we have:
    \begin{equation*}
        \mathcal{J}(\alpha,z)=\left(b+\frac{1}{2}\right)\alpha^2-\left(a+cz\right)\alpha,
    \end{equation*}
    which is minimized for given $z$ for:
    $$
    \hat\alpha=\frac{a+cz}{1+2b}.
    $$
    So $\hat\balpha$ will be a Nash equilibrium if for $\lambda_I$ almost every $x\in I$ we have 
    \begin{equation*}
        \hat{\alpha}_x=\frac{c}{1+2b}[\Zb\hat{\balpha}]_x+\frac{a}{1+2b},
    \end{equation*}
    which gives:
    \begin{equation}
\label{fo:Z_hat_alpha}        
[\Zb\hat{\balpha}]_x=\frac{1+2b}{c}\hat{\alpha}_x-\frac{a}{c}.
    \end{equation}
    Using \eqref{eq:connection_W_A}, the definition of the graphon operator and \eqref{fo:Z_hat_alpha}  we have:
    \begin{equation*}
        \frac{1+2b}{c}\hat{\balpha}-\frac{a}{c}\mathbf{1}=\bW \left(a\mathbf{1}-b \hat{\balpha}+c\Bigl(\frac{1+2b}{c}\hat{\balpha}-\frac{a}{c}\mathbf{1}\Bigr)\right)=(1+b)\bW \hat{\balpha}
\end{equation*}
    and we conclude
    $$
    \hat\balpha=\frac{a}{1+2b}[\bI-\frac{c(1+b)}{1+2b}\bW]^{-1}\mathbf{1}
    $$
    because the inverse exists because of our assumptions.
\end{proof}
Again, the equilibrium strategy profile is given by the Katz centrality measure of the graphon, and it can be computed explicitly for the examples of graphon considered above.
 
\section{Conclusion} 
\label{sec:conclusion}
To summarize, we have rigorously defined a class of non-atomic games called graphon games, making use of the limit objects for sequences of dense graphs. We proved existence and uniqueness of Nash equilibria for graphon games and we have made rigorous connections to finite-player network games in two directions: 1) the controls from the finite-player network games converge in some sense to the controls of the graphon game, and 2) the controls from the graphon game can be used to construct an $\epsilon$-Nash equilibrium for the finite-player network games. We have also made rigorous connections to mean field games when the graphon has certain symmetrical properties. Finally, we motivated our study with a few applications for which we showed explicit computations of the Nash equilibria, as well as, at least for one of them, the social optimal strategies computed by a central planner.


\end{document}